\documentclass[oneside,english,12pt,reqno]{amsart}
\usepackage[colorlinks=true,linkcolor=blue,citecolor=red]{hyperref}
\allowdisplaybreaks

\usepackage{units}
\usepackage{mathrsfs}
\usepackage{amstext}
\usepackage{amsthm}
\usepackage{esint}
\usepackage{amsmath}
\usepackage{amssymb}
\usepackage{soul}
\usepackage{cancel}
\newcommand{\R}{\mathbb{R}}
\newcommand{\pa}{\partial}
\newcommand{\eps}{\varepsilon}
\usepackage{a4wide}
\usepackage{enumitem}
\newcommand{\bra}[1]{\left(#1\right)}
\newcommand{\sbra}[1]{\left[#1\right]}
\newcommand{\norm}[1]{\left\|#1\right\|}
\newcommand{\abs}[1]{\left|#1\right|}
\newcommand{\na}{\nabla}
\newcommand{\Lam}{\Lambda}
\newcommand{\mr}{\mathrm{mr}}
\newcommand{\uu}{\mathsf{U}}

\newcommand{\ff}{\mathsf{F}}
\newcommand{\F}{\mathscr{F}}

\newcommand{\wh}{\widehat}

\newcommand{\mM}{\mu_M}
\newcommand{\pO}{p_{\Omega}}
\newcommand{\pM}{p_M}
\newcommand{\intO}{\int_{\Omega}}
\newcommand{\intM}{\int_{M}}
\newcommand{\intQT}{\iint_{Q_T}}
\newcommand{\intMT}{\iint_{M_T}}

\newcommand{\intQtautwo}{\iint_{Q_{\tau,\tau+2}}}

\newcommand{\intMtautwo}{\iint_{M_{\tau,\tau+2}}}
\newcommand{\vat}{\varphi_\tau}

\renewcommand{\L}{\mathscr{L}}
\renewcommand{\H}{\mathscr{H}}

\newcommand{\vk}{\varkappa}

\newcommand{\sumi}{\sum_{i=1}^{m_1}}
\newcommand{\sumj}{\sum_{j=1}^{m_2}}

\newcommand{\LO}[1]{L^{#1}(\Omega)}
\newcommand{\LM}[1]{L^{#1}(M)}
\newcommand{\LQ}[1]{L^{#1}(Q_T)}
\newcommand{\LS}[1]{L^{#1}(M_T)}
\newcommand{\LStau}[1]{L^{#1}(M_{\tau,\tau+1})}
\newcommand{\LStaut}[1]{L^{#1}(M_{\tau,\tau+2})}
\newcommand{\LQtau}[1]{L^{#1}(Q_{\tau,\tau+1})}
\newcommand{\LQtaut}[1]{L^{#1}(Q_{\tau,\tau+2})}

\makeatletter
\numberwithin{equation}{section}
\numberwithin{figure}{section}

\newtheorem{theorem}{Theorem}[section]
\newtheorem{lemma}{Lemma}[section]
\newtheorem{remark}{Remark}[section]
  
\makeatother

\usepackage{babel}

\newtheorem{proposition}{Proposition}[section]

\title[Global existence of volume-surface reaction-diffusion systems]{Global well-posedness for volume-surface reaction-diffusion systems}
\author[J. Morgan]{Jeff Morgan}
\address{Jeff Morgan \hfill\break
	Department of Mathematics, 
	University of Houston, Houston, Texas 77004, USA}
\email{jmorgan@math.uh.edu}
\author[B.Q. Tang]{Bao Quoc Tang}
\address{Bao Quoc Tang \hfill\break
	Institute of Mathematics and Scientific Computing, University of Graz, 
	Heinrichstrasse 36, 8010 Graz, Austria}
\email{quoc.tang@uni-graz.at, baotangquoc@gmail.com}

\begin{document}
\begin{abstract}
	We study the global existence of classical solutions to volume-surface reaction-diffusion systems with control of mass. Such systems appear naturally from modeling evolution of concentrations or densities appearing both in a volume domain and its surface, and therefore have attracted considerable attention. Due to the characteristic volume-surface coupling, global existence of solutions to general systems is a challenging issue. In particular, the duality method, which is powerful in dealing with mass conserved systems in domains, is not applicable on its own. In this paper, we introduce a new family of $L^p$-energy functions and combine them with a suitable duality method for volume-surface systems, to ultimately obtain global existence of classical solutions under a general assumption called the \textit{intermediate sum condition}. For systems that conserve mass, but do not satisfy this condition, global solutions are shown under a quasi-uniform condition, that is, under the assumption that the diffusion coefficients are close to each other. In the case of mass dissipation, we also show that the solution is bounded uniformly in time by studying the system on each time-space cylinder of unit size, and showing that the solution is sup-norm bounded independently of the cylinder. Applications of our results include global existence and boundedness for systems arising from membrane protein clustering or activation of Cdc42 in cell polarization.
\end{abstract}
\maketitle
\tableofcontents

\section{Introduction and Main Results}

\subsection{Problem setting}
Let $\Omega\subset\R^n$, $n\geq 1$, be a bounded domain with smooth boundary $M:= \partial\Omega$. Let $m_1\geq 0$ and $m_2\geq 0$. We consider the following volume-surface reaction-diffusion system
\begin{equation}\left\{
\begin{aligned}
\frac{\partial u_{i}}{\partial t}&=d_{i}\Delta u_{i}+F_{i}(u), &&\text{ on \ensuremath{\Omega\times(0,T)\text{ for \ensuremath{i=1,...,m_1}}}}\\
d_{i}\frac{\partial u_{i}}{\partial\eta}&=G_{i}(u,v), &&\text{ on \ensuremath{M\times(0,T)\text{ for \ensuremath{i=1,...,m_1}}}}\\
\frac{\partial v_{j}}{\partial t}&=\delta_{j}\Delta_M v_{j}+H_{j}(u,v), && \text{ on \ensuremath{M\times(0,T)\text{ for \ensuremath{j=1,...,m_2}}}}\\
u&=u_0 && \text{ on \ensuremath{\overline{\Omega}\times{\{0\}}}}\\
v&=v_0 && \text{ on \ensuremath{M\times{\{0\}}}}
\end{aligned}\right.\label{eq:mainsys}
\end{equation}
where $\pa/\pa{\eta}$ denotes the outward normal flux on the boundary $M$, $u = (u_1,u_2,\ldots, u_{m_1})$, $v = (v_1,v_2, \ldots, v_{m_2})$ are vectors of concentrations, $F_{i}:\R^{m_1}\to \R^{m_1}$ and $G_{i},H_{j}:\R^{m_1+m_2}\to \R^{m_1+m_2}$ satisfy assumptions which will be specified later. The operators $\Delta$ and $\Delta_M$ represent the traditional Laplacian in $\Omega$ and Laplace-Beltrami operator on $M$. The initial
data $u_0=(u_{0,{i}})_{i=1,\ldots, m_1}$ and $v_0=(v_{0,{j}})_{j=1,\ldots, m_2}$ are assumed to be smooth, bounded and component-wise non-negative on $\overline{\Omega}$ and $M$, respectively. Also, when $n=1$, we assume $m_2 = 0$. Throughout this paper, we will assume the following conditions on the domain
\begin{enumerate}[label=(O),ref=O]
	\item\label{O} $\Omega\subset\mathbb R^n$, $n\geq 1$, is bounded domain with smooth boundary $M = \pa\Omega$ such that $\Omega$ lies locally on one side of $M$;
\end{enumerate}
on the diffusion coefficients
\begin{enumerate}[label=(D),ref=D]
	\item\label{D} $d_i > 0$ and $\delta_j > 0$ for all $i=1,\ldots,\; m_1, j=1,\ldots, m_2$;
\end{enumerate}
and on the nonlinearities
\begin{enumerate}[label=(F\theenumi),ref=F\theenumi]
	\item\label{F0} (local Lipschitz continuity) $F_i: \R_{+}^{m_1} \to \R$, $G_i: \R_+^{m_1+m_2}\to \R_+$, $H_j: \R_+^{m_1+m_2}\to \R$ are locally Lipschitz continuous for $i\in \{1,\ldots, m_1\}$ and $j\in \{1, \ldots, m_2\}$;
	\item\label{F1} (quasi-positivity) $$F_{i}(u),G_{i}(u,v)\ge0\text{ for all }u\in \R_{+}^{m_1}, v\in \R_{+}^{m_2}\text{ with }u_{i}=0, \; \forall i=1,\ldots, m_1,$$
	and
	\begin{equation*}
	H_j(u,v) \geq 0\text{ for all } u\in \R_+^{m_1}, v\in \R_+^{m_2} \text{ with } v_{j} = 0, \; \forall j=1,\ldots, m_2.
	\end{equation*}
	\item\label{F2} (mass control) there exist $a_i,b_j>0$ and $L\in \R$, $K\geq 0$ such that 
	\begin{equation}\label{eq:balsysmain}
	\begin{aligned}\sum_{i=1}^{m_1}a_i F_i(u)&\le L\left(\sum_{i=1}^{m_1}u_i\right) + K,\\ \sum_{i=1}^{m_1}a_i G_i(u,v)+\sum_{j=1}^{m_2}b_j H_j(u,v)&\le L\left(\sumi u_i + \sumj v_j\right) + K,
	\end{aligned}
	\end{equation}
	for all $u\in \R_+^{m_1}$ and $v\in \R_+^{m_2}$;
	\item\label{F3} (polynomial growth) there are constants $K_1, r > 0$ such that, for all $i=1,\ldots, m_1$, $j=1,\ldots, m_2$, 
	\begin{equation}
	F_i(u),G_i(u,v),H_j(u,v)\le K_1\left(\sum_{i=1}^{m_1}u_{i}^{r}+\sum_{j=1}^{m_2}v_{j}^{r}+1\right),\label{eq:polybound}
	\end{equation}
	for all $u\in \R_{+}^{m_1}$ and $v\in \R_{+}^{m_2}$.
\end{enumerate}
With \eqref{F0}, it can be shown that \eqref{eq:mainsys} possesses a unique local strong solution on a maximal interval $(0,T_{\max})$, cf. \cite{sharma2016global}. Moreover,
\begin{equation}\label{blowup-criterion}
	T_{\max}<+\infty \quad \Longrightarrow \quad  \limsup_{t\nearrow T_{\max}}\bra{\sumi \|u_i(t)\|_{\LO{\infty}} +\sumj \|v_j(t)\|_{\LM{\infty}}} = +\infty.
\end{equation}
The quasi-positivity assumption \eqref{F1} assures that the solution to \eqref{eq:mainsys} is (component-wise) non-negative (as long as it exists) if the initial data is non-negative. This assumption also has a simple physical interpretation: if a concentration is zero, it cannot be consumed in a reaction. The condition \eqref{F2} implies that the total mass of the system is finite for all time, which, together with the non-negativity of solutions, implies that
\begin{equation*}
	u_i \in L^\infty(0,T;\LO{1}),\; v_j\in L^\infty(0,T;\LM{1}) \quad \forall i=1,\ldots, m_1, \; j=1,\ldots, m_2, \; T>0.
\end{equation*}
Comparing to \eqref{blowup-criterion}, these a-priori estimates are far from enough to conclude the global existence of solutions to \eqref{eq:mainsys}. In fact, it was shown in \cite{pierre2000blowup} that the above assumptions are very likely not enough to prevent blow-up in finite time even for reaction-diffusion systems in domains. The global well-posedness question for \eqref{eq:mainsys} is, therefore, challenging and generally open. 

\medskip
In this paper, we will show that under the above natural assumptions, and a so-called {\it intermediate sum condition}, there exists a unique global strong solution to \eqref{eq:mainsys}. Moreover, if the total mass is bounded by a constant independent of time, the solution is sup-norm bounded uniformly for all time.
\subsection{State of the art}
Volume-surface (or bulk-surface) systems of the form \eqref{eq:mainsys} have recently attracted a lot of attention. On the one hand, such systems arise naturally from many applications. For example, in asymmetric stem cell division \cite{fellner2016quasi,fellner2018well}, in modeling receptor-ligand dynamics in cell biology \cite{garcia2014mathematical,alphonse2018coupled}, in crystal growth \cite{kwon2001modeling,yeckel2005computer}, in population modeling \cite{berestycki2013influence,berestycki2015effect}, in chemistry with fast sorption \cite{souvcek2019continuum,augner2020analysis}, or in fluid mechanics \cite{mielke2013thermomechanical,glitzky2013gradient}, and so on. 

On the other hand, the volume-surface coupling yields new and highly non-trivial challenges in the analysis of such systems. Of interest and importance is the question of global existence of solutions, and most of the existing works rely heavily on special structures of the considered systems. For instance, when the system is linear, or the reactions have at most linear growth, global (weak, strong) solutions were shown in \cite{fellner2016quasi,hausberg2018well}. For nonlinear coupling on the boundary, \cite{fellner2018well} shows global bounded solutions for the reversible reaction $\alpha \mathcal U \leftrightharpoons \beta \mathcal V$, where $\mathcal U$ and $\mathcal V$ are volumic- and surface-concentrations, respectively. For general systems, \cite{disser2020global} showed that if the $L^\infty$-bound is preserved, then one can get global classical solutions. The works \cite{sharma2016global,sharma2017uniform} showed global existence, and boundedness of solutions to volume-surface systems assuming a linear upper bound for $F_i$, $G_i$, and the sum of $G_i + H_j$, i.e. for any $i\in \{1,\ldots, m_1\}$ and any $j\in \{1,\ldots, m_2\}$,
\begin{equation}\label{linear_growth}
	F_i(u) \leq L\bra{1+\sum_{k=1}^{m_1} u_k}, \quad G_i(u,v) + H_j(u,v) \leq L\bra{1+\sum_{k=1}^{m_1}u_k + \sum_{l=1}^{m_2}v_l}
\end{equation}
for some $L>0$. Due to these restrictions, the results therein are not applicable in many systems arising from cell biology (see e.g. \cite{stolerman2019stability,borgqvist2020cell}).

It is also remarked that significant progress has been made lately concerning mass controlled reaction-diffusion systems {\it in domains}, i.e. when $m_2 = 0$ and $G\equiv 0$. For instance, a major problem concerning the global existence of strong solutions for systems with quadratic nonlinearities in all dimensions has been settled in three recent works \cite{souplet2018global,caputo2019solutions,fellner2020global}. If the nonlinearities have $L^1(\Omega\times(0,T))$-bound a priori, global weak solutions can be shown \cite{pierre2010global}. In particular, by assuming only an entropy inequality, \cite{fischer2015global} shows global existence of renormalized solutions without any restriction on the growth of nonlinearities. We emphasize that none of these works seem to readily extend to the case of volume-surface systems. 

\medskip
In this paper, by introducing a new family of $L^p$-energy function and combining it with duality methods for volume-surface systems, we show the global existence of solutions to a large class of systems of type \eqref{eq:mainsys}, which include the systems in \cite{stolerman2019stability,borgqvist2020cell} as special cases.

\subsection{Main results and Key ideas}
To study the global existence of solutions to \eqref{eq:mainsys}, we assume the so-called {\it intermediate sum condition} (see e.g. \cite{morgan2020boundedness}) for the nonlinearities of \eqref{eq:mainsys}, i.e. there exists an $(m_1+m_2)\times (m_1+m_2)$ lower
triangular matrix $A$ with non-negative elements everywhere and positive entries on the main diagonal,
and constants $L_2\ge0$, $\mM >0$, such that 
\begin{equation}\label{eq:intsum1}
A\begin{bmatrix}F(u)\\ \vec{0}_{m_2}\end{bmatrix}
\le L_2\bra{\sumi u_i^{\pO}+1}\vec{1}_{m_1+m_2},
\end{equation}
\begin{equation}\label{eq:intsum2}
A\begin{bmatrix}G(u,v)\\H(u,v)\end{bmatrix}\le L_2\bra{\sumi u_i^{\pM} + \sumj v_j^{\pM}+1}\begin{bmatrix}\vec{1}_{m_1}\\ \vec{0}_{m_2}\end{bmatrix} + L_2\bra{\sumi u_i^{\mM} + \sumj v_j^{\mM}+1}\begin{bmatrix}\vec{0}_{m_1}\\ \vec{1}_{m_2}\end{bmatrix},
\end{equation}
for all $u\in \R_{+}^{m_1}$ and $v\in \R_{+}^{m_2}$, where the exponent $\pO, \pM$ and $\mM$ will be specified later,
$F(u) = (F_1(u), \ldots, F_{m_1}(u))$ and $G(u,v) = (G_1(u,v), \ldots, G_{m_2}(u,v))$. We denote by $\vec{1}_m$ (rsp. $\vec{0}_m$) the vector of all $1s$ elements (rsp. $0s$ elements) in $\mathbb R^m$. By dividing each row of $A$ by the diagonal element, we can assume w.l.o.g. that $a_{ii} = 1$ for all $i=1,\ldots, m_1+m_2$, {\it and we will do so for the rest of this paper}. 

\begin{remark}
It's emphasized that assumptions \eqref{eq:intsum1} and \eqref{eq:intsum2} do not require the components of the reaction vector fields $F$, $G$ and $H$ to be at most of order $\pO$, $\pM$ or $\mM$. It simply requires a "trade-off" of higher order terms between components of $F$, $G$ and $H$, and additionally requires that $F_1(u)$ be bounded above by a polynomial in $u$ of order $p_\Omega$ and $G_1(u,v)$ be bounded above by a polynomial in $u$ and $v$ of order $\pM$.
\end{remark}

\medskip
The first main result of this paper is the following theorem.
\begin{theorem}[\eqref{O}-\eqref{D}-\eqref{F0}-\eqref{F1}-\eqref{F2}-\eqref{F3}-\eqref{eq:intsum1}-\eqref{eq:intsum2}]\label{main:thm0}
	Assume
	\begin{equation}\label{pOpM}
	1\leq \pO < 1+\frac{2}{n}, \quad \text{ and } \quad 1\leq \pM < 1+\frac{1}{n},
	\end{equation}
	and
	\begin{equation}\label{mM}
		1\leq \mM \leq 1+\frac{4}{n+1}.
	\end{equation}
	Then for any nonnegative, bounded initial data $(u_0,v_0)\in (W^{2-2/p}(\Omega))^{m_1}\times (W^{2-2/p}(M))^{m_2}$ for some $p>n$ satisfying the compatibility condition
	\begin{equation}\label{compatibility}
		d_i\pa_{\eta}u_{i,0} = G_i(u_0,v_0) \quad \text{on M} \quad \text{ for all } i=1,\ldots, m_1,
	\end{equation}
	the system \eqref{eq:mainsys} has a unique global classical solution. Moreover, if $L < 0$ or $L=K=0$ in \eqref{F2}, then 
	\begin{equation*}
	\sup_{i=1,\ldots, m_1}\sup_{j=1,\ldots, m_2}\sup_{t\geq 0}\bra{\|u_i(t)\|_{L^\infty(\Omega)} + \|v_j(t)\|_{L^\infty(M)}} <+\infty.	
	\end{equation*}
\end{theorem}
\begin{remark}[Improvements or generalizations]\label{remark0}\hfill\
	\begin{itemize}
		\item The first remark is that the smoothness of initial data as well as the compatibility condition \eqref{compatibility} are only used 
		to obtain a local classical solution (see e.g. \cite{sharma2016global} or \cite{hausberg2018well}).  We believe that this local existence can in fact be proved for merely bounded initial data, though, due to the volume-surface coupling, the proof is more delicate comparing to the case of systems in domains (see e.g. \cite{alikakos1981regularity}). The details are left for the interested reader.
		\item The intermediate sum conditions involving $F$ and $G$ in \eqref{eq:intsum1} and \eqref{eq:intsum2} are used to construct $L^p$-energy functions, which is essential in the proof of global existence. In fact, one can construct such functions under more general (but technical) assumptions (see Remark \ref{remark2}). We choose to present Theorem \ref{main:thm0} under the intermediate sum conditions \eqref{eq:intsum1} and \eqref{eq:intsum2} as they are more constructive, and appear naturally in many applications (see Section \ref{sec:application}).
		\item The condition of $\mM$ in \eqref{mM} can be slightly improved as
		\begin{equation*}
		\mM < 1 + \frac{2(2+\eps)}{n+1}
		\end{equation*}
		for a sufficiently small $\eps>0$.
		\item The assumption $L<0$ or $L=K=0$ is used obtain uniform-in-time $L^1$-bounds (see Lemma \ref{L1bound_cylinder}), i.e.
		\begin{equation*}
			\sup_{t\geq 0}\bra{\|u_i(t)\|_{\LO{1}} + \|u_i\|_{L^1(M\times{(t,t+1)})} + \|v_j(t)\|_{\LM{1}}} < +\infty,
		\end{equation*}
		which eventually leads to the uniform-in-time $L^\infty$-bound. If this $L^1$-bound can be obtain differently, for instance, by using special structures of a specific model, the assumption $L <0 $ or $L=K=0$ can be removed.
		\item If $L<0$ and $K=0$ in \eqref{F2}, we can show that the global solution decays exponentially to zero as $t\to \infty$. The details are left for the interested reader.
		\item The sub-critical order of $\pO$ (and $\pM$) in \eqref{pOpM} is typical for reaction-diffusion systems (in domains) once an $L^1$-bound is known (see e.g. \cite{alikakos1979lp}). It is again emphasized that we do not assume the nonlinearities, but only an intermediate sum of them, to have the growth of order $\pO$.
		\item It's finally remarked that the results of Theorem \ref{main:thm0} are new even in the case without surface concentrations, i.e. $m_2 = 0$.
	\end{itemize}
\end{remark}

Let us recall that for reaction-diffusion systems \textit{in domains}, i.e. $m_2=0$ and homogeneous boundary conditions $G_i \equiv 0$, $\forall i=1,\ldots, m_1$, the global existence with control of mass and intermediate sum conditions has been successfully relied on the famous duality method, cf. \cite{pierre2010global,morgan2020boundedness}. Thanks to this method, the control of mass condition implies an $L^{2+\eps}$-estimate a-priori. This, together with the intermediate sum condition, allows the use of a bootstrap procedure to ultimately obtain $L^{\infty}$-estimate, which is sufficient for the global existence. It is remarked that the duality is not only crucial to get the initial $L^{2+\eps}$-bound, but also important in the bootstrapping procedure. Such a strategy, unfortunately, fails to apply to {\it volume-surface systems} of the form \eqref{eq:mainsys}, see \cite{morgan2019martin}. In this paper, our {\it first key idea} to overcome this difficulty is to introduce a new family of $L^p$-energy functions. Some preliminary ideas of such $L^p$-energy functions have been carried out also for reaction-diffusion systems, but have been less noticed comparing to the duality method\cite{malham1998global,kouachi2001existence,abdelmalek2007proof}\footnote{Aside from this method's technicality, another reason for this ignorance might be that the elegant duality method, cf. \cite{pierre2000blowup}, has proved to be very efficient in studying reaction-diffusion systems with control of mass, see also the survey \cite{pierre2010global}. A recent work \cite{sharma2020global} studied reaction-diffusion systems with nonlinear boundary conditions to which this $L^p$-energy method is well adapted while the duality method seems not. It is also remarked that the assumption therein do not allow for the use of intermediate sums, and restrict to primary application of the results to two component systems.}. The main aim is to provide a generalization of the usual method of constructing $L^p$-energy functions by (traditionally) multiplying both sides of a parabolic equation for a function $w(x,t)$ by $w(x,t)^{p-1}$. Instead, we consider $L^p$-energy functions consisting of all (mixed) multi-variable polynomials of order $p$ with carefully chosen coefficients. Certainly, the essential difficulty is to choose an $L^p$-energy function so that it is ``compatible" with both diffusion and reaction parts of the system, i.e. its evolution in time should lead to useful a-priori estimates. We show in this paper that the intermediate sum conditions \eqref{eq:intsum1} and \eqref{eq:intsum2} allow us to find such an energy function. Yet, this strategy alone is not enough to obtain sufficient estimates for global existence due to the surface concentrations. Our {\it second key idea} is, therefore, to combine the constructed $L^p$-energy functions with a duality method for volume-surface reaction-diffusion systems. This method has been used in previous work, see e.g. \cite{sharma2016global,morgan2019martin}, but due to the lack of $L^p$-estimates obtained in the current work, it has only applied to systems with restrictive conditions, for instance under assumptions \eqref{linear_growth}. We stress, as it will become apparent in our paper, that only the combination of these two ideas makes it possible to obtain global existence of \eqref{eq:mainsys} under the general assumptions \eqref{eq:intsum1} and \eqref{eq:intsum2}.

\medskip
Let us briefly sketch the main ideas of the proof of Theorem \ref{main:thm0}, which can be roughly divided in several steps.
\begin{itemize}
	\item \textbf{Step 1}. Firstly, by \eqref{F2} one has
	\begin{equation*}
		\|u_i\|_{L^\infty(0,T;\LO{1})}, \|u_i\|_{L^1(M\times(0,T))}, \|v_j\|_{L^\infty(0,T;\LM{1})} \leq C_T
	\end{equation*}
	where $C_T$ is a constant depending on the time horizon $T>0$.
	
	\medskip
\item \textbf{Step 2}. The intermediate sum condition \eqref{eq:intsum1} allows us to construct for any $2\leq p\in \mathbb N$ an $L^p$-energy function of the form
\begin{equation*}
\L_p[u](t) = \intO \sum_{|\beta|=p}\begin{pmatrix}
p\\ \beta 	\end{pmatrix}\theta^{\beta^2}u^{\beta}(x,t)
\end{equation*}
with the convention $u^{\beta} = \prod\limits_{i=1}^{m_1}u_i^{\beta_i}$, $\theta^{\beta^2} = \prod\limits_{i=1}^{m_1}\theta_i^{\beta_i^2}$, and $\begin{pmatrix}p\\\beta\end{pmatrix} = \dfrac{p!}{\beta_{1}!\beta_2!\cdots\beta_{m_1}!}$, where $\theta\in (0,\infty)^{m_1}$ are chosen suitably\footnote{It is remarked that such a function $\L_p[u]$ can be constructed under a more general but less constructive condition than \eqref{eq:intsum1} (see Lemma \ref{lem6} and Remark \ref{remark2}).}. Since all $u_i$ are non-negative, $\bra{\L_p[u]}^{1/p}$ is an equivalent norm to the usual $L^p$-norm of $u$. 
Due to the volume-surface coupling, it does not seem possible to show that $\L_p[u]$ is non-increasing in time, or even bounded. Instead, one obtains for any $\eps>0$ a constant $C_\eps>0$ such that 
\begin{equation}\label{intro1}
\quad \bra{\L_p[u]}' + C\sumi\bra{\intO u_i^{p-1+\pO} + \intM u_i^{p-1+\pM}} \leq C_{\eps} + \eps\sumj\intM v_j^{p-1+\pM}.
\end{equation}
This estimate is {\it crucial} in the analysis of this paper. The ``left over" $L^p$-integrals of surface concentrations $v_j$ in \eqref{intro1} are treated using a duality method in the next steps.

\medskip
\item {\bf Step 3}. We derive an improved duality estimate using the dual problem suited for volume-surface systems
\begin{equation*}
\begin{cases}
\pa_t \phi + \Delta \phi = 0, & (x,t)\in \Omega\times(0,T),\\
\pa_t \phi + \delta \Delta_M \phi = -\psi, &(x,t)\in M\times (0,T),\\
\phi(x,T) = 0, &x\in\overline{\Omega} 
\end{cases}
\end{equation*}
and \eqref{intro1} to obtain
\begin{equation}\label{twoplus}
\sumi\bra{\|u_i\|_{L^{2+\eps}(\Omega\times(0,T))}+\|u_i\|_{L^{2+\eps}(M\times(0,T))}} + \sumj\|v_j\|_{L^{2+\eps}(M\times(0,T))} \leq C_T
\end{equation}
for some $\eps>0$. 

\medskip
\item {\bf Step 4.} By using the estimates in \eqref{twoplus} and the duality method, we are able to show: for any $p>1$, any $k\in \{1,\ldots, m_2\}$, and any $\eps>0$, there exists $C_{T,\eps}>0$ such that
\begin{equation*}
\|v_k\|_{L^p(M\times(0,T))} \leq C_{T,\eps} + C_T\sum_{j=1}^{k-1}\|v_j\|_{L^p(M\times(0,T))} + \eps\sumj \|v_j\|_{L^p(M\times(0,T))}.
\end{equation*}
This ultimately leads to the boundedness of $v_j$ in $L^p(M\times(0,T))$, and consequently, thanks to \eqref{intro1}, the boundedness of $u_i$ in $L^p(\Omega\times(0,T))$ as well as in $L^p(M\times(0,T))$.

From $u_i\in L^p(\Omega\times(0,T))$, $u_i\in L^p(M\times(0,T))$, and $v_j\in L^p(M\times(0,T))$ for all $p>1$, by using the polynomial growth of the nonlinearities in \eqref{F3} and regularization of the heat operator with inhomogeneous boundary conditions, we obtain finally $u_i\in L^\infty(\Omega\times(0,T))$ and $v_j\in L^\infty(M\times(0,T))$, which concludes the global existence of bounded solutions.

\medskip
\item {\bf Step 5}. To obtain uniform-in-time bounds of the global solution, we use for each $\tau\in \mathbb N$, a smooth cut-off function $\vat: \R \to [0,1]$ with $\vat|_{(-\infty,\tau]} = 0$ and $\vat|_{[\tau,\infty)} = 1$ to study \eqref{eq:mainsys} on the cylinder $\Omega\times(\tau,\tau+2)$. By repeating the previous steps on this cylinder, we obtain that $u_i$ and $v_j$ are bounded in $L^\infty(\Omega\times(\tau,\tau+2))$ and $L^\infty(M\times(\tau,\tau+2))$, respectively, {\it uniformly in $\tau\in \mathbb N$}. This concludes the uniform-in-time bounds of the global solution, and consequently completes the proof of Theorem \ref{main:thm0}.
\end{itemize}

\medskip
It is noticed from the proof of Theorem \ref{main:thm0} that the $L^p$-energy method is used for the volume concentrations, while the duality method is well adapted for the surface concentration. An improved duality method, see e.g. \cite{canizo2014improved}, shows that one can deal with higher order nonlinearities in the intermediate sums by assuming {\it quasi-uniform diffusion coefficients}, i.e. when the diffusion coefficients are not too different from each other. Following this idea, the second main result of this paper shows global existence of strong solutions to \eqref{eq:mainsys} with large order $\mM$ of the nonlinearities on the boundary, assuming that the diffusion coefficients $\delta_j$ are quasi-uniform. To state our second main theorem, we denote by
\begin{equation}\label{dmaxmin}
\delta_{\max} = \max\{\delta_1,\ldots, \delta_{m_2}\}, \quad \delta_{\min} = \min\{\delta_1,\ldots, \delta_{m_2}\}.
\end{equation}
\begin{theorem}[\eqref{O}-\eqref{D}-\eqref{F0}-\eqref{F1}-\eqref{F2}-\eqref{F3}-\eqref{eq:intsum1}-\eqref{eq:intsum2}-\eqref{pOpM}]\label{thm:main1}
	Let $\mM > 0$ be fixed. Assume that there exists a constant $\Lam>0$ such that
	\begin{equation}\label{mM_general}
		\Lam\geq \frac{(\mM-1)(n+1)}{2}, \quad \text{ or equivalently } \quad \mM \leq 1+\frac{2\Lam}{n+1}
	\end{equation}
	and
	\begin{equation}\label{quasi-uniform}
	\frac{\delta_{\max}-\delta_{\min}}{\delta_{\max}+\delta_{\min}}C_{\mr,\Lambda'}^M < 1,
	\end{equation}
	where $\Lambda' = \Lambda/(\Lambda - 1)$, and $C_{\mr,\Lambda'}^M$, which depends only on $\Lam$ and $M$, is the maximal regularity constant in Lemma \ref{mr_constant}. Then for any nonnegative, initial data $(u_0,v_0)\in (W^{2-2/p}(\Omega))^{m_1}\times (W^{2-2/p}(M))^{m_2}$ for some $p>n$ satisfying the compatibility condition \eqref{compatibility}, the system \eqref{eq:mainsys} has a unique nonnegative, global strong solution. Moreover, if $L<0$ or $L=K=0$ in \eqref{F2}, the solution is bounded uniformly in time, i.e.
	\begin{equation*}
	\sup_{i=1,\ldots, m_1}\sup_{j=1,\ldots, m_2}\sup_{t\geq 0}\bra{\|u_i(t)\|_{L^\infty(\Omega)} + \|v_j(t)\|_{L^\infty(M)}} <+\infty.
	\end{equation*}
\end{theorem}
\begin{remark}\label{remark00}\hfill\
	\begin{itemize}
		\item We will show in this paper that for any fixed $\delta_{\max}$ and $\delta_{\min}$, condition \eqref{quasi-uniform} is always satisfied with $\Lam = 2$. This implies that Theorem \ref{main:thm0} is in fact a consequence of Theorem \ref{thm:main1}.
		\item We see that for fixed $\mM>0$, condition \eqref{quasi-uniform} is fulfilled if we fix $\Lam = (\mM-1)(n+1)/2$ and require the surface diffusion coefficients $\delta_j$ to be close enough to each other (relatively to their sums). Since $C_{\mr,q}^M$ is increasing as $p\searrow 1$ and $\lim_{q\searrow 1}C_{\mr,q}^M = +\infty$, \eqref{quasi-uniform} means that $\delta_j$ are required to get closer to each other as the nonlinearity order $\mM$ in the intermediate sum increases.
		\item Similarly to Remark \ref{remark0}, we believe that the results of Theorem \ref{thm:main1} still hold true for non-negative, and bounded initial data $(u_0,v_0)\in L^\infty(\Omega)^{m_1}\times L^\infty(M)^{m_2}$.
	\end{itemize}
\end{remark}
Thanks to the quasi-uniform condition \eqref{quasi-uniform}, one can use the duality method to obtain some $L^{\Lam}$-estimates on $u_i$ and $v_j$, which should form the starting point of a bootstrap argument. Unfortunately, $L^\Lam$-estimates just fall short in case $\Lam$ satisfies \eqref{mM_general} with an equality. To overcome this difficulty, an important observation is that the condition \eqref{quasi-uniform} is ``open" in the sense that if \eqref{quasi-uniform} is true for some $\Lam$, then it also holds for $\Lam+\eps$ with $\eps>0$ small enough\footnote{This was observed and utilized in many recent works, see e.g. \cite{canizo2014improved,pierre2017global}.}. This observation allows us to use an improved duality argument to prove that $u_i$ are bounded in $L^{\Lam+\eps}(\Omega\times(0,T))$ and $L^{\Lam+\eps}(M\times(0,T))$, and $v_j$ are bounded in $L^{\Lam+\eps}(M\times(0,T))$, for some small $\eps>0$. Starting from these estimates, we can use the duality method as in {\bf Step 4} above to get $L^p$-estimates for the solution for all $p\geq 1$. This is enough to conclude that the solution is global. To show the uniform-in-time bound, we again use the smooth cut-off function $\vat$ and repeat the arguments for global existence on each cylinder $\Omega\times(\tau,\tau+2)$, to obtain $L^\infty$-bound which are independent of $\tau\in \mathbb N$. 

\medskip
One notices in Theorem \ref{thm:main1} that by imposing the quasi-uniform condition on diffusion coefficients we are able to improve only the order of the nonlinearities in intermediate sums for surface concentrations. The reason is that with the $L^1$-estimates implied from \eqref{F2}, the upper bound $1+\frac 2n$ of $\pO$ seems to be the critical exponent to obtain $L^p$-energy estimates for all $p\geq 1$. For a specific system, it might be possible to obtain better a-priori estimates (see Section \ref{Exp3}), which consequently allows a larger range of $\pO$, $\pM$ and $\mM$. More precisely, we have the following conditional result.
\begin{theorem}[\eqref{O}-\eqref{D}-\eqref{F0}-\eqref{F1}-\eqref{F3}-\eqref{eq:intsum1}-\eqref{eq:intsum2}]\label{main:thm2}
	Assume there exist $a, b\ge 1$ such that, for any $T>0$, and for all $i=1,\ldots, m_1$, $j=1,\ldots, m_2$
	\begin{equation}\label{gg8}
	\|u_i\|_{L^\infty(0,T;\LO{a})} + \|u_i\|_{L^b(M\times(0,T))} + \|v_j\|_{L^{b}(M\times(0,T))} \leq \F(T)
	\end{equation}
	where $\F\in C([0,\infty))$. Assume that
	\begin{equation}\label{pOpM_ab}
		1\leq \pO < 1 + a\cdot\min\left\{\frac{2}{n}; \frac{3}{n+2}\right\}, \quad \text{and} \quad 1\leq \pM < 1 + \frac{a}{n}
	\end{equation}
	and
	\begin{equation}\label{mM_ab}
		1\leq \mM < 1 + \frac{2b}{n+1}.
	\end{equation}
		Then for any nonnegative initial data $(u_0,v_0)\in (W^{2-2/p}(\Omega))^{m_1}\times (W^{2-2/p}(M))^{m_2}$ for some $p>n$ satisfying the compatibility condition
	\begin{equation*}
	d_i\pa_{\eta}u_{i,0} = G_i(u_0,v_0) \quad \text{on M} \quad \text{ for all } i=1,\ldots, m_1,
	\end{equation*}
	the system \eqref{eq:mainsys} has a unique global classical solution. Moreover, if $\sup_{t\geq 0}\F(t) <+\infty$, then 
	\begin{equation*}
	\sup_{i=1,\ldots, m_1}\sup_{j=1,\ldots, m_2}\sup_{t\geq 0}\bra{\|u_i(t)\|_{L^\infty(\Omega)} + \|v_j(t)\|_{L^\infty(M)}} <+\infty.	
	\end{equation*}
\end{theorem}
\begin{remark}
	It is remarked that Theorem \ref{main:thm2} does not impose the mass control assumption \eqref{F2}, since the condition is only used to obtain a-priori estimates of type \eqref{gg8}, which are now given.
\end{remark}

\medskip
{\bf Organization of the paper.} In the next section, we present the construction of $L^p$-energy functions, and show its relation to the intermediate sum condition \eqref{eq:intsum1}. Section \ref{sec:3} is devoted to an improved duality method for volume-surface systems, where we show that assumption \eqref{quasi-uniform}, in combination with the previously constructed $L^p$-energy functions, gives $L^{\Lam+\eps}$-estimates of the solutions. In Section \ref{sec:4}, we start with the proof of Theorem \ref{thm:main1} in subsections \ref{sec:4.1} and \ref{sec:uniform-in-time}, where the first subsection shows the global existence while the second one proves the uniform-in-time bound of the solutions. As pointed out in Remark \ref{remark00}, we prove Theorem \ref{main:thm0} in subsection \ref{sec:4.3} by showing that \eqref{quasi-uniform} is always satisfied for $\Lam=2$. The last subsection \ref{sec:4.4} presents the proof of Theorem \ref{main:thm2}. The last section is devoted to applications of our results to some recent models arising from cell biology. It is noted that previous works are unlikely to be applicable to these systems. Finally, we give in the Appendix \ref{appendix} two technical lemmas concerning the
construction of $L^p$-energy functions.

\medskip
{\bf Notation}. For the rest of this paper, we will use the following notation:
\begin{itemize}
	\item As some of our intermediate lemmas are of independent interest, we use the convention 
	
	\noindent``{\bf Theorem X.} ((A)-(B)-(C))" 
	
	\noindent to indicate that this theorem assumes only conditions (A), (B), (C) (besides the assumptions stated explicitly therein). It's also useful to verify which condition is applicable to which lemmas or theorems.
	\item For $0\leq \tau < T$,
	\begin{equation*}
		Q_{\tau,T}:= \Omega\times (\tau, T) \quad \text{and}\quad M_{\tau,T}:= M\times(\tau,T).
	\end{equation*}
	When $\tau = 0$, we simply write $Q_T$ and $M_T$.
	\item For $1\leq p < \infty$,
	\begin{equation*}
		\|f\|_{L^p(Q_{\tau,T})}:= \bra{\int_\tau^T\intO |f(x,t)|^p}^{\frac 1p}
	\end{equation*}
	and for $p = \infty$,
	\begin{equation*}
		\|f\|_{L^\infty(Q_{\tau,T})}:= \mathrm{ess\,sup}_{Q_T}|f(x,t)|.
	\end{equation*}
	The spaces $L^p(M_{\tau,T})$ with $1\leq p\leq \infty$ are defined in the similar way.
	\item For $1\leq p \leq \infty$,
	\begin{equation*}
		W^{2,1}_p(Q_{\tau,T}):= \left\{f\in L^p(Q_{\tau,T}): \pa_t^r\pa_x^sf\in L^p(Q_{\tau,T}) \; \forall r,s\in \mathbb N, \, 2r+s\leq 2\right\}
	\end{equation*}
	with the norm
	\begin{equation*}
		\|f\|_{W^{2,1}_p(Q_{\tau,T})}:= \sum_{2r+s\leq 2}\|\pa_t^r\pa_x^s f\|_{L^p(Q_{\tau,T})}.
	\end{equation*}
\end{itemize}

\section{Construction of $L^p$-energy functions}
	In this section, we firstly state local existence of \eqref{eq:mainsys}, and provide the blow-up criterion, which were proved in \cite{sharma2016global}. The main part concerns the construction of an $L^p$-like energy function, and its relation to the intermediate sum condition \eqref{eq:intsum1}. Estimates derived from this energy function are crucial in the sequel analysis of this paper.
	
	\begin{theorem}[\eqref{O}-\eqref{D}-\eqref{F0}]\label{thm:local}\cite[Theorem 3.2]{sharma2016global}
		For any smooth initial data $(u_0,v_0)\in (W^{2-2/p}(\Omega))^{m_1}\times (W^{2-2/p}(M))^{m_2}$ for some $p>n$ satisfying the compatibility condition \eqref{compatibility}, there exists a unique classical solution to \eqref{eq:mainsys} on a maximal interval $(0,T_{\max})$, i.e. for any $0<T<T_{\max}$,
		\begin{equation*}
			(u,v)\in C([0,T];L^p(\Omega)^{m_1}\times L^p(M)^{m_2}) \cap L^\infty(0,T;L^\infty(\Omega)^{m_1}\times \LM{\infty}^{m_2}),
		\end{equation*}
		for any $p>n$,
		\begin{equation*}
			u\in (C^{2,1}(\overline{\Omega}\times(\tau,T)))^{m_1}, \quad v \in (C^{2,1}(M\times(\tau,T)))^{m_2} \quad \text{ for all } \quad 0<\tau<T,
		\end{equation*}
		and the equations \eqref{eq:mainsys} satisfy pointwise.
		
		\medskip
		The following blow-up criterion holds
		\begin{equation}\label{blowup}
			T_{\max}<+\infty \quad \Longrightarrow \quad \limsup_{t\nearrow T_{\max}}\bra{\sumi\|u_i(t)\|_{\LO{\infty}} + \sumj\|v_j(t)\|_{\LM{\infty}}} = +\infty.
		\end{equation}
		Moreover, if \eqref{F1} holds, then $(u(t),v(t))$ is (component-wise) non-negative provided the initial data $(u_0,v_0)$ is non-negative.
	\end{theorem}

	Thanks to Theorem \ref{thm:local}, the global existence of \eqref{eq:mainsys} follows if we can show that 
	\begin{equation}\label{eee}
		\sumi\|u_i\|_{\LQ{\infty}} + \sumj\|v_j\|_{\LS{\infty}} \leq C_T
	\end{equation}
	where $C_T$ depends continuously on $T>0$, and $C_T$ is finite for all $T>0$. For simplicity, we consider for the rest of this paper $0<T<T_{\max}$, and ultimately prove the estimate \eqref{eee}.
	
	We first show that, under the mass control condition \eqref{F2}, the solution is bounded in $L^\infty(0,T;\LO{1})$ and $L^\infty(0,T;\LM{1})$.
	\begin{lemma}[\eqref{O}-\eqref{D}-\eqref{F0}-\eqref{F1}-\eqref{F2}]\label{lemma:L1}
		There exists a constant $C_T$ depending on $T$ such that 
		\begin{equation}\label{L1_bound}
			\sumi\|u_i\|_{L^\infty(0,T;\LO{1})} + \sumj\|v_j\|_{L^\infty(0,T;\LM{1})} \leq C_T.
		\end{equation}
		Moreover, there exists a constant $C_T$ depending on $T$ such that
		\begin{equation}\label{L1_MT_bound}
			\sumi\|u_i\|_{\LS{1}} \leq C_T.
		\end{equation}
	\end{lemma}
	\begin{proof}
		Thanks to \eqref{eq:balsysmain}, we have
		\begin{equation}\label{ee1}
		\pa_t\bra{\sumi \intO a_iu_i + \sumj \intM b_jv_j} \leq L\intO\bra{\sumi u_i + 1} + L\intM\bra{\sumi u_i + \sumj v_j + 1}.
		\end{equation}
		Thus, for some constant $C>0$,
		\begin{equation}\label{ee2}
			\sumi \intO a_iu_i(t) + \sumj \intM b_jv_j(t) \leq C\int_0^t\bra{\sumi \intO a_iu_i + \sumj\intM b_jv_j} + C\sumi\int_0^t\intM u_i + C.
		\end{equation}
		We need to deal with the boundary integral of $u_i$ on the right hand side. Let $K>0$ be a constant to be determined later and $\phi_0$ is a non-negative, smooth function in $\overline{\Omega}$ satisfying $\pa_{\eta}\phi_0 = 1 + K\phi_0$ on $M$. Suppose  $\phi\in C^{2,1}(\bar{\Omega}\times[0,t])$ be a nonnegative function such that $\phi_t+\Delta\phi=0$ on $\Omega\times(0,t)$, $\frac{\partial \phi}{\partial\eta}=1 + K\phi$ on $M\times(0,t)$, and $\phi(\cdot,t) = \phi_0$ on $\overline{\Omega}$. It follows from the comparison principle that $0\leq \phi$. Moreover, if we set $\theta = -\pa_t\phi - \Delta_M\phi$ on $M\times(0,t)$.
		By integration by parts we have for $i=1,...,m_1$
		\begin{equation}\label{eq:L1esta}
		\begin{aligned}
		\int_0^t\intM a_id_iu_i(1+K\phi)&=\int_0^t\intM a_id_i u_i \frac{\partial \phi}{\partial\eta}\\
		&\le \int_0^t\intM \phi \cdot a_i G_i(u,v)+\int_0^t\intO \phi\cdot a_i F_i(u) \\
		&\quad +\int_0^t\intO a_iu_i\left(d_i-1\right)\Delta \phi+a_i\intO\phi(\cdot,0)u_i(\cdot,0).
		\end{aligned}
		\end{equation}
		Furthermore,  $j=1,...,m_2$
		\begin{equation}\label{eq:L1estb}
		\begin{aligned}
		\int_0^t\intM b_jv_j \theta &=\int_0^t\intM b_jv_j(-\phi_t-\Delta_M\phi)\\
		&\leq \int_0^t\intM \left( \phi \cdot b_j H_j(u,v) + b_jv_j(\delta_j-1)\Delta_M\phi \right)+b_j\intM v_j(\cdot,0)\phi(\cdot,0)
		\end{aligned}
		\end{equation}
		Now sum (\ref{eq:L1esta}) from $i=1,...,m_1$, with (\ref{eq:L1estb}) from $j=1,...,m_2$, and apply (\ref{eq:balsysmain}), it follows that
		\begin{equation*}
			\begin{aligned}
			&\sumi \int_0^t\intM a_id_iu_i(1+K\phi) + \sumj \int_0^t\intM b_jv_j\theta\\
			&\leq  L\int_0^t\intO \phi\bra{\sumi u_i + 1} + L\int_0^t\intM\phi\bra{\sumi u_i + \sumj v_j + 1}\\
			&\quad + \sumi \int_0^t\intO a_iu_i(d_i-1)\Delta \phi + \sumi a_i\intO \phi(\cdot,0)u_{i,0}\\
			&\quad + \sumj \int_0^t\intM b_jv_j(\delta_j-1)\Delta_M\phi + \sumj \intM \phi(\cdot,0)v_{j,0}\\
			&\leq L\sumi\int_0^t\intM \phi u_i + C\sumi\int_0^t\intO a_i u_i + C\sumj \int_0^t\intM b_j v_j + C(1+t)
			\end{aligned}
		\end{equation*}
		where the last step uses $\phi\in C^{2,1}(\overline \Omega\times[0,t])$. By choosing $K$ large enough such that $Kd_ia_i \geq L$ for all $i=1,\ldots, m_1$, and using $v_j\geq 0$ and $\theta\in \LS{\infty}$, we obtain
		\begin{equation}\label{ee3}
			\sumi \int_0^t\intM a_id_iu_i \leq C\sumi \int_0^t\intO a_iu_i + C\sumj \int_0^t\intM b_jv_j + C(1+t).
		\end{equation}
		Inserting this into \eqref{ee2} yields
		\begin{equation*}
			\sumi \intO a_iu_i(t) + \sumj \intM b_jv_j(t) \leq C\int_0^t\bra{\sumi \intO a_iu_i + \sumj\intM b_jv_j} + C(1+t).	
		\end{equation*}
		A direct application of Gronwall's inequality gives the estimates \eqref{L1_bound}. The bound \eqref{L1_MT_bound} follows directly from \eqref{ee3}.
\end{proof}

In the following lemma, we show that \eqref{eq:intsum1} implies the existence of functions $g_j, j=m_1-1,\ldots, 1$, which allow us to construct a desired $L^p$-energy function. 
\begin{lemma}[\eqref{eq:intsum1}]\label{lem6}
	\label{thm:thetastuff}There exist componentwise increasing functions $g_j:\R^{m_1-j}\to \R$ for $j=1,...,m_1-1$, such that if $\ell\in (0,\infty)^{m_1}$ with $\ell_j> g_j(\ell_{j+1},...,\ell_{m_1})$ for $j=1,...,m_1-1$ then there exists $L_\ell>0$ such that
	\begin{equation}\label{f1}
		\sumi \ell_i F_i(u)\le L_\ell \left(\sumi u_i^{\pO}+1\right)\text{ for all }u\in \mathbb{R}_+^{m_1},
	\end{equation}	
	and
	\begin{equation}\label{f2}
	\sumi \ell_i G_i(u,v)\le L_\ell \left(\sumi u_i^{\pM} + \sumj v_j^{\pM}+1\right)\text{ for all }u\in \mathbb{R}_+^{m_1}, v\in \R_+^{m_2}.
	\end{equation}
\end{lemma}

\begin{proof}
	Without loss of generality we assume that and $a_{i,j}>0$ for $i>j$ with $i,j\in{1,...,m_1}$. We construct two sequences of functions $g_j: \R^{m_1-j}\to \R$, $j=m_1-1,m_1-2,\ldots, 1$ and $\alpha_j: \R^{m_1-j+1} \to \R$, $j=m_1,\ldots, 1$ inductively as follows:
	\begin{itemize}
		\item $g_{m_1-1}(x_{m_1}):= a_{m_1,m_1-1}x_{m_1}$. We also define the function $\alpha_{m_1}(x_{m_1}):= x_{m_1}$. Note that $\wh{\alpha}_{m_1}:= \alpha_{m_1}(\ell_{m_1}) = \ell_{m_1}>0$.
		\item For $i=m_1-2,m_1-3,\ldots, 2, 1$, we constructed the function $g_i$ using established functions $g_j$ and $\alpha_j$ for $j\ge i+1$. More precisely, we define
		\begin{equation*}
			g_{i}(x_{i+1},x_{i+2},\ldots, x_{m_1}):= \sum_{j=i+1}^{m_1}a_{j,i}\alpha_j(x_j,x_{j+1}, \ldots, x_{m_1}),
		\end{equation*}
		and
		\begin{equation*}
			\alpha_i(x_i,\ldots, x_{m_1}):= x_i - g_i(x_{i+1},x_{i+2},\ldots, x_{m_1}).
		\end{equation*}
		Due to the assumptions of $(\ell_i)_{i=1,\ldots, m_1}$,
		\begin{equation*}
			\wh{\alpha}_i:= \alpha_i(\ell_i,\ell_{i+1},\ldots, \ell_{m_1})=\ell_i - g_i(\ell_{i+1},\ldots, \ell_{m_1})>0.
		\end{equation*}
	\end{itemize}
	Therefore, we have
	\begin{align*}
		\sumi \ell_iF_i(u) &= \sumi\sbra{\wh{\alpha}_i + g_i(\ell_{i+1},\ldots, \ell_{m_1})}F_i(u)\\
		&=\sumi\sbra{\wh{\alpha}_i + \sum_{j=i+1}^{m_1}\wh{\alpha}_ja_{j,i}}F_i(u)\\
		&=\sumi \sum_{j=i}^{m_1}\wh{\alpha}_ja_{j,i}F_i(u) \quad (\text{since } a_{i,i}=1)\\
		&=\sumi \wh{\alpha}_i\sum_{j=1}^{i}a_{i,j}F_j(u)\\
		&\leq \sumi \wh{\alpha}_iL\sbra{\sumi u_i^{\pO}+1} \quad (\text{due to } (\ref{eq:intsum1}))
	\end{align*}
	which implies directly \eqref{f1}. The inequality \eqref{f2} can be verified in the same way, so we omit it here.
\end{proof}

\begin{remark}\label{remark2}
	The above lemma shows that the intermediate sum conditions involving $F$ and $G$ in \eqref{eq:intsum1} and \eqref{eq:intsum2} imply the existence of functions $g_j$, $j=1,\ldots, m_1-1$. It will be shown later on that once we have these functions $g_j$, an $L^p$-energy can be constructed. This means that the conclusion of Theorem \ref{main:thm0} still holds true if we assume the existence of $g_j$, $j=1,\ldots, m_1-1$, instead of the intermediate sum conditions (involving $F$ and $G$) \eqref{eq:intsum1} and \eqref{eq:intsum2}. We choose to present Theorem \ref{main:thm0} under \eqref{eq:intsum1} and \eqref{eq:intsum2} as they are more constructive, and they appear naturally in many applications (see Section \ref{sec:application}).
	
	It is important to remark that the existence of functions $g_j$ is more general than the intermediate sum conditions \eqref{eq:intsum1}--\eqref{eq:intsum2}. For example, the nonlinearities
\begin{align*}
F_1(u_1,u_2)=u_1u_2^3-u_1^4\\
F_2(u_1,u_2)=u_1^4-u_1u_2^4,
\end{align*}
clearly do not satisfy \eqref{eq:intsum1} for high dimensions $n\geq 3$. However, if we define $g_1(x) = x$, then for any $\ell = (\ell_1,\ell_2)$ with $\ell_1 > g_1(\ell_2) = \ell_2$ we have
\begin{align*}
	\ell_1F_1(u_1,u_2) + \ell_2F_2(u_1,u_2) &= u_1\sbra{\ell_1 u_2^3 - \ell_2 u_2^4 + (\ell_2-\ell_1)u_1^3}\\
	&\leq u_1\sbra{\ell_1u_2^3 - \ell_2 u_2^4} \leq L_\ell (1+u_1)
\end{align*}
for all $u_1,u_2\ge 0$, and a constant $L_\ell$ depending only on $\ell$.
\end{remark}

The following interpolation inequality might be of independent interest.
\newcommand{\rO}{r_{\Omega}}
\newcommand{\rM}{r_{M}}
\begin{lemma}[\eqref{O}]\label{prepare}
	Assume that $w:\Omega \to [0,\infty)$ with $\|w\|_{\LO{a}} \leq K$. Then for any $p\geq 2a$, $\eps>0$, and any $\rO, \rM$ satisfying
	\begin{equation*}
	1\leq \rO < 1+ \frac{2a}{n} \quad \text{ and } 1\leq \rM < 1 + \frac{a}{n},
	\end{equation*}
	there exists a constant $C_{p,\eps,K}$ depending on $p, \eps$ and $K$, but not on $w$, such that
	\begin{equation}\label{c1_1}
		\intO w^{p-1+\rO} + \intM w^{p-1+\rM} \leq \eps\bra{\intO w^{p-2}|\na w|^2 + \intO w^{p}} + C_{p,\eps,K}.
	\end{equation}
\end{lemma}
\begin{proof}
	We will estimate the domain term on the left hand side of \eqref{c1_1}, while the boundary term will follow as a consequence. First, by using Sobolev's embedding we have
	\begin{equation}\label{c2_1}
	\begin{aligned}
	\intO w^{p-2}|\na w|^2 + \intO w^p = \frac{4}{p^2}\intO|\na(w^{p/2})|^2 + \intO(w^{p/2})^2\geq \frac{4}{p^2}\|w^{p/2}\|_{H^1(\Omega)}^2.
	\end{aligned}
	\end{equation}
	
	Define $\rO = 1+\eta$ and $\beta = \frac{2\eta}{p}$. Then we have $\eta\in [0,2a/n)$ and $\beta \in [0,\frac{4a}{np})$, and
	\begin{equation*}
		\intO w^{p-1+\rO} = \intO w^{p+\eta} = \intO|w^{p/2}|^{2+\beta} = \|y\|_{\LO{2+\beta}}^{2+\beta},
	\end{equation*}
	where $y:= w^{p/2}$. 
	Thanks to the Gagliardo-Nirenberg inequality, we have
	\begin{equation*}
		\|y\|_{\LO{2+\beta}}^{2+\beta} \leq C_{\rm{GN}}\|y\|_{H^1(\Omega)}^{\alpha(2+\beta)}\|y\|_{\LO{1}}^{(1-\alpha)(2+\beta)} 
	\end{equation*}
	where $\alpha \in (0,1)$ satisfies
	\begin{equation*}
		\frac{1}{2+\beta} = \bra{\frac 12 - \frac 1n}\alpha + \frac{1-\alpha}{1}.
	\end{equation*}
	It follows that
	\begin{equation*}
		\alpha(2+\beta) = \frac{2n(\beta+1)}{n+2}\quad \text{and}\quad (1-\alpha)(2+\beta) = \frac{4+2\beta - \beta n}{n+2}.
	\end{equation*}
	Therefore,
	\begin{equation*}
		\|y\|_{\LO{2+\beta}}^{2+\beta} \leq C_{\rm{GN}}\|y\|_{H^1(\Omega)}^{\frac{2n(\beta+1)}{n+2}}\|y\|_{\LO{1}}^{\frac{4+2\beta-\beta n}{n+2}}.
	\end{equation*}
	Since $\beta<4a/(np)$ and $p\geq 2a$, it holds $\frac{2n(\beta+1)}{n+2}<2$. We can then use Young's inequality to estimate
	\begin{equation}\label{c0-1}
		\|y\|_{\LO{2+\beta}}^{2+\beta} \leq \eps \|y\|_{H^1(\Omega)}^2 + C_{\eps}\|y\|_{\LO{1}}^{\frac{4+2\beta-\beta n}{2-n\beta}}.
	\end{equation}
	Changing the variable $y = w^{p/2}$ we have
	\begin{equation}\label{c0}
		\|y\|_{\LO{1}}^{\frac{4+2\beta-\beta n}{2-n\beta}} = \|w\|_{\LO{\frac p2}}^{\frac{p(4+2\beta-\beta n)}{2(2-n\beta)}}.
	\end{equation}
	If $p = 2a$, then this term is bounded by a constant depending on $K$, since $\|w\|_{\LO{a}} \leq K$. If $p>2a$, we use interpolation inequality to have
	\begin{equation}\label{c0_1}
		\|w\|_{\LO{\frac p2}} \leq \|w\|_{\LO{p+\eta}}^{\theta}\|w\|_{\LO{a}}^{1-\theta} \leq K^{1-\theta}\|w\|_{\LO{p+\eta}}^{\theta}
	\end{equation}
	where $\theta \in (0,1)$ satisfies
	\begin{equation*}
		\frac{2}{p} = \frac{\theta}{p+\eta} + \frac{1-\theta}{a}.
	\end{equation*}
	Note that 
	\begin{equation}\label{c0_2}
		\theta {\frac{p(4+2\beta-\beta n)}{2(2-n\beta)}} = \frac{(p+\eta)(p-2a)(4+2\beta-\beta n)}{2(2-n\beta)(p+\eta-a)} < p+\eta
	\end{equation}
	due to $\beta=2\eta/p$ and $\eta < 2a/n$. From \eqref{c0}--\eqref{c0_2} and Young's inequality, it follows that
	\begin{equation*}
		C_\eps\|y\|_{\LO{1}}^{\frac{4+2\beta-\beta n}{2-n\beta}} \leq C_\eps C_K\|w\|_{\LO{p+\eta}}^{\frac{(p+\eta)(p-2a)(4+2\beta-\beta n)}{2(2-n\beta)(p+\eta-a)}} \leq \frac 12 \|w\|_{\LO{p+\eta}}^{p+\eta} + C_{p,\eps,K}.
	\end{equation*}
	Inserting this into \eqref{c0-1}, and noting that $\|y\|_{\LO{2+\beta}}^{2+\beta} = \|w\|_{\LO{p+\eta}}^{p+\eta}$, we get
	\begin{equation*}
		\|w\|_{\LO{p+\eta}}^{p+\eta} \leq 2\eps\|w^{p/2}\|_{H^1(\Omega)}^2 + C_{p,\eps,K}.
	\end{equation*}
	Combining this with \eqref{c2_1} leads to the desired estimate for the domain term, i.e.
	\begin{equation}\label{c2}
		\intO w^{p-1+\rO} \leq \eps\bra{\intO w^{p-2}|\na w|^2 + \intO w^p} + C_{p,\eps,K}.
	\end{equation}
	\medskip
	To treat the boundary term $\intM w^{p-1+\pM}$, we first define $\pM = 1+\xi$ for $\xi \in [0,a/n)$. We then use the following interpolation trace inequality, see e.g. \cite[Proof of Theorem 1.5.1.10, page 41]{grisvard2011elliptic},
	\begin{equation*}
	\begin{aligned}
		\intM w^{p-1+\pM} = \intM w^{p+\xi} &\leq C\intO w^{p+\xi - 1}|\na w| + C\intO w^{p+\xi}\\
		&\leq \frac{\eps}{2}\intO w^{p-2}|\na w|^2 + C_{\eps}\intO w^{p+2\xi} + C\intO w^{p+\xi}\\
		&\leq \frac{\eps}{2}\intO w^{p-2}|\na w|^2 + C_{\eps}\intO w^{p+2\xi} + C.
	\end{aligned}
	\end{equation*}
	Since $\xi \in [0,a/n)$, $2\xi \in [0,2a/n)$. Therefore, we can use \eqref{c2} to show
	\begin{equation*}
		C_\eps\intO w^{p+2\xi} \leq \frac{\eps}{2}\bra{\intO w^{p-2}|\na w|^2 + \intO w^{p}} + C_{p,\eps,K}.
	\end{equation*}
	Therefore, we obtain the estimate for the boundary term in \eqref{c1_1}, and thus finish the proof of Lemma \ref{prepare}.
\end{proof}

To construct our $L^p$-energy function, we write $\mathbb{Z}_{+}^{k}$ for the set of all $k$-tuples of non negative integers. Addition and
scalar multiplication by non negative integers of elements in $\mathbb{Z}_{+}^{k}$
is understood in the usual manner. If $\beta=(\beta_{1},...,\beta_{k})\in \mathbb{Z}_{+}^{k}$ and $p\in \mathbb N$,
then we define $\beta^{p}=((\beta_{1})^{p},...,(\beta_{k})^{p})$.
Also, if $\alpha=(\alpha_{1},...,\alpha_{k})\in  \mathbb{Z}_{+}^{k}$, then
we define $|\alpha|=\sum_{i=1}^{k}\alpha_{i}$. Finally, if $z=(z_{1},...,z_{k})\in \mathbb{R}_{+}^{k}$
and $\alpha=(\alpha_{1},...,\alpha_{k})\in \mathbb{Z}_{+}^{k}$, then we define
$z^{\alpha}=z_{1}^{\alpha_{1}}\cdot...\cdot z_{k}^{\alpha_{k}}$,
where we interpret $0^{0}$ to be $1$. For $2\leq p\in \mathbb N$, we build our $L^p$-energy function of the form
\begin{equation}\label{Lp}
	\L_p[u](t) = \intO \H_p[u](t),
\end{equation}
where
\begin{equation}\label{Hp}
	\H_p[u](t) = \sum_{\beta\in \mathbb Z_+^{m_1}, |\beta| = p}\begin{pmatrix}
	p\\ \beta\end{pmatrix}\theta^{\beta^2}u(t)^{\beta},
\end{equation}
and the positive constants $\theta= (\theta_1,\ldots, \theta_{m_1})$ are to be chosen. For convenience, hereafter we drop the subscript $\beta\in \mathbb Z_+^{m_1}$ in the sum as it should be clear.

The main result of this section is the following lemma.
\begin{lemma}[\eqref{O}-\eqref{D}-\eqref{F0}-\eqref{F1}-\eqref{F2}-\eqref{eq:intsum1}-\eqref{eq:intsum2}-\eqref{pOpM}]\label{lemma:Loft}
	For any positive integer $p\geq 2$ and any constant $\eps>0$, there exists $K_{p,\eps}>0$ such that 
	\begin{equation}\label{eps_critical}
		\sumi\bra{\norm{u_i}_{\LQ{p-1+\pO}}^{p-1+\pO} + \norm{u_i}_{\LS{p-1+\pM}}^{p-1+\pM}} \leq K_{p,\eps}(1+T) + \eps\sumj\norm{v_j}_{\LS{p-1+\pM}}^{p-1+\pM}
	\end{equation}
	for a possibly different constant $K_{p,\eps}$. 
	
	Consequently, for any $1<p<\infty$ and any $\eps>0$, there exists a constant $K_{p,\eps}>0$ such that
	\begin{equation}\label{ff4}
		\sumi\bra{\|u_i\|_{\LQ{p}}+\|u_i\|_{\LS{p}}} \leq K_{p,\eps}(1+T) + \eps\sumj\|v_j\|_{\LS{p}}.
	\end{equation}
\end{lemma}
\begin{proof}
First, we choose $\theta=(\theta_{1},...,\theta_{m_1})$
with $\theta_{i}\ge1$ for all $i=1,...,m_1$ such that 
\begin{enumerate}[label=($\theta$\theenumi),ref=$\theta$\theenumi]
	\item\label{theta1} The matrix $\mathscr{M}=\left(\mathscr{M}_{i,j}\right)$ is positive
	definite, where 
	\begin{equation}\label{M}
	\mathscr{M}_{i,j}=\begin{cases}
	\begin{array}{cc}
	d_{i}\theta_{i}^{2}, & \text{ if \ensuremath{i=j}}\\
	\frac{d_{i}+d_{j}}{2}, & \text{ if \ensuremath{i\ne j}}
	\end{array}\end{cases}
	\end{equation}
	
	\item\label{theta2} $\theta_j> \underset{i_1,...,i_{m_1-j}\in\{1,...,2p-1\}}{\max}g_j\left(\theta_{j+1}^{i_1},...,\theta_{m_1}^{i_{m_1-j}}\right)$ for all $j=1,...,m_1-1$, where the functions $g_j$ are given by Lemma \ref{thm:thetastuff}.
\end{enumerate}
Such a choice of $\theta$ is possible since the off-diagonal elements of $\mathscr{M}$ are fixed, therefore we choose successively $\theta_{m_1}>0$, $\theta_{m_1-1}$, $\ldots$, $\theta_1$ such that \eqref{theta2} is fulfilled, and $\theta_{j}$ is large enough such that $\mathscr{M}$ is diagonally dominant, which implies its positive definiteness. With this chosen $\theta$, we define for $2\leq p\in \mathbb Z$ our $L^p$-energy function $\L_p[u]$ as in \eqref{Lp}. Then, thanks to Lemma \ref{Hp-lem7} 
\begin{equation}\label{f5}
\begin{aligned}
\bra{\L_p[u]}'(t)&=\int_{\Omega}\sum_{|\beta|=p-1}\left(\begin{array}{c}
p\\
\beta
\end{array}\right)\theta^{\beta^{2}}u^{\beta}\sum_{i=1}^{m_1}\theta_{i}^{2\beta_{i}+1}\frac{\partial}{\partial t}u_{i}\\
&=\int_{\Omega}\sum_{|\beta|=p-1}\left(\begin{array}{c}
p\\
\beta
\end{array}\right)\theta^{\beta^{2}}u^{\beta}\sum_{i=1}^{m_1}\theta_{i}^{2\beta_{i}+1}\left(d_{i}\Delta u_{i}+F_{i}(u)\right).
\end{aligned}
\end{equation}
Integration by parts gives 
\[
\int_{\Omega}\sum_{|\beta|=p-1}\left(\begin{array}{c}
p\\
\beta
\end{array}\right)\theta^{\beta^{2}}u^{\beta}\sum_{i=1}^{m_1}d_{i}\Delta u_{i}=:(I)+(II),
\]
where
\[
(I)=\int_{M}\sum_{|\beta|=p-1}\left(\begin{array}{c}
p\\
\beta
\end{array}\right)\theta^{\beta^{2}}u^{\beta}\sum_{i=1}^{m_1}\theta_{i}^{2\beta_{i}+1}G_{i}(u,v),
\]
and, thanks to Lemma \ref{Hp-lem8},
\begin{equation}\label{f3}
(II)=-\int_{\Omega}\sum_{|\beta|=p-2}\left(\begin{array}{c}
p\\
\beta
\end{array}\right)\theta^{\beta^{2}}u^{\beta}\sum_{l=1}^{n}\sum_{i,j=1}^{m_1}a_{i,j}\frac{\partial}{\partial x_{l}}u_{i}\frac{\partial}{\partial x_{l}}u_{j}
\end{equation}
with 
\begin{equation}\label{a_ij}
a_{i,j}=\begin{cases}
\begin{array}{cc}
d_{j}\theta_{i}^{2\beta_{i}+1}\theta_{j}^{2\beta_{j}+1}, & \text{ if \ensuremath{i\ne j}},\\
d_{i}\theta_{i}^{4\beta_{i}+4}, & \text{ if \ensuremath{i=j}}.
\end{array}\end{cases}
\end{equation}
Note that 
\begin{equation}\label{f4}
\sum_{i,j=1}^{m_1}a_{i,j}\frac{\partial}{\partial x_{l}}u_{i}\frac{\partial}{\partial x_{l}}u_{j}=\sum_{i,j=1}^{m_1}b_{i,j}\frac{\partial}{\partial x_{l}}u_{i}\frac{\partial}{\partial x_{l}}u_{j}
\end{equation}
where 
\[
b_{i,j}=\begin{cases}
\begin{array}{cc}
\frac{d_{i}+d_{j}}{2}\theta_{i}^{2\beta_{i}+1}\theta_{j}^{2\beta_{j}+1}, & \text{ if \ensuremath{i\ne j}}\\
d_{i}\theta_{i}^{4\beta_{i}+4}, & \text{ if \ensuremath{i=j}}
\end{array}\end{cases}
\]
Furthermore, if we define $\mathscr B = \left(b_{i,j}\right)$, and the $m_1\times m_1$ diagonal matrix 
\[
\mathscr C=\text{diag}\left(\theta_{i}^{-2\beta_{i}-1}\right)
\]
then
\[
\mathscr C^{-1}\mathscr M\mathscr C^{-1}=\mathscr{B}
\]
where $\mathscr{M}$ is defined in \eqref{M}. Consequently, from the choice of $\theta$, the matrix $\mathscr B$
is positive definite. Therefore, there exists $\lambda>0$ such that
\begin{equation*}
	\sum_{i,j=1}^{m_1}b_{i,j}\frac{\pa}{\pa x_l}u_i\frac{\pa}{\pa x_l}u_j \geq \lambda \sum_{i=1}^{m_1}\abs{\frac{\pa}{\pa x_l}u_i}^2.
\end{equation*}
Inserting this into \eqref{f4} and \eqref{f3} leads to, for some $c_p>0$,
\begin{equation}\label{eq:gradpstuff}
\begin{aligned}
(II) &\leq -\lambda \int_{\Omega}\sum_{|\beta|=p-2}\left(\begin{array}{c}
p\\
\beta
\end{array}\right)\theta^{\beta^{2}}u^{\beta}\sum_{l=1}^{n}\sum_{i=1}^{m_1}\abs{\frac{\partial}{\partial x_{l}}u_{i}}^2\\
&\leq -c_p\sum_{i=1}^{m_1}\intO u_i^{p-2}|\na u_i|^2.
\end{aligned}
\end{equation}

Also, $\theta_i\le\theta_i^{2\beta_i+1}\le\theta_i^{2p-1}$ for all $i=1,...,m_1$. Therefore, from Lemma \ref{thm:thetastuff}, \eqref{eq:intsum1}--\eqref{eq:intsum2} and the choice of $\theta$, there is a value $L_\theta>0$ such that
\[
\sum_{i=1}^{m_1}\theta_{i}^{2\beta_{i}+1}G_{i}(u,v)\le L_\theta \left(\sum_{i=1}^{m_1}u_{i}^{\pM}+\sum_{j=1}^{m_2}v_{j}^{\pM}+1\right)
\]
and
\[
\sum_{i=1}^{m_1}\theta_{i}^{2\beta_{i}+1}F_{i}(u)\le L_\theta \left(\sum_{i=1}^{m_1}u_{i}^{\pO}+1\right).
\]
As a result, there exist $c_{p},K_{p,\theta}>0$ so that 
\begin{equation}
\sum_{|\beta|=p-1}\left(\begin{array}{c}
p\\
\beta
\end{array}\right)\theta^{\beta^{2}}u^{\beta}\sum_{i=1}^{m_1}\theta_{i}^{2\beta_{i}+1}G_{i}(u,v)\le K_{p,\theta}\sum_{j=1}^{m_1}u_{j}^{p-1}\left(\sum_{i=1}^{m_1}u_{i}^{\pM}+\sum_{j=1}^{m_2}v_{j}^{\pM}+1\right),\label{eq:Gpstuff}
\end{equation}
and
\begin{equation}
\sum_{|\beta|=p-1}\left(\begin{array}{c}
p\\
\beta
\end{array}\right)\theta^{\beta^{2}}u^{\beta}\sum_{i=1}^{m_1}\theta_{i}^{2\beta_{i}+1}F_{i}(u)\le K_{p,\theta}\left(\sum_{i=1}^{m_1}u_{i}^{p-1+\pO}+1\right)\label{eq:Fpstuff}.
\end{equation}
By applying \eqref{eq:gradpstuff}, \eqref{eq:Gpstuff} and \eqref{eq:Fpstuff} into \eqref{f5}, there exists $K_{p,\theta}>0$ such that 
\begin{equation}\label{eq:LprimeEqn}
\begin{aligned}
(\L_p[u])'(t)&+c_p\sum_{i=1}^{m_1}\int_\Omega  u_i^{p-2}|\nabla u_i|^2\\
&\le  K_{p,\theta}\left[1+\int_\Omega \sum_{j=1}^{m_1}u_j^{p-1+\pO}+ \sumi\intM u_i^{p-1+\pM}+ \sumi\sumj\int_M u_i^{p-1}v_j^{\pM}\right]\\
&\leq K_{p,\theta,\eps}\sbra{1+\sumi\intO u_i^{p-1+\pO} + \sumi\intM u_i^{p-1+\pM}} + \eps\sumj\intM v_j^{p-1+\pM},
\end{aligned}
\end{equation}
where we used Young's inequality at the last step. Adding $$\sumi\bra{\frac{c_p}{2}\intO u_i^p + \intO u_i^{p-1+\pO} + \intM u_i^{p-1+\pM}}$$ to both sides gives
\begin{equation}\label{ff5}
\begin{aligned}
	\bra{\L_p[u]}'(t) + \frac{c_p}{2}\sumi \bra{\intO u_i^{p-2}|\na u_i|^2 + \intO u^p} + \frac 12\sumi\bra{\intO u_i^{p-1+\pO}+\intM u_i^{p-1+\pM}}\\
	\leq \sumi\bra{\frac{c_p}{2}\intO u_i^p + \intO u_i^{p-1+\pO} + \intM u_i^{p-1+\pM}} + K_{p,\theta,\eps} + \eps\sumj\intM v_j^{p-1+\pM}\\
	\leq K_{p,\theta,\eps} + \eps\sumi\bra{\intO u_i^{p-2}|\na u_i|^2 + \intO u_i^p} + \eps\sumj\intM v_j^{p-1+\pM}
\end{aligned}
\end{equation}
where we used Lemma \ref{prepare} at the last step. Integrating the resultant on $(0,T)$ finishes the proof of \eqref{eps_critical}. 

\medskip
To prove \eqref{ff4}, we first show that for each $p\in \mathbb N$, $\eps>0$, there exists $K_{p,\eps}$ such that
\begin{equation}\label{ff13}
	\sumi \|u_i\|_{\LQ{p-1+\pM}}^{p-1+\pM} \leq K_{p,\eps,T} + \eps\sumj\|v_j\|_{\LS{p-1+\pM}}^{p-1+\pM}.
\end{equation}
Indeed, from the $L^1$-bound in Lemma \ref{L1_bound} and interpolation inequality we have, for $\gamma\in (0,1)$ with $\frac{1}{p-1+\pM} = \frac{\gamma}{1} + \frac{1-\gamma}{p-1+\pO}$,
\begin{align*}
	\sumi\|u_i\|_{\LQ{p-1+\pM}}^{p-1+\pM} &\leq \sumi\|u_i\|_{\LQ{1}}^{\gamma(p-1+\pM)}\sumi\|u_i\|_{\LQ{p-1+\pO}}^{(1-\gamma)(p-1+\pM)}\\
	&\leq K_T\sumi\|u_i\|_{\LQ{p-1+\pO}}^{(1-\gamma)(p-1+\pM)}\\
	&\leq K_T\bra{1+\sumi\|u_i\|_{\LQ{p-1+\pO}}^{p-1+\pO}}\\
	&\leq K_{p,\eps,T} + \eps\sumj\|v_j\|_{\LS{p-1+\pM}}^{p-1+\pM}.
\end{align*}
The estimate \eqref{ff4} now follows from \eqref{eps_critical}, \eqref{ff13}, and interpolation.

%
\end{proof}

\section{Duality method}\label{sec:3}
We first collect some useful results which will be used in the following sections. The next lemma is about the $L^p$-maximal regularity for heat equation in a smooth manifold without boundary.
\begin{lemma}[\eqref{O}]\label{mr_constant}
	Let $1<p<\infty$, $0\leq \tau < T < \infty$. There exists a constant $C_{\mr,p}^M$ ($\mr$ stands for ``maximal regularity"), which depends only on $p$, $M$, if the dimension $n$, such that, for any $\ff\in L^p(M_{\tau,T})$, and $\uu$ is the solution to
	\begin{equation}\label{parabolic_M}
	\begin{cases}
	\pa_t\uu - \Delta_{M} \uu = \ff, &(x,t)\in M_{\tau,T},\\
	\uu(x,0) = 0, &x\in M,
	\end{cases}
	\end{equation}
	we have the following estimate
	\begin{equation*}
	\|\Delta_{M} \uu\|_{L^p(M_{\tau,T})} \leq C_{\mr,p}^M\|\ff\|_{L^p(M_{\tau,T})}.
	\end{equation*}
\end{lemma}
\begin{proof}
	The proof of this lemma can be found in \cite[Theorem 1]{lamberton1987equations}\footnote{A similar result was shown in \cite[Theorem 3.5]{sharma2016global} where the constant $C_{\mr,p}^M$ might depend on the time horizon $T - \tau$}. We emphasize that fact that the constant $C_{\mr,p}^M$ is independent of $T$ and $\tau$.
\end{proof}
The following crucial lemma provides the regularity of the duality problem suited for volume-surface systems.
\begin{lemma}[\eqref{O}]\label{dual_boundary}
	Assume that $0<\tau<T$, $1<q<\infty$ and $\psi\in L^{q}(M_{\tau,T})$. Let $\phi$ be the solution to 
	\begin{equation}\label{dual_eq}
	\left\{
	\begin{aligned}
	\pa_t\phi + \Delta \phi = 0, &\text{ on }Q_{\tau,T},\\
	\pa_t\phi + \delta \Delta_M \phi = -\psi, &\text{ on } M_{\tau,T},\\
	\phi(x,T) = 0 &\text{ on } \overline{\Omega}.
	\end{aligned}
	\right.
	\end{equation}
	Then, we have the following estimate
	\begin{equation}\label{scaled_ineq}
	\|\Delta_M \phi\|_{L^q(M_{\tau,T})} \leq \frac{C_{\mr,q}^M}{\delta}\|\psi\|_{L^q(M_{\tau,T})},		
	\end{equation}
	where $C_{\mr,q}^M$ is the maximal regularity constant in Lemma \ref{mr_constant}. Moreover, with $\xi = \frac{q}{n+1}$ we have
	\begin{equation}\label{W21_estimate}
	\|\phi\|_{W^{2,1}_{q+\xi}(Q_{\tau,T})} + \|\phi\|_{W^{2,1}_q(M_{\tau,T})} + \|\phi(\tau)\|_{\LO{q+\xi}}+ \|\phi(\tau)\|_{\LM{q}} \leq C_{T-\tau} \|\psi\|_{L^q(M_{\tau,T})},
	\end{equation}
	and
	\begin{equation}\label{flux_estimate}
	\norm{\pa_{\eta}\phi}_{L^{q+\xi}(M_{\tau,T})} \leq C_{T-\tau}\|\psi\|_{L^{q}(M_{\tau,T})}.
	\end{equation}
	Consequently, 
	\begin{equation}\label{Lq^*}
		\|\phi\|_{L^{q^{\dag}}(Q_{\tau,T})} + \|\phi\|_{L^{q^*}(M_{\tau,T})} \leq C_{T-\tau}\|\psi\|_{L^{q}(M_{\tau,T})}
	\end{equation}
	where
	\begin{equation}\label{crit_exp}
		q^{\dag} = 
		\begin{cases}
					\frac{(n+2)q}{n+1-2q} &\text{ if } q < \frac{n+1}{2},\\
					<+\infty \text{ abitrary } &\text{ if } q = \frac{n+1}{2},\\
					+\infty &\text{ if } q > \frac{n+1}{2}.
					\end{cases}\quad
		\text{ and }\quad
		q^* = \begin{cases}
			\frac{(n+1)q}{n+1-2q} &\text{ if } q < \frac{n+1}{2},\\
			<+\infty \text{ abitrary } &\text{ if } q = \frac{n+1}{2},\\
			+\infty &\text{ if } q > \frac{n+1}{2}.
			\end{cases}
	\end{equation}
	
	Moreover, if $\psi \geq 0$ a.e. in $M_{\tau,T}$, then $\phi\geq 0$.
\end{lemma}
\begin{remark}\hfill\
	\begin{itemize}
		\item At the first glance, the dual problem \eqref{dual_eq} looks like a backward parabolic equation. However, since the ``initial data" is considered at $t = T$, by the change of variable $s = T-\tau-t$, \eqref{dual_eq} transforms into the usual forward parabolic equation.
		\item To solve the dual problem, we first solve the boundary equation $\pa_t\phi + \delta\Delta_M\phi = -\psi$ and define $\Psi:= \phi|_{M_{\tau,T}}$. Then, we consider the heat operator with inhomogeneous Dirichlet boundary condition: $\pa_t \phi + \Delta \phi = 0$ in $Q_{\tau,T}$ and $\phi|_{M_{\tau,T}} = \Psi$. See more details in \cite{sharma2016global}.
\end{itemize}
\end{remark}
\begin{proof}
	The bound 
	\begin{equation*}
		\norm{\phi}_{W_q^{2,1}(M_{\tau,T})} \leq C_{T-\tau}\norm{\psi}_{L^q(M_{\tau,T})}
	\end{equation*}
	can be found in \cite{sharma2016global}. In particular,
	\begin{equation*}
		\|\pa_t\phi\|_{L^q(M_{\tau,T})} \leq C_{T-\tau}\|\psi\|_{L^q(M_{\tau,T})}.
	\end{equation*}
	Therefore, by H\"older's inequality
	\begin{equation*}
		\|\phi(\tau)\|_{\LM{q}}^q = \intM\abs{\int_\tau^T\pa_t\phi}^q \leq (T-\tau)^{q-1}\intM\int_{\tau}^{T}|\pa_t\phi|^q \leq C_{T-\tau}\|\psi\|_{L^{q}(M_{\tau,T})}^q
	\end{equation*}
	which implies
	\begin{equation*}
		\|\phi(\tau)\|_{\LM{q}} \leq C_{T-\tau}\|\psi\|_{L^q(M_{\tau,T})}.
	\end{equation*}
	
	To show the improved bound in $Q_{\tau,T}$ and for the normal derivative, we use the following embedding, cf. \cite[Lemma 3.3]{LSU68},
	\begin{equation*}
		W_q^{2,1}(M_{\tau,T})\hookrightarrow L^{r}(\tau,T;W^{2-\frac{1}{r}, r}(M))\cap W^{1-\frac{1}{2r}, r}(0,T;L^r(M))
	\end{equation*}
	for all
	\begin{equation*}
		r \leq \frac{(n+2)q}{n+1} = q + \frac{q}{n+1}.
	\end{equation*}
	Therefore, by choosing $\xi=\frac{q}{n+1}$ and defining $r= q + \xi$, we have $\Psi:= \phi|_{M_T}\in L^{r}(\tau,T;W^{2-\frac{1}{r}, r}(M))\cap W^{1-\frac{1}{2r}, r}(0,T;L^r(M))$. Now we can apply the maximal regularity for equation with inhomogeneous Dirichlet boundary condition, see e.g. \cite{pruss2002maximal},
	\begin{equation*}
	\begin{cases}
	\pa_t\phi + \Delta \phi = 0, &\text{ on } Q_{\tau,T},\\
	\phi = \Psi, &\text{ on } M_{\tau,T},\\
	\phi(x,T) = 0, &\text{ on } \overline{\Omega},
	\end{cases}
	\end{equation*}
	to obtain
	\begin{align*}
		\norm{\phi}_{W_{q+\xi}^{2,1}(Q_{\tau,T})} \leq C_{T-\tau}\norm{\Psi}_{L^{r}(\tau,T;W^{2-\frac{1}{r}, r}(M))\cap W^{1-\frac{1}{2r}, r}(0,T;L^r(M))}\\
		\leq C_{T-\tau}\norm{\phi}_{W_q^{2,1}(M_{\tau,T})} \leq C_{T-\tau}\norm{\psi}_{L^q(M\times(\tau,T))}.
	\end{align*}
	The estimate of the normal derivative \eqref{flux_estimate} follows from the bound of $\phi$ in $W_{q+\xi}^{2,1}(Q_{\tau,T})$ and Lemma \cite[Lemma 6.9]{sharma2016global}.
	
	The estimate \eqref{Lq^*} then follows from \eqref{W21_estimate} and the embedding theory (\cite[Lemma 3.3]{LSU68}).
	
	It remains to show \eqref{scaled_ineq}. We define scaled functions $\wh{\phi}(x,t) = \phi(x,t/\delta)$ and $\wh{\psi}(x,t) = \psi(x,t/\delta)$. From the equation of $\phi$, it follows that
	\begin{equation*}
	\begin{cases}
		\partial_t \wh{\phi} - \Delta_M\wh{\phi} = \frac{1}{\delta}\wh{\psi}, &(x,t)\in M_{\delta\tau,\delta T},\\
		\wh{\phi}(x,\delta \tau) = 0, &x\in M.
	\end{cases}
	\end{equation*}
	From Lemma \ref{mr_constant},
	\begin{equation*}
		\|\Delta_M\wh{\phi}\|_{L^q(M_{\delta \tau, \delta T})} \leq \frac{C_{\mr,q}^M}{\delta}\|\wh \psi\|_{L^q(M_{\delta \tau, \delta T})}.
	\end{equation*}
	By switching back to the original variables we obtain easily
	\begin{equation*}
		\delta\|\Delta_M\phi\|_{L^q(M_{\tau,T})}^q = \int_{\delta \tau}^{\delta T}\intM |\Delta_M\wh\phi|^q \leq \bra{\frac{C_{\mr,q}^M}{\delta}}^q\int_{\delta \tau}^{\delta T}\intM |\wh \psi|^q = \frac{\bra{C_{\mr,q}^M}^q}{\delta^{q-1}}\|\psi\|_{L^q(M_{\tau,T})}^q,
	\end{equation*}
	which yields the desired inequality \eqref{scaled_ineq}.
\end{proof}
We show that if \eqref{quasi-uniform} is true for some $\Lam$, then it's also true for $\Lam+\varkappa$ for small enough $\varkappa>0$ in the following sense.
\begin{lemma}\label{lem:improved_quasi_uniform}
	Assume that \eqref{quasi-uniform} holds for some $\Lam$. Then there exists $\varkappa_0>0$ such that
	\begin{equation*}
	\frac{\delta_{\max} - \delta_{\min}}{\delta_{\max} + \delta_{\min}}C_{\mr,(\Lam+\varkappa)'}^M < 1 \quad \text{ for all } \quad 0<\varkappa<\varkappa_0
	\end{equation*}
	where 
	\begin{equation*}
	(\Lam+\vk)' = \frac{\Lam+\vk}{\Lam + \vk - 1}
	\end{equation*}
	is the H\"older conjugate exponent of $\Lam+\vk$.
\end{lemma}
\begin{proof}
	It's sufficient to show that 
	\begin{equation}\label{claim2}
	\bra{C_{\mr,\Lam'}^M}^-:= \liminf_{\eta\to 0^+}C_{\mr,\Lam' - \eta}^M \leq C_{\mr,\Lam'}^M.
	\end{equation}
	Let $\Lam_\eta'$ satisfy
	\begin{equation*}
	\frac{1}{\Lam_\eta'} = \frac{1}{2}\sbra{\frac{1}{\Lam'} + \frac{1}{\Lam'-\eta}} \quad \text{ or equivalently } \quad \Lam_\eta' = \Lam' - \frac{\Lam'\eta}{2\Lam' - \eta}.
	\end{equation*}
	By the Riesz-Thorin interpolation theorem (cf. \cite[Chapter 2]{lunardi2018interpolation}),
	\begin{equation*}
	C_{\mr,\Lam_\eta'}^M \leq C_{\mr,\Lam'}^{1/2}C_{\mr,\Lam'-\eta}^{1/2}.
	\end{equation*}
	By letting $\eta \to 0$,
	\begin{equation*}
	\bra{C_{\mr,\Lam'}^M}^{-} \leq C_{\mr,\Lam'}^{1/2}\sbra{\bra{C_{\mr,\Lam'}^M}^{-}}^{1/2}
	\end{equation*}
	which yields the desired claim \eqref{claim2} and thus finishes the proof of Lemma \ref{lem:improved_quasi_uniform}.
\end{proof}

\begin{proposition}[\eqref{O}-\eqref{D}-\eqref{F0}-\eqref{F1}-\eqref{F2}-\eqref{eq:intsum1}-\eqref{eq:intsum2}-\eqref{pOpM}-\eqref{quasi-uniform}]\label{pro:duality}
	There exist constants $C_T$ and $\gamma>0$ such that
	\begin{equation*}
	\sumi\bra{\norm{u_i}_{\LQ{\Lambda+\gamma}} + \norm{u_i}_{\LS{\Lambda+\gamma}}} +\sumj \norm{v_j}_{\LS{\Lambda+\gamma}} \leq C_T.
	\end{equation*}
\end{proposition}
\begin{proof}
	Let $\vk_0$ be given in Lemma \ref{lem:improved_quasi_uniform}, and choose $\vk \in (0,\vk_0)$ small enough such that
	\begin{equation*}
	\frac{\Lam+\vk}{n+1} > \frac{\vk}{\Lam-1}.
	\end{equation*}
	Note that this is equivalent to
	\begin{equation}\label{l1}
	\Lam' = \frac{\Lam}{\Lam-1}< (\Lam+\varkappa)' + \frac{(\Lam+\varkappa)'}{n+1}
	\end{equation}
	where $(\Lam+\vk)'$ is the H\"older conjugate exponent of $\Lam+\vk$. Since we only need \eqref{l1} for $\vk$ sufficiently small and positive, note that \eqref{l1} is true when $\vk =0$, and therefore it is true for small positive $\vk$.
	
	Let $0\leq \psi\in \LS{(\Lam+\vk)'}$ with $ \|\psi\|_{\LS{(\Lam+\vk)'}} = 1$, and let $\phi$ be the solution to the dual problem \eqref{dual_eq} with
	\begin{equation*}
	\delta = \frac{\delta_{\max}+\delta_{\min}}{2},
	\end{equation*}
	where $\delta_{\max}$ and $\delta_{\min}$ are defined in \eqref{dmaxmin}.
	Thanks to \eqref{W21_estimate} in Lemma \ref{dual_boundary} and \eqref{l1}, we have
	\begin{equation}\label{eq:W12p}
	\|\phi\|_{W^{2,1}_{\Lam'}(Q_T)} + \norm{\pa_{\eta}\phi}_{\LS{\Lam'}} + \|\phi\|_{W^{2,1}_{(\Lam+\vk)'}(M_T)} + \norm{\phi(0)}_{\LO{\Lam'}} + \norm{\phi(0)}_{\LM{(\Lam+\vk)'}}\le C_{\Lam',T}.
	\end{equation}
	In particular, thanks to Lemma \ref{dual_boundary},
	\begin{equation*}
	\norm{\Delta_M \phi}_{\LS{(\Lam+\vk)'}} \leq \frac{C_{\mr,(\Lam+\vk)'}^M}{\delta} =  \frac{2C_{\mr,(\Lam+\vk)'}^M}{\delta_{\max}+\delta_{\min}}.
	\end{equation*}
	By integration by parts, we have
	\begin{align}
	0 &= -\sumi\intQT a_iu_i(\pa_t\phi + \Delta\phi)\nonumber\\
	&=\sumi\intO a_iu_{i,0}\phi(0) - \sumi \intMT a_id_iu_i\pa_{\eta}\phi + \sumi \intMT a_i\phi d_i\pa_{\eta}u_i\nonumber\\
	&\quad + \sumi \intQT a_i\phi(\pa_t u_i - d_i\Delta u_i) + \sumi \intQT(d_i-1)a_i u_i\Delta \phi\nonumber\\
	&\leq C\sumi \|u_{i,0}\|_{\LO{\Lam}}\norm{\phi(0)}_{\LO{\Lam'}} + C\sumi \norm{u_i}_{\LS{\Lam}}\norm{\pa_{\eta}\phi}_{\LS{\Lam'}}\nonumber\\
	&\quad + \intMT \phi \bra{\sumi a_iG_i(u,v)} + \intQT \phi \sumi a_iF_i(u) + C\sumi\norm{u_i}_{\LQ{\Lam}}\norm{\Delta\phi}_{\LQ{\Lam'}}\nonumber\\
	&\leq C_T\sumi\bra{ \norm{u_{i,0}}_{\LO{\Lam}} +\norm{u_i}_{\LQ{\Lam}}+ \norm{u_i}_{\LS{\Lam}}}\nonumber\\
	&\quad + L\intQT \phi\bra{\sumi u_i+1} + \intMT \phi\bra{\sumi a_iG_i(u,v)}\nonumber\\
	&\leq C_T\sumi\bra{1+ \norm{u_{i,0}}_{\LO{\Lam}} +\norm{u_i}_{\LQ{\Lam}}+ \norm{u_i}_{\LS{\Lam}}} + \intMT \phi\bra{\sumi a_iG_i(u,v)}.\label{est1}
	\end{align}
	On the other hand, we have
	\begin{align}
	&\intMT \bra{\sumj b_jv_j}\psi\nonumber\\
	&= -\sumj \intMT b_jv_j(\pa_t\phi + \delta\Delta_M \phi)\nonumber\\
	&= \sumj \intM v_{j,0}\phi(0) +\sumj\intMT b_j\phi(\pa_tv_j - \delta_j\Delta_Mv_j) + \sumj\intMT b_j(\delta_j - \delta)v_j\Delta_M\phi\nonumber\\
	&\leq \sumj\norm{v_{j,0}}_{\LM{\Lam+\vk}}\norm{\phi(0)}_{\LM{(\Lam+\vk)'}} + \sumj\intMT \phi b_jH_j(u,v)\nonumber\\
	&\quad + \frac{\delta_{\max}-\delta_{\min}}{2}\norm{\sumj b_jv_j}_{\LS{\Lam+\vk}}\norm{\Delta_M\phi}_{\LS{(\Lam+\vk)'}}\nonumber\\
	&\leq C_{(\Lam+\vk)',T}\sumj\norm{v_{j,0}}_{L^{\Lam+\vk}(M)} +  \intMT \phi \bra{ \sumj b_jH_j(u,v)}\nonumber\\
	&\quad + \frac{\delta_{\max}-\delta_{\min}}{\delta_{\max}+\delta_{\min}}C_{\mr,(\Lam+\vk)'}^M\norm{\sumj b_jv_j}_{\LS{\Lam+\vk}}.\label{est2}
	\end{align}
	By summing \eqref{est1} and \eqref{est2} and using the second inequality \eqref{F2}, we have
	\begin{equation}\label{est3}
	\begin{aligned}
	&\intMT \bra{\sumj b_jv_j}\psi\\
	&\leq C_T\sumi\bra{1+\norm{u_{i,0}}_{\LO{\Lam}}+\norm{u_i}_{\LQ{\Lam}}+\norm{u_i}_{\LS{\Lam}}} + C_T\sumj\|v_{j,0}\|_{\LM{\Lam+\vk}}\\
	&\quad + \intMT\phi\bra{\sumi u_i + \sumj v_j+1}+ \frac{\delta_{\max}-\delta_{\min}}{\delta_{\max}+\delta_{\min}}C_{\mr,(\Lam+\vk)'}^M\norm{\sumj b_jv_j}_{\LS{\Lam+\vk}}.
	\end{aligned}
	\end{equation}
	We show now that by choosing $\vk$ small enough, we have $\norm{\phi}_{\LS{\Lam'}} \leq C_T$. Indeed, from the boundary bound in \eqref{eq:W12p} we can use the embedding theorem to have
	\begin{equation*}
	\norm{\phi}_{\LS{h(\vk)}} \leq C_{T}\norm{\phi}_{W^{2,1}_{(\Lam+\vk)'}(M_T)} \leq C_T
	\end{equation*}
	where $h(\vk) = \frac{(n+1)(\Lam+\vk)'}{n+1-2(\Lam+\vk)'}$ when $(\Lam+\vk)' < \frac{n+1}{2}$, and $h(\vk)$ can be chosen arbitrarily large if $(\Lam+\vk)' \geq \frac{n+1}{2}$. Since $h(\vk)$ is strictly increasing in $\vk$ and $h(0) > \Lam'$, we can choose $\vk \in (0,\vk_0)$ small enough such that \begin{equation}\label{vk_small}
	\Lam' < h(\vk) = \frac{(n+1)(\Lam+\vk)'}{n+1-2(\Lam+\vk)'},
	\end{equation}
	and consequently,
	\begin{equation*}
	\norm{\phi}_{\LS{\Lam'}} \leq C_T\norm{\phi}_{\LS{h(\vk)}} \leq C_{T}.
	\end{equation*}
	Thus, we can estimate
	\begin{equation}\label{est4}
	\begin{aligned}
	&\intMT\phi\bra{\sumi u_i + \sumj v_j + 1}\\
	&\leq \bra{C_T+\sumi \norm{u_i}_{\LS{\Lam}}}\norm{\phi}_{\LS{\Lam'}} + \frac{1}{\min_{j=1,\ldots, m_2}b_j}\norm{\phi}_{\LS{\Lam'}}\norm{\sumj b_jv_j}_{\LS{\Lam}}\\
	&\leq C_T\bra{1+\sumi \norm{u_i}_{\LS{\Lam}} + \norm{\sumj b_jv_j}_{\LS{\Lam}}}.
	\end{aligned}
	\end{equation}
	Thanks to the interpolation inequality and the $L^1$-bound in Lemma \ref{lemma:L1} we have
	\begin{equation}\label{est5}
	\norm{\sumj b_jv_j}_{\LS{\Lam}} \leq \norm{\sumj b_jv_j}_{\LS{1}}^{\beta}\norm{\sumj b_jv_j}_{\LS{\Lam+\vk}}^{1-\beta} \leq C_{T,\beta}\norm{\sumj b_jv_j}_{\LS{\Lam+\vk}}^{1-\beta}
	\end{equation}
	where $\beta\in (0,1)$ satisfies $\frac{1}{\Lam} = \frac{\beta}{1} + \frac{1-\beta}{\Lam+\vk}$. Inserting \eqref{est4} and \eqref{est5} into \eqref{est3}, and applying \eqref{eps_critical} we have for any $\eps >0$
	\begin{equation*}
	\begin{aligned}
	\intMT \bra{\sumj b_jv_j}\psi &\leq C_{T,\eps}\sumi\bra{1+\norm{u_{i,0}}_{\LO{\Lam}}} + C_T\sumj\|v_{j,0}\|_{\LM{\Lam+\vk}} +\eps\norm{\sumj b_jv_j}_{\LS{\Lam+\vk}}\\
	&\quad + \frac{\delta_{\max}-\delta_{\min}}{\delta_{\max}+\delta_{\min}}C_{\mr,(\Lam+\vk)'}^M\norm{\sumj b_jv_j}_{\LM{\Lam+\vk}}+ C_{T,\beta}\norm{\sumj b_jv_j}_{\LM{\Lam+\vk}}^{1-\beta}. 
	\end{aligned}
	\end{equation*}
	By Young's inequality we have
	\begin{equation*}
	C_{T,\beta}\norm{\sumj b_jv_j}_{\LM{\Lam+\vk}}^{1-\beta} \leq C_{T,\eps,\beta} + \eps\norm{\sumj b_jv_j}_{\LS{\Lam+\vk}},
	\end{equation*}
	and consequently,
	\begin{equation*}
	\begin{aligned}
	\intMT\bra{\sumj b_jv_j}\psi\leq C_{T,\eps,\beta} + \bra{2\eps+ \frac{\delta_{\max}-\delta_{\min}}{\delta_{\max}+\delta_{\min}}C_{\mr,(\Lam+\vk)'}^M}\norm{\sumj b_jv_j}_{\LS{\Lam+\vk}}.
	\end{aligned}
	\end{equation*}
	Since $0\leq \psi\in \LS{(\Lam+\vk)'}$ with $\norm{\psi}_{\LS{(\Lam+\vk)'}} = 1$ arbitrary, we obtain by duality
	\begin{equation*}
	\begin{aligned}
	\norm{\sumj b_jv_j}_{\LS{\Lam+\vk}} \leq C_{T,\eps,\beta} + \bra{2\eps+ \frac{\delta_{\max}-\delta_{\min}}{\delta_{\max}+\delta_{\min}}C_{\mr,(\Lam+\vk)'}^M}\norm{\sumj b_jv_j}_{\LS{\Lam+\vk}}.
	\end{aligned}
	\end{equation*}
	Now thanks to Lemma \ref{lem:improved_quasi_uniform} we can choose $\eps$ small enough such that, for all $\vk\in (0,\vk_0)$ verifying \eqref{vk_small},
	\begin{equation*}
	2\eps+ \frac{\delta_{\max}-\delta_{\min}}{\delta_{\max}+\delta_{\min}}C_{\mr,(\Lam+\vk)'}^M < 1,
	\end{equation*}
	and finally obtain
	\begin{equation*}
	\norm{\sumj b_jv_j}_{\LS{\Lam+\vk}} \leq \bra{1-2\eps- \frac{\delta_{\max}-\delta_{\min}}{\delta_{\max}+\delta_{\min}}C_{\mr,(\Lam+\vk)'}^M}^{-1}C_{T,\eps,\beta}.
	\end{equation*}
	Combining this with \eqref{eps_critical} we can finish the proof of Proposition \ref{pro:duality}.
\end{proof}

\section{Proof of main results}\label{sec:4}
\subsection{Theorem \ref{thm:main1}: Global existence}\label{sec:4.1}
\begin{lemma}\label{elementary}
	Let $\{y_j\}_{j=1,\ldots, m_2}$ be a sequence of non-negative numbers. Assume that there is a constant $K>0$ such that, for any $\eps>0$, there exists $C_\eps>0$ independent of $\{y_i\}$ such that if $k\in \{1,\ldots, m_2\}$ then
	\begin{equation*}
		y_k \leq C_\eps+ K\sum_{j=1}^{k-1}y_j + \eps\sumj y_j,
	\end{equation*}
	where if $k=1$, the sum $\sum_{j=1}^{k-1}y_j$ is neglected. Then there exists a constant $C$ independent of the sequence $\{y_j\}$ such that 
	\begin{equation}\label{bound_elementary}
		\sumj y_j \leq C.
	\end{equation}
\end{lemma}
\begin{proof}
	We prove by induction that for each $k\in \{1,\ldots, m_2\}$, there exists a constant $B_k, R_k>0$ independent of $\eps$ such that
	\begin{equation}\label{induction}
		y_k \leq B_k + R_k\eps \sumj y_j.
	\end{equation}
	With $k=1$, \eqref{induction} follows for $B_1 = C_\eps$ and $R_1 = 1$. Assume that this is true for all $j=1,\ldots, k-1$. Then
	\begin{align*}
		y_k \leq C_\eps+K\sum_{j=1}^{k-1}y_j + \eps \sumj y_j \leq C_\eps + K\sum_{j=1}^{k-1}\bra{B_j + R_j\eps\sumj y_j} + \eps\sumj y_j \\
		\leq \underbrace{\sbra{C_\eps+K\sum_{j=1}^{k-1}B_j}}_{=: B_k} + \underbrace{\sbra{K\sum_{j=1}^{k-1}R_j + 1}}_{=:R_k}\eps\sumj y_j
	\end{align*}
	which proves \eqref{induction}. Now summing \eqref{induction} for $k=1,\ldots, m_2$ gives
	\begin{equation*}
		\sum_{k=1}^{m_2}y_k \leq \sum_{k=1}^{m_2}B_k + \bra{\sum_{k=1}^{m_2}R_k}\eps\sumj y_j,
	\end{equation*}
	and consequently, by choosing $\eps = \frac 12\bra{\sum_{k=1}^{m_2}R_k}^{-1}$, we obtain the desired bound \eqref{bound_elementary}.
\end{proof}

\begin{lemma}[\eqref{O}-\eqref{D}-\eqref{F0}-\eqref{F1}-\eqref{F2}-\eqref{eq:intsum1}-\eqref{eq:intsum2}-\eqref{pOpM}-\eqref{mM_general}-\eqref{quasi-uniform}]\label{lem:dual1}
For any $p>\Lam+\vk$, any $\eps>0$ and any $k\in \{1,\ldots, m_2\}$, there exists a constant $C_{T,\eps}$ depending on $T$ and $\eps$ such that 
\begin{equation}\label{inductive_estimate}
	\|v_k\|_{\LS{p}} \leq C_{T,\eps} + C_T\sum_{j=1}^{k-1}\|v_j\|_{\LS{p}} + \eps \sumj \|v_j\|_{\LS{p}}.
\end{equation}
Naturally, when $k = 1$, the first sum on the right-hand side is neglected.
\end{lemma}
\begin{proof}
	It's enough to show \eqref{inductive_estimate} for $p$ large enough. Let $0\le \psi \in \LS{p'}$ with $\|\psi\|_{\LS{p'}} = 1$, and $\phi$ be the solution to \eqref{dual_eq} with $\delta = \delta_k$. Recall that from the assumption  \eqref{eq:intsum2} we have for any $k = 1,\ldots, m_2$
	\begin{equation*}
		\sumi a_{(k+m_1)i}G_i(u,v) + \sum_{j=1}^k a_{(k+m_1)(j+m_1)}H_j(u,v) \leq L_2\sbra{\sumi u_i^{\mM} + \sumj v_j^{\mM} + 1}
	\end{equation*}
	which implies, recalling $a_{kk} = 1$ for all $k=1,\ldots, m_1+m_2$,
	\begin{equation}\label{a1}
	\begin{aligned}
		 H_k(u,v)
		& \leq -\sumi a_{(k+m_1)i}G_i(u,v) - \sum_{j=1}^{k-1} a_{(k+m_1)(j+m_1)}H_j(u,v)\\
		&\quad + L_2\sbra{\sumi u_i^{\mM} + \sumj v_j^{\mM} + 1}.
	\end{aligned}
	\end{equation}
	By integration by parts we have
	\begin{equation}\label{a2}
	\begin{aligned}
		\intMT v_k\psi &= -\intMT v_k(\partial_t \phi + \delta_k \Delta_M \phi)\\
		&= \intM v_{k,0}\phi(0) + \intMT \phi (\partial_t v_k - \delta_k \Delta_Mv_k)\\
		&= \intM v_{k,0}\phi(0) + \intMT \phi H_k(u,v)\\
		&=: (A) + (B).
	\end{aligned}
	\end{equation}
	From Lemma \ref{dual_boundary} we have
	\begin{equation}\label{A}
		|(A)| \leq C\|v_{k,0}\|_{\LM{p}}\|\phi(0)\|_{\LM{p'}} \leq C\|v_{k,0}\|_{\LM{p}}.
	\end{equation}
	To estimate $(B)$ we use $\phi\geq 0$ and \eqref{a1} to have
	\begin{equation}\label{B}
	\begin{aligned}
		(B) &\leq -\sumi \intMT a_{(k+m_1)i}G_i(u,v)\phi - \sum_{j=1}^{k-1}\intMT a_{(k+m_1)(j+m_1)}H_j(u,v)\phi\\
		&\quad + L_2\intMT \phi\sbra{\sumi u_i^{\mM} + \sumj v_j^{\mM} + 1 }\\
		&=: (B1) + (B2) + (B3).
	\end{aligned}
	\end{equation}
	
	\medskip
	\underline{Estimate of $(B1)$}. From the equation \eqref{eq:mainsys} we have
	\begin{align*}
		(B1) & =  -\sumi \intMT a_{(k+m_1)i}G_i(u,v)\phi\\
		& = -\sumi \intMT a_{(k+m_1)i}(d_i\partial_\eta u_i) \phi\\
		&= -\sumi a_{(k+m_1)i}\sbra{\intQT d_i\Delta u_i \phi + \intMT d_iu_i\partial_{\eta}\phi - \intQT d_iu_i\Delta \phi}\\
		&= -\sumi a_{(k+m_1)i}\sbra{\intQT \sbra{\partial_t u_i - F_i(u)} \phi + \intMT d_iu_i\partial_{\eta}\phi - \intQT d_iu_i\Delta \phi}\\
		&= \sumi a_{(k+m_1)i}\intO u_{i,0}\phi(0) + \intQT \phi\sbra{\sumi a_{(k+m_1)i}F_i(u)}\\
		&\quad + \sumi a_{(k+m_1)i}d_i\intMT u_i\partial_{\eta}\phi + \sumi a_{(k+m_1)i}\intQT u_i\sbra{\partial_t \phi + d_i\Delta \phi}\\
		&=: (B11) + (B12) + (B13) + (B14).
	\end{align*}
	By using Lemma \ref{dual_boundary} we have thanks to H\"older's inequality
	\begin{equation}\label{B11}
		|(B11)| \leq C\sumi\|u_{i,0}\|_{\LO{p}}\|\phi(0)\|_{\LO{p'}} \leq C_T\sumi\|u_{i,0}\|_{\LO{p}}.
	\end{equation}
	From \eqref{eq:intsum1} we have
	\begin{equation*}
		\sumi a_{(k+m_1)i}F_i(u) \leq L_2\sbra{\sumi u_i^{\pO} + 1}.
	\end{equation*}
	Therefore, by H\"older's inequality,
	\begin{equation}\label{B12r}
	\begin{aligned}
		|(B12)| \leq L_2\intQT \phi \sbra{\sumi u_i^{\pO} + 1} \leq L_2\sumi \intQT \phi u_i^{\pO} + C_T\|\phi\|_{\LQ{p'}}.
	\end{aligned}
	\end{equation}
	From Lemma \ref{dual_boundary} we have
	\begin{equation}\label{ff3}
		\|\phi\|_{\LQ{(p')^{\dag}}} \leq C_T
	\end{equation}
	with $(p')^{\dag}$ defined similarly to $q^{\dag}$ in \eqref{crit_exp}. For any $\beta\in (0,1)$, it follows from H\"older's inequality that
	\begin{equation}\label{c5}
		\intQT \phi u_i^{\pO} \leq \bra{\intQT \phi ^{(p')^{\dag}}}^{\frac{1}{(p')^\dag}}\bra{\intQT u_i^{(p_\Omega-\beta)s}}^{\frac 1s}\bra{\intQT u_i^p}^{\frac{\beta}{p}}
	\end{equation}
	where
	\begin{equation}\label{Sigma}
		\frac{1}{(p')^\dag} + \frac 1s + \frac{\beta}{p} = 1.
	\end{equation}
	This implies
	\begin{equation*}
		s = \frac{(n+2)p}{(n+2)(p-\beta) + 2p - (n+1)(p-1)}.
	\end{equation*}
	Therefore, from \eqref{pOpM} and the fact that $\Lam\geq 2$, we can always choose $\beta\in (0,1)$ and $p$ large enough such that 
	\begin{equation*}
		(\pO - \beta)s < \Lam+\vk,
	\end{equation*}
	and consequently
	\begin{equation*}
		\|u_i\|_{\LQ{(\pO-\beta)s}} \leq C_T\|u_i\|_{\LQ{\Lam+\vk}} \leq C_T.
	\end{equation*}
	Thus it follows from \eqref{c5} and \eqref{ff3} that
	\begin{equation}\label{c6}
		\intQT \phi u_i^{\pO} \leq \|\phi\|_{\LQ{(p')^\dag}}\|u_i\|_{\LQ{(p_\Omega-\beta)s}}^{\pO-\beta}\|u_i\|_{\LQ{p}}^{\beta} \leq C_T\|u_i\|_{\LQ{p}}^{\beta}.
	\end{equation}
	Therefore, from \eqref{B12r} we get the estimate for $(B12)$,
	\begin{equation}\label{B12}
		|(B12)| \leq C_{T} + \sumi\|u_i\|_{\LQ{p}}^{\beta} \leq C_T + \sumi \|u_i\|_{\LQ{p}},
	\end{equation}
	since $\beta\in (0,1)$. Due to H\"older's inequality, Lemmas \ref{lemma:Loft} and \ref{dual_boundary} we have
	\begin{equation}\label{B13}
	\begin{aligned}
		|(B13)| \leq C\sumi\|u_i\|_{\LS{p}}\|\partial_{\eta}\phi\|_{\LS{p'}}\leq C_T\sumi\|u_i\|_{\LS{p}}.
	\end{aligned}
	\end{equation}
	Finally, since $\partial_t \phi + \Delta \phi = 0$ in $Q_T$ we estimate $(B14)$ as
	\begin{equation}\label{B14}
	\begin{aligned}
		|(B14)| \leq \sumi a_{(k+m_1)i}|d_i -1|\intQT |u_i||\Delta \phi| &\leq C\sumi \|u_i\|_{\LQ{p}}\|\Delta\phi\|_{\LQ{p'}}\\
		& \leq C\sumi \|u_i\|_{\LQ{p}}.
	\end{aligned}
	\end{equation}
	From \eqref{B11}, \eqref{B12}, \eqref{B13}, and \eqref{B14} we obtain the estimate for $(B1)$ as
	\begin{equation}\label{B1}
		|(B1)| \leq C_{T} + \sumi\bra{\|u_i\|_{\LQ{p}}+\|u_i\|_{\LS{p}}}.
	\end{equation}
	
	\medskip
	\underline{Estimate of $(B2)$}. Since $\partial_tv_j - \delta_j\Delta_Mv_j = H_j(u,v)$, we have
	\begin{equation}\label{B2}
	\begin{aligned}
		|(B2)| &= \abs{\sum_{j=1}^{k-1}a_{(k+m_1)(j+m_1)}\intMT(\partial_tv_j - \delta_j\Delta_Mv_j)\phi}\\
		&= \abs{\sum_{j=1}^{k-1}a_{(k+m_1)(j+m_1)}\sbra{\intM v_{j,0}\phi(0) - \intMT v_j(\partial_t\phi + \delta_j \Delta_M\phi)}}\\
		&\leq C\sum_{j=1}^{k-1}\biggl[\|v_{j,0}\|_{\LM{p}}\|\phi(0)\|_{\LM{p'}}\\
		&\qquad\qquad\quad  + \|v_j\|_{\LS{p}}\bra{|\delta_j-\delta_k|\|\Delta_M\phi\|_{\LS{p'}} + \|\psi\|_{\LS{p'}}}\biggr]\\
		&\leq C\sum_{j=1}^{k-1}\bra{\|v_{j,0}\|_{\LM{p}} + \|v_j\|_{\LS{p}}}
	\end{aligned}
	\end{equation}	
	where we used Lemma \ref{dual_boundary} at the last step.
	
	\medskip
	\underline{Estimate of $(B3)$}. 
	We split $(B3)$ into three parts, as
	\begin{equation*}
		(B3) = L_2\sumi \intMT \phi u_i^{\mM} + L_2\sumj \intMT \phi v_j^{\mM} + L_2\intMT \phi =: (B31) + (B32) + (B33).
	\end{equation*}
	The term $(B33)$ can be estimated directly as
	\begin{equation}\label{B33}
		|(B33)| \leq C_T\|\phi\|_{\LS{p'}} \leq C_T.
	\end{equation}	
	For any $\alpha \in (0,1)$ we use H\"older's inequality to estimate
	\begin{equation}\label{ff1}
	\begin{aligned}
		\intMT \phi u_i^{\mM} = \intMT \phi u_i^{\mM-\alpha}u_i^{\alpha}\leq \bra{\intMT \phi^{(p')^*}}^{\frac{1}{(p')^*}}\bra{\intMT u_i^{(\mM-\alpha)r}}^{\frac{1}{r}}\bra{\intMT u_i^{p}}^{\frac{\alpha}{p}}
	\end{aligned}
	\end{equation}
	where $(p')^*$ is defined similarly as $q^*$ in \eqref{crit_exp}, and
	\begin{equation}\label{ff11}
		\frac{1}{(p')^*} + \frac{1}{r} + \frac{\alpha}{p} = 1.
	\end{equation}
	It follows that
	\begin{equation}\label{ff12}
		r = \frac{(n+1)p}{(n+1)(1-\alpha)+2p}.
	\end{equation}
	We now choose $\alpha \in (0,1)$ such that
	\begin{equation}\label{alpha}
	1-\alpha < \frac{2\vk}{(n+1)\sbra{1-\frac{\Lam+\vk}{p}}}.
	\end{equation}
	Combining this with \eqref{mM_general} gives
	\begin{equation*}
	\mM - \alpha < \frac{(\Lam+\vk)\sbra{(n+1)(1-\alpha) + 2p}}{(n+1)p},
	\end{equation*}
	and consequently
	\begin{equation*}
	(\mM - \alpha)r < \Lam+\vk.
	\end{equation*}
	Therefore, it follows from \eqref{ff1} that
	\begin{equation*}
		\intMT \phi u_i^{\mM} \leq \|\phi\|_{\LS{(p')^*}}\|u_i\|_{\LS{(\mM-\alpha)r}}^{\mM-\alpha}\|u_i\|_{\LS{p}}^{\alpha} \leq C_T\|u_i\|_{\LS{p}}^{\alpha}
	\end{equation*}
	thanks to \eqref{Lq^*} and $\|u_i\|_{\LS{(\mM-\alpha)r}} \leq C_T\|u_i\|_{\LS{\Lam+\vk}} \leq C_T$. Therefore, we have
	\begin{equation}\label{B31}
		|(B31)| \leq C_T\sumi\|u_i\|_{\LS{p}}^{\alpha} \leq C_T + \sumi \|u_i\|_{\LS{p}}.
	\end{equation}
	In the same way we can estimate
	\begin{equation}\label{B32}
		|(B32)| \leq C_T\sumj\|v_j\|_{\LS{p}}^{\alpha}.
	\end{equation}
	From \eqref{B31}, \eqref{B33}, and \eqref{B32} we obtain
	\begin{equation}\label{B3}
		|(B3)| \leq C_{T} + \sumi \|u_i\|_{\LS{p}}+ \sumj\|v_j\|_{\LS{p}}^{\alpha},
	\end{equation}
	with $\alpha \in (0,1)$ satisfying \eqref{alpha}.
	
	\medskip
	From the estimates of $(B1), (B2), (B3)$ in \eqref{B1}, \eqref{B2}, \eqref{B3}, we get
	\begin{equation*}
		|(B)| \leq C_T + C_T\bra{\sum_{j=1}^{k-1}\|v_j\|_{\LS{p}} +  \sumj\|v_j\|_{\LS{p}}^{\alpha} + \sumi\bra{\|u_i\|_{\LQ{p}} + \|u_i\|_{\LS{p}}}}.
	\end{equation*}
	By using estimate \eqref{ff4} and Young's inequality, we find
	\begin{equation*}
	|(B)| \leq C_T + C_T\bra{\sum_{j=1}^{k-1}\|v_j\|_{\LS{p}} +  \eps\sumj\|v_j\|_{\LS{p}}}.
	\end{equation*}
	Together with \eqref{A}, we get finally from \eqref{a2},
	\begin{equation*}
	\intMT v_k\psi \leq C_{T} + C_T\bra{\sum_{j=1}^{k-1}\|v_j\|_{\LS{p}} +  \eps\sumj\|v_j\|_{\LS{p}}},
	\end{equation*}
	which yields the desired estimate \eqref{inductive_estimate} thanks to duality.
\end{proof}

By combining Lemmas \ref{lemma:Loft}, \ref{elementary}, and \ref{lem:dual1} we get
\begin{lemma}[\eqref{O}-\eqref{D}-\eqref{F0}-\eqref{F1}-\eqref{F2}-\eqref{eq:intsum1}-\eqref{eq:intsum2}-\eqref{pOpM}-\eqref{mM_general}-\eqref{quasi-uniform}]\label{lem1}
	For any $2\leq p \in \mathbb Z$, there exists a constant $C_{T,p}$ such that
	\begin{equation*}
		\sumi\bra{\|u_i\|_{\LQ{p}} + \|u_i\|_{\LS{p}}} + \sumj\|v_j\|_{\LS{p}} \leq C_{T,p}.
	\end{equation*}
\end{lemma}
\begin{proof}
	From Lemmas \ref{lem:dual1} and \ref{elementary} we obtain
	\begin{equation*}
		\sumj\|v_j\|_{\LS{p}} \leq C_{T,p}.
	\end{equation*}
	This and Lemma \ref{lemma:Loft} imply the desired estimate.
\end{proof}
As a consequence, we obtain the global existence of \eqref{eq:mainsys} in Theorem \ref{thm:main1} part (i).
\begin{theorem}[\eqref{O}-\eqref{D}-\eqref{F0}-\eqref{F1}-\eqref{F2}-\eqref{F3}-\eqref{eq:intsum1}-\eqref{eq:intsum2}-\eqref{pOpM}-\eqref{mM_general}-\eqref{quasi-uniform}]\label{thm2}
	Let $p>n$. 
	For any non-negative initial data $(u_0,v_0) \in (W^{2-2/p}(\Omega))^{m_1}\times (W^{2-2/p}(M))^{m_2}$ satisfying the compatibility condition \eqref{compatibility}, the system \eqref{eq:mainsys} has a unique global strong solution in all dimensions.
\end{theorem}
\begin{proof}
	From Lemma \ref{lem1} and \eqref{F3}, the nonlinearities $F_i(u)$ are bounded in $\LQ{p}$, and $G_i(u,v), H_j(u,v)$ are bounded in $\LS{p}$, for all $p\geq 1$. It follows that 
	\begin{equation*}
		\begin{cases}
			\pa_t u_i - d_i\Delta u_i = F_i(u)\in \LQ{p},\\
			d_i\pa_{\eta}u_i = G_i(u,v)\in \LS{p},
		\end{cases}
	\end{equation*}
	for any $p\geq 1$. By the regularizing effect of linear parabolic equations with inhomogeneous boundary conditions, cf. \cite[Proposition 3.1]{nittka2014inhomogeneous}, it follows that $\|u_i\|_{\LQ{\infty}}$ is bounded. Similarly, $\|v_j\|_{\LS{\infty}}$ is bounded.	This implies that the solution is bounded in $L^{\infty}$, which implies the desired global existence.
\end{proof}

\subsection{Theorem \ref{thm:main1}: Uniform-in-time bounds}\label{sec:uniform-in-time}
To complete the proof of Theorem \ref{main:thm0}, it remains to show the uniform-in-time bound of the solution. To this end, we study the system \eqref{eq:mainsys} on each cylinder $Q_{\tau,\tau+1} = \Omega\times(\tau,\tau+1)$, $\tau\in \mathbb N$.

\medskip
For the rest of this section, {\it all constants are independent of $\tau$} unless otherwise stated.	We also consider an increasing, smooth function $\varphi\in C^\infty(\R;[0,1])$ such that $\varphi(s) = 0$ for $s\in (-\infty,0]$, $\varphi(s) = 1$ for $s\geq 1$, and its shifted version $\varphi_\tau(\cdot) = \varphi(\cdot - \tau)$. By multiplying the system \eqref{eq:mainsys} by $\vat$ we have the truncated system
\begin{equation}\label{shifted_sys}
\begin{cases}
	\pa_t(\vat u_i) = d_i\Delta(\vat u_i) + \vat' u_i + \vat F_i(u), &(x,t)\in Q_{\tau,\tau+2},\; i=1,\ldots, m_1,\\
	d_i\pa_\eta(\vat u_i) = \vat G_i(u,v), &(x,t)\in M_{\tau,\tau+2},\; i=1,\ldots, m_1,\\
	\pa_t(\vat v_j) = \delta_j\Delta_M(\vat v_j) + \vat' v_j + \vat H_j(u,v), &(x,t)\in M_{\tau,\tau+2}, \; j = 1,\ldots, m_2,
\end{cases}
\end{equation}
with {\it zero initial data}
\begin{equation}\label{zero_initial}
\begin{cases}
	(\vat u_i)(x,\tau) = 0, & x\in\Omega,\; i=1,\ldots, m_1,\\
	(\vat v_j)(x,\tau) = 0, & x\in M, \; j=1,\ldots, m_2.
\end{cases}
\end{equation}

\begin{lemma}[\eqref{O}-\eqref{D}-\eqref{F0}-\eqref{F1}-\eqref{F2}]\label{L1bound_cylinder}
	If $L< 0$ or $L=K=0$ in \eqref{F2}, then 
	\begin{equation*}
		\sup_{t\ge 0}\sumi\|u_i(t)\|_{\LO{1}} + \sup_{t\ge 0}\sumj\|v_j(t)\|_{\LM{1}} + \sup_{\tau\in\mathbb N}\sumi\|u_i\|_{\LStau{1}} \leq C.
	\end{equation*}
\end{lemma}
\begin{proof}
	From \eqref{F2} it follows that
	\begin{equation*}
		\frac{d}{dt}\bra{\sumi\intO a_iu_i + \sumj\intM b_jv_j} \leq L\bra{\sumi \intO u_i + \sumj \intM v_j} + K.
	\end{equation*}
	If $L = K = 0$, it yields directly by integrating on $(0,t)$ that 
	\begin{equation}\label{b1}
		\sup_{t\ge 0}\sumi\|u_i(t)\|_{\LO{1}} + \sup_{t\ge 0}\sumj\|v_j(t)\|_{\LM{1}} \leq C.
	\end{equation}
	If $L<0$, there exists $\sigma>0$ such that
	\begin{equation*}
		\frac{d}{dt}\bra{\sumi\intO a_iu_i + \sumj\intM b_jv_j} \leq -\sigma\bra{\sumi\intO a_iu_i + \sumj\intM b_jv_j} + K,
	\end{equation*}
	which also implies \eqref{b1} thanks to Gronwall's lemma. Let $\phi\in C^{2,1}(\bar{\Omega}\times[\tau,\tau+2])$ be a nonnegative function such that $\phi_t +\Delta \phi = 0$ on $Q_{\tau,\tau+2}$, $\pa_{\eta}\phi = 1$ on $M_{\tau,\tau+2}$ and $\phi(\cdot,\tau+2) = 0$ in $\bar \Omega$.  Define $\theta = -\pa_t\phi - \Delta_M\phi$ on $M\times(\tau,\tau+2)$. By integration by parts,
	\begin{equation}\label{d1}
	\begin{aligned}
		\intMtautwo a_id_i(\vat u_i) &= \intMtautwo a_id_i(\vat u_i)\pa_{\eta}\phi\\
		&= \intMtautwo \vat \phi \cdot a_iG_i(u,v) + a_i\intQtautwo \phi\vat'  u_i\\
		&\quad + \intQtautwo \vat \phi \cdot a_iF_i(u) + \intQtautwo a_i(\vat u_i)(d_i-1)\Delta \phi,
	\end{aligned}
	\end{equation}
	and
	\begin{equation}\label{d2}
	\begin{aligned}
		\intMtautwo b_j(\vat v_j)\theta = \intMtautwo\sbra{\vat\phi\cdot  b_jH_j(u,v) + \vat' v_j\phi + b_jv_j(\delta_j - 1)\Delta_M\phi}.
	\end{aligned}
	\end{equation}
	By summing \eqref{d1} in $i=1,\ldots, m_1$, summing \eqref{d2} in $j=1,\ldots, m_2$, and adding the resultants we can apply \eqref{F2} (recalling $L=0$) to get
	\begin{equation*}
	\begin{aligned}
		& \sumi \intMtautwo a_id_i(\vat u_i) + \sumj \intMtautwo b_j(\vat v_j)\theta\\
		& \leq \sumi \sbra{\intQtautwo \phi \vat' u_i + \intQtautwo a_i(\vat u_i)(d_i-1)\Delta \phi} + K\bra{\intQtautwo \vat \phi + \intMtautwo \vat \phi}\\
		&\quad + \sumj \sbra{\intMtautwo \phi \vat' u_i + \intMtautwo b_j(\vat v_j)(\delta_j-1)\Delta_M\phi}.
	\end{aligned}
	\end{equation*}
	Thanks to the fact that $\theta \in \LStaut{\infty}$, $\sup_{t\geq 0}\|u_i(t)\|_{\LO{1}} \leq C$ and $\sup_{t\geq 0}\|v_j(t)\|_{\LM{1}} \leq C$, $\phi \in C^{2,1}(\bar \Omega\times[\tau,\tau+2])$, we conclude that
	\begin{equation*}
		\sumi \intMtautwo \vat u_i \leq C \quad \text{ for all } \quad \tau\in \mathbb N.
	\end{equation*}
	Since $\vat \ge 0$ and $\vat|_{[\tau+1,\tau+2]} = 1$, we get finally
	\begin{equation*}
		\sup_{\tau\in \mathbb N}\sumi\|u_i\|_{\LStau{1}} \leq C.
	\end{equation*}
\end{proof}

\begin{lemma}[\eqref{O}-\eqref{D}-\eqref{F0}-\eqref{F1}-\eqref{F2}-\eqref{eq:intsum1}-\eqref{eq:intsum2}-\eqref{pOpM}]\label{Lp_cylinder} If $L<0$ or $L=K=0$ in \eqref{F2}, then for any positive integer $p\ge 2$, and any $\eps>0$, there exists $K_{p,\eps}>0$ such that
		\begin{equation}\label{d3_1}
		\begin{aligned}
			\sumi \bra{\intQtautwo (\vat u_i)^{p-1+\pO} + \intMtautwo (\vat u_i)^{p-1+\pM}}\\
			\leq K_{p,\eps} + \eps\sbra{\sumi\bra{\intQtautwo u_i^{p-1+\pM} + \intMtautwo u_i^{p-1+\pM}}+ \sumj \intMtautwo v_j^{p-1+\pM}}.
		\end{aligned}
		\end{equation}
		
		As a consequence, for any $1<p<\infty$ and any $\eps>0$, there exists $K_{p,\eps}>0$ such that
		\begin{equation}\label{ff6}
		\begin{aligned}
			\sumi\bra{\|\varphi_\tau u_i\|_{\LQtaut{p}} + \|\varphi_\tau u_i\|_{\LStaut{p}}}\\
			\leq K_{p,\eps} + \eps\sbra{\sumi\bra{\|u_i\|_{\LQtaut{p}} + \|u_i\|_{\LStaut{p}}} + \sumj\|v_j\|_{\LStaut{p}}}.
		\end{aligned}
		\end{equation}
\end{lemma}
\begin{proof}
	Recall the Lyapunov-like function $\L_p[u]$ in \eqref{Lp}, with $\theta$ is chosen in \eqref{theta1} and \eqref{theta2}. Thanks to \eqref{ff5},
	\begin{equation*}
	\begin{aligned}
		\bra{\L_p[u]}'(t) + C\sumi \bra{\intO u_i^{p-1+\pO} + \intM u_i^{p-1+\pM}}\leq K_{p,\theta}\sbra{1 + \eps\sumj\intM v_j^{p-1+\pM}}.
	\end{aligned}
	\end{equation*}
	Therefore, we have
	\begin{equation}\label{d4}
	\begin{aligned}
		\bra{\vat\L_p[u]}' + C\sumi \bra{\intO \vat u_i^{p-1+\pO} + \intM \vat u_i^{p-1+\pM}}\\
		\leq \vat' \L_p(t) + K_{p,\theta,\eps}\vat + \eps\sumj\intM \vat v_j^{p-1+\pM}.
	\end{aligned}
	\end{equation}
	Since $0\leq \vat \leq 1$, 
	\begin{equation*}
		\vat u_i^{p-1+\pO} \ge (\vat u_i)^{p-1+\pO}, \quad \vat u_i^{p-1+\pM} \ge (\vat u_i)^{p-1+\pM}.
	\end{equation*}
	From \eqref{Lp} and $0\leq \vat' \leq C$, it follows that 
	\begin{equation*}
		\abs{\vat' \L_p[u](t)} \leq C\sumi \intO u_i^p \leq C_\eps + \eps\sumi \intO u_i^{p-1+\pM}.
	\end{equation*}
	Putting all these into \eqref{d4} and integrating the resultant on $(\tau,\tau+2)$, noticing that $\vat(\tau) = 0$, we obtain
	\begin{equation*}
	\begin{aligned}
	\intQtautwo (\vat u_i)^{p-1+\pO} + \intMtautwo (\vat u_i)^{p-1+\pM}\\
	\leq K_{p,\theta,\eps} + \eps\sbra{\sumi\bra{\intQtautwo u_i^{p-1+\pM} + \intMtautwo u_i^{p-1+\pM}} + \sumj\intMtautwo  v_j^{p-1+\pM}},
	\end{aligned}
	\end{equation*}
	which is the desired estimate \eqref{d3_1}. From this, \eqref{ff6} can be obtained similarly to the last step of Lemma \ref{lemma:Loft}.
\end{proof}

We need the following elementary result whose proof is straightforward.
\begin{lemma}\label{lem:elementary}
	Let $\{y_n\}_{n\ge 0}$ be a nonnegative sequence and $\mathscr N = \{n\in \mathbb N: y_{n-1}\leq y_n \}$. If there exists $K>0$ (independent of $n$) such that
	\begin{equation*}
	y_n \leq K \quad \text{ for all } \quad n\in \mathscr N,
	\end{equation*}
	then
	\begin{equation*}
	y_n \leq \max\{y_0, K\} \quad \text{ for all } \quad n\in \mathbb N.
	\end{equation*}
\end{lemma}

\begin{lemma}[\eqref{O}-\eqref{D}-\eqref{F0}-\eqref{F1}-\eqref{F2}-\eqref{eq:intsum1}-\eqref{eq:intsum2}-\eqref{pOpM}-\eqref{quasi-uniform}]\label{lem4} There exist constants $C>0$ and $\gamma>0$ such that for all $\tau\in \mathbb N$,
	\begin{equation}\label{desired1}
		\sumi\bra{\|u_i\|_{\LQtau{\Lam+\gamma}}+\|u_i\|_{\LStau{\Lam+\gamma}}} + \sumj\|v_j\|_{\LStau{\Lam+\gamma}} \leq C.
	\end{equation}
\end{lemma}
\begin{proof}
	As in Lemma \ref{pro:duality}, we choose $\vk>0$ small enough such that \eqref{l1} holds. Let $0\leq \psi \in \LStaut{(\Lam+\vk)'}$ with $\|\psi\|_{\LStaut{(\Lam+\vk)'}} = 1$, and let $\phi$ be the solution to \eqref{dual_eq} with $T = \tau+2$ and
	\begin{equation*}
		\delta = \frac{\delta_{\max}+\delta_{\min}}{2}
	\end{equation*}
	with $\delta_{\max}$ and $\delta_{\min}$ are in \eqref{dmaxmin}. From Propositions \ref{dual_boundary} and \eqref{l1}, we have
	\begin{equation}\label{e1}
		\|\phi\|_{W^{2,1}_{\Lam'}(Q_{\tau,\tau+2})} + \|\pa_\eta\phi\|_{\LStaut{\Lam'}} + \|\phi\|_{W^{2,1}_{(\Lam+\vk)'}(M_{\tau,\tau+2})} \leq C.
	\end{equation}
	In particular,
	\begin{equation}\label{ff7}
		\|\Delta_M\phi\|_{\LStaut{(\Lam+\vk)'}} \leq \frac{2C_{\mr,(\Lam+\vk)'}^M}{\delta_{\max}+\delta_{\min}}.
	\end{equation}
	By integration by parts (see the proof of Lemma \ref{pro:duality}) we have
	\begin{equation}\label{d11}
	\begin{aligned}
		0&= -\sumi \intQtautwo a_i(\vat u_i)(\pa_t\phi + \Delta \phi)\\
		&= \sumi\intQtautwo a_i\phi(\vat' u_i + \vat F_i(u)) - \sumi \intMtautwo d_ia_i\vat u_i\pa_{\eta}\phi\\
		&\quad + \sumi \intMtautwo \phi \vat a_iG_i(u,v) + \sumi \intQtautwo a_i(d_i-1)\vat u_i\Delta\phi
	\end{aligned}
	\end{equation}
	and
	\begin{equation}\label{d12}
	\begin{aligned}
		\intMtautwo\bra{\sumj b_j\vat v_j}\psi &=\sumj \intMtautwo b_j\phi \vat' v_j + \sumj \intMtautwo \phi \vat b_jH_j(u,v) \\
		&\quad + \sumj \intMtautwo b_j(\delta_j-\delta)(\vat v_j)\Delta_M\phi.
	\end{aligned}
	\end{equation}
	Sum \eqref{d11} and \eqref{d12}, and note that either $L < 0$ or $L=K=0$ in \eqref{eq:intsum1} and \eqref{eq:intsum2}. We provide the argument in the case when $L=K=0$, and leave the similar case when $L<0$ for the reader. We calculate
	\begin{equation}\label{d13}
	\begin{aligned}
		\intMtautwo\bra{\sumj b_j\vat v_j}\psi &\leq \sumi\intQtautwo \phi a_i\vat' u_i + \sumi \intQtautwo a_i(d_i-1)\vat u_i\Delta \phi\\
		&\quad-\sumi\intMtautwo a_id_i\vat u_i\pa_{\eta}\phi + \sumj \intMtautwo b_j\phi \vat' v_j \\
		&\quad + \sumj \intMtautwo b_j(\delta_j-\delta)(\vat v_j)\Delta_M\phi\\
		&\quad=: (I)+(II)+(III)+(IV)+(V).
	\end{aligned}
	\end{equation}
	We estimate five terms on the right hand side of \eqref{d13} as follows. We use \eqref{e1} to estimate, for $\gamma\in (0,1)$ such that $\frac{1}{\Lam}=\frac{\gamma}{1}+\frac{1-\gamma}{\Lam+\vk}$ and any $\eps>0$,
	\begin{equation}\label{e3}
	\begin{aligned}
		|(I)| &\leq C\sumi\|\phi\|_{\LQtaut{\Lam'}}\|u_i\|_{\LQtaut{\Lam}} \leq C\sumi\|u_i\|_{\LQtaut{1}}^{\gamma}\|u_i\|_{\LQtaut{\Lam+\vk}}^{1-\gamma}\\
		&\leq C\sumi\|u_i\|_{\LQtaut{\Lam+\vk}}^{1-\gamma} \leq C_\eps + \eps\sumi\|u_i\|_{\LQtaut{\Lam+\vk}}.
	\end{aligned}
	\end{equation}
	Similarly,
	\begin{equation}\label{e4}
	|(II)| \leq C\sumi\|u_i\|_{\LQtaut{\Lam}}\|\Delta\phi\|_{\LQtaut{\Lam'}} \leq C_\eps + \eps\sumi\|u_i\|_{\LQtaut{\Lam+\vk}},
	\end{equation}
	\begin{equation}\label{e5}
		|(III)| \leq C\sumi \|u_i\|_{\LStaut{\Lam}}\|\pa_\eta\phi\|_{\LStaut{\Lam'}} \leq C_\eps + \eps\sumi \|u_i\|_{\LStaut{\Lam+\vk}},
	\end{equation}
	\begin{equation}\label{e6}
	\begin{aligned}
		|(IV)| \leq C\sumj\|\phi\|_{\LStaut{\Lam'}}\|v_j\|_{\LStaut{\Lam}}\leq C_\eps +\eps\sumj \|v_j\|_{\LStaut{\Lam+\vk}}.
	\end{aligned}
	\end{equation}
	For $(V)$ we estimate using \eqref{ff7}
	\begin{equation}\label{e7}
	\begin{aligned}
	|(V)| &\leq \frac{\delta_{\max}-\delta_{\min}}{\delta_{\max}+\delta_{\min}}\intMtautwo\abs{\sumj b_j\vat v_j}|\Delta_M\phi|\\
	&\leq  \frac{\delta_{\max}-\delta_{\min}}{\delta_{\max}+\delta_{\min}}\norm{\sumj b_j\vat v_j}_{\LStaut{\Lam+\vk}}\|\Delta_M\|_{\LStaut{(\Lam+\vk)'}}\\
	&\leq \frac{\delta_{\max}-\delta_{\min}}{\delta_{\max}+\delta_{\min}}C_{\mr,(\Lam+\vk)'}^M\norm{\sumj b_j\vat v_j}_{\LStaut{\Lam+\vk}}.
	\end{aligned}
	\end{equation}
	Using \eqref{e3}--\eqref{e7} into \eqref{d13}, it follows from duality, \eqref{quasi-uniform} and Lemma \ref{lem:improved_quasi_uniform}, that
	\begin{equation}\label{e9}
	\begin{aligned}
	&\norm{\sumj b_j\vat v_j}_{\LStaut{\Lam+\vk}}\\
	&\leq C_\eps + \eps \sbra{\sumi\bra{\|u_i\|_{\LQtaut{\Lam+\vk}} + \|u_i\|_{\LStaut{\Lam+\vk}}} + \sumj\|v_j\|_{\LStaut{\Lam+\vk}}}.
	\end{aligned}
	\end{equation}
	Combining \eqref{e9} with with \eqref{ff6} in Lemma \ref{Lp_cylinder} (choosing $p = \Lam+\vk$) we have
	\begin{equation}\label{ff8}
		\begin{aligned}
		&\sumi \bra{\|\vat u_i\|_{\LQtaut{\Lam+\vk}} + \|\vat u_i\|_{\LStaut{\Lam+\vk}}} +  \norm{\vat \sumj b_j v_j}_{\LStaut{\Lam+\vk}}\\
		&\leq C_\eps +\eps \sbra{\sumi\bra{\|u_i\|_{\LQtaut{\Lam+\vk}} + \|u_i\|_{\LStaut{\Lam+\vk}}} + \sumj\|v_j\|_{\LStaut{\Lam+\vk}}}.
		\end{aligned}
	\end{equation}
	Recall that $\vat\geq 0$ and $\vat|_{[\tau,\tau+1]}\equiv 1$, it follows from \eqref{ff8} that
	\begin{equation}\label{ff9}
	\begin{aligned}
		&\sumi \bra{\|u_i\|_{\LQtau{\Lam+\vk}} + \|u_i\|_{\LStau{\Lam+\vk}}} + \norm{\sumj b_jv_j}_{\LStau{\Lam+\vk}}\\
		&\leq C_\eps +C\eps \sbra{\sumi\bra{\|u_i\|_{\LQtaut{\Lam+\vk}} + \|u_i\|_{\LStaut{\Lam+\vk}}} + \norm{\sumj b_jv_j}_{\LStaut{\Lam+\vk}}}.
		\end{aligned}
	\end{equation}
	For $\tau \in \mathbb N$, we define 
	\begin{equation*}
		y_{\tau}:= \sumi\bra{\|u_i\|_{\LQtau{\Lam+\vk}} + \|u_i\|_{\LStau{\Lam+\vk}}} + \norm{\sumj b_jv_j}_{\LStau{\Lam+\vk}}.
	\end{equation*}
	Inequality \eqref{ff9} implies
	\begin{equation}\label{ff10}
		y_\tau \leq C + C\eps(y_\tau + y_{\tau+1}).
	\end{equation}
	Define $\mathscr N  = \{\tau \in \mathbb N: y_{\tau}\leq y_{\tau+1}\}$. Then for any $\tau\in \mathscr N$, by choosing $\eps$ sufficiently small, we obtain from \eqref{ff10}
	\begin{equation*}
		y_\tau \leq C,
	\end{equation*}
	where $C$ is independent of $\tau$. From Lemma \ref{elementary}, we have
	\begin{equation*}
		y_\tau \leq C \quad \text{ for all } \quad \tau\in \mathbb N,
	\end{equation*}
	which proves the desired estimate \eqref{desired1}.
\end{proof}

\begin{lemma}[\eqref{O}-\eqref{D}-\eqref{F0}-\eqref{F1}-\eqref{F2}-\eqref{eq:intsum1}-\eqref{eq:intsum2}-\eqref{pOpM}-\eqref{mM_general}-\eqref{quasi-uniform}]\label{lem2} Assume that $L<0$ or $L=K=0$ in \eqref{F2}. 
	For any $\tau\in \mathbb N$, $2\leq p$, any $k\in \{1,\ldots, m_2\}$, and any $\eps>0$, there exists a constant $C_\eps>0$ such that
	\begin{equation}\label{d6}
	\begin{aligned}
		\|\vat v_k\|_{\LStaut{p}} \leq C_\eps &+ C_\eps\sum_{j=1}^{k-1}\|\vat v_j\|_{\LStaut{p}}+ \eps\sumi\bra{\|u_i\|_{\LQtaut{p}} +  \|u_i\|_{\LStaut{p}}}\\
		&+ \eps\sumj\|v_j\|_{\LStaut{p}} .
	\end{aligned}
	\end{equation}
	Consequently, for any $\eps>0$, there exists $C_\eps>0$ such that 
	\begin{equation}\label{d6_1}
		\|\vat v_k\|_{\LStaut{p}} \leq C_\eps + \eps\sumi\bra{\|u_i\|_{\LQtaut{p}} +  \|u_i\|_{\LStaut{p}}} + \eps\sumj\|v_j\|_{\LStaut{p}} 
	\end{equation}
	for all $k=1,\ldots, m_2$.
\end{lemma}
\begin{proof}	
	Let $0\leq \psi \in \LStaut{p'}$ with $\|\psi\|_{\LStaut{p'}} = 1$, and $\phi$ be the solution to \eqref{dual_eq} with $T = \tau+2$. Thanks to the calculations of Lemma \ref{lem:dual1}, we can write
	\begin{equation}\label{d5}
	\begin{aligned}
		&\intMtautwo (\vat v_k)\psi\\
		&\leq \intMtautwo \phi \vat' v_k + \sumi \intQtautwo a_{(k+m_1)i}\phi \vat' u_i + \sumi \intQtautwo a_{(k+m_1)i}\phi \vat F_i(u)\\
		&\quad + \sumi \intQtautwo a_{(m+k_1)i}(\vat u_i)(d_i-1)\Delta \phi + \sumi \intMtautwo a_{(k+m_1)i}d_i(\vat u_i)\pa_{\eta}\phi\\
		&\quad - \sum_{j=1}^{k-1}\intMtautwo a_{(k+m_1)(k+j)}\phi \vat (\pa_t v_j - \delta_j\Delta_M v_j) + L_2\intMtautwo \phi \vat \sbra{\sumi u_i^{\mM} + \sumj v_j^{\mM} +1}\\
		&\quad =: (I) + (II) + (III) + (IV) + (V) + (VI) + (VII).
	\end{aligned}
	\end{equation}
	We estimate the terms on the right hand side of \eqref{d5} separately.
	\begin{itemize}
		\item {\bf Estimate $(I)$}. From Lemma \ref{dual_boundary}, there exists $q>p'$ such that
		\begin{equation*}
			\|\phi\|_{\LStaut{q}} \leq C\|\psi\|_{\LStaut{p'}} = C.
		\end{equation*}
		Therefore, by Young's inequality and $q' = \frac{q}{q-1}$ is the H\"older conjugate exponent of $q$, we have
		\begin{equation*}
			|(I)| \leq C\|\phi\|_{\LStaut{q}}\|v_k\|_{\LStaut{q'}} \leq C\|v_k\|_{\LStaut{q'}}.
		\end{equation*}
		Since $q>p' = \frac{p}{p-1}$, we have $q' < p$. Therefore, it follows from H\"older's inequality that
		\begin{align*}
			\|v_k\|_{\LStaut{q'}} &\leq \|v_k\|_{\LStaut{1}}^{\theta_0}\|v_k\|_{\LStaut{p}}^{1-\theta_0} \\
			&\leq C\|v_k\|_{\LStaut{p}}^{1-\theta_0}\leq C_\eps + \eps \|v_k\|_{\LStaut{p}},
		\end{align*}
		with $\theta_0\in (0,1)$ satisfying $\frac{1}{q'} = \frac{\theta_0}{1} + \frac{1-\theta_0}{p}$. Therefore,
		\begin{equation}\label{estI}
			|(I)| \leq C_\eps + \eps \|v_k\|_{\LStaut{p}}.
		\end{equation}
		\item {\bf Estimate $(II)$}. Similarly to the estimate of $(I)$, we can use Lemma \ref{dual_boundary} to estimate
		\begin{equation}\label{estII}
			|(II)| \leq C\sumi\|u_i\|_{\LQtaut{p}}^{1-\theta_1} \leq C_\eps + \eps\sumi\|u_i\|_{\LQtaut{p}}
		\end{equation}
		for some $\theta_1\in (0,1)$.
		\item {\bf Estimate $(III)$}. We use the condition \eqref{eq:intsum1} to find
		\begin{equation*}
			(III) \leq L_2 \intQtautwo \phi \vat \sbra{\sumi u_i^{\pO}+1}.
		\end{equation*}
		Looking at the estimate of $(B12)$ in Lemma \ref{lem:dual1}, and using $|\vat'| \leq C$, we have
		\begin{equation*}
			|(III)| \leq C\sumi \intQtautwo \phi u_i^{\pO} + C\|\phi\|_{\LQtaut{p'}} \leq C\sumi \intQtautwo \phi u_i^{\pO} + C.
		\end{equation*}
		Similarly to \eqref{c5}--\eqref{c6}, 
		\begin{equation*}
		\begin{aligned}
			\intQtautwo \phi u_i^{\pO} &\leq \|\phi\|_{\LQtaut{(p')^{\dag}}}\|u_i\|_{\LQtaut{(\pO-\beta)s}}^{\pO-\beta}\|u_i\|_{\LQtaut{p}}^{\beta}\\
			&\leq C\|u_i\|_{\LQtaut{p}}^{\beta} \leq C_\eps+ \eps\|u_i\|_{\LQtaut{p}}
		\end{aligned}
		\end{equation*}
		where $\beta$ and $s$ are in \eqref{Sigma}. Therefore,
		\begin{equation}\label{estIII}
			|(III)| \leq C_\eps + \eps\sumi \|u_i\|_{\LQtaut{p}}.
		\end{equation}
		\item {\bf Estimate $(IV)$}. Thanks to Lemma \ref{dual_boundary} we have for some $s>p'$,
		\begin{equation*}
			\|\Delta \phi\|_{\LQtaut{s}} \leq C\|\psi\|_{\LStaut{p'}} \leq C.
		\end{equation*}
		Therefore, for $s' = \frac{s}{s-1}$,
		\begin{equation*}
			|(IV)| \leq C\|\Delta \phi\|_{\LQtaut{s}}\sumi\|u_i\|_{\LQtaut{s'}} \leq C\sumi \| u_i\|_{\LQtaut{s'}}.
		\end{equation*}
		Since $s>p'$, $s' < p = \frac{p'}{p'-1}$. Therefore, by interpolation inequality
		\begin{align*}
			\|u_i\|_{\LQtaut{s'}} \leq \|u_i\|_{\LQtaut{1}}^{\theta_3}\|u_i\|_{\LQtaut{p}}^{1-\theta_3}\leq C\|u_i\|_{\LQtaut{p}}^{1-\theta_3} 
		\end{align*}
		with $\theta_3\in (0,1)$ satisfying $\frac{1}{s'} = \frac{\theta_3}{1} + \frac{1-\theta_3}{p}$. From that we obtain
		\begin{equation}\label{estIV}
			|(IV)| \leq C\sumi \|u_i\|_{\LQtaut{p-1+\pO}}^{1-\theta_3}\leq C_\eps + \eps\sumi \|u_i\|_{\LQtaut{p}}.
		\end{equation}
		\item {\bf Estimate $(V)$}. From \eqref{flux_estimate} in Lemma \ref{dual_boundary} we have for $\xi = \frac{p'}{n+1}$
		\begin{equation*}
			\|\pa_{\eta}\phi\|_{\LStaut{p'+\xi}} \leq C\|\psi\|_{\LStaut{p'}} \leq C.
		\end{equation*}
		Therefore, with $s = \frac{p'+\xi}{p'+\xi - 1} = \frac{p+\xi(p-1)}{p+(p-1)(\xi-1)} < p$, we can estimate
		\begin{equation}\label{estV}
		\begin{aligned}
		 |(V)| &\leq C\sumi \|u_i\|_{\LStaut{\frac{p'+\xi}{p'+\xi - 1}}}\|\pa_{\eta}\phi\|_{\LStaut{p'+\xi}}\\
			&\leq C\sumi \|u_i\|_{\LStaut{\frac{p'+\xi}{p'+\xi - 1}}}\\
			&\leq C\sumi\|u_i\|_{\LStaut{1}}^{\theta_4}\|u_i\|_{\LStaut{p}}^{1-\theta_4} \quad \bra{\text{ with } \frac{1}{\frac{p'+\xi}{p'+\xi - 1}} = \frac{\theta_4}{1} + \frac{1-\theta_4}{p}}\\
			&\leq C_\eps + \eps \sumi\|u_i\|_{\LStaut{p}}.
		\end{aligned}
		\end{equation}
		
		\item {\bf Estimate $(VI)$}. By integration by parts,
		\begin{equation*}
			(VI) = \sum_{j=1}^{k-1}\intMtautwo \sbra{a_{(k+m_1)(k+j)}v_j\phi \vat' + v_j\vat \bra{(\delta_j-\delta_k)\Delta_M\phi + \psi}}.
		\end{equation*}
		Similar to the estimate of $(I)$ above
		\begin{align*}
			\abs{\sum_{j=1}^{k-1}\intMtautwo a_{(k+m_1)(k+j)}v_j\phi \vat'} &\leq C\sum_{j=1}^{k-1}\|v_j\|_{\LStaut{p}}^{1-\theta_5}\\
			&\leq C_\eps + \eps \sumj \|v_j\|_{\LStaut{p}}
		\end{align*}
		for some $\theta_5\in (0,1)$. By H\"older's inequality and $$\|\Delta_{M} \phi\|_{\LStaut{p'}} \leq C\|\psi\|_{\LStaut{p'}} \leq C,$$
		we get
		\begin{equation*}
		\begin{aligned}
			&\abs{\sum_{j=1}^{k-1}\intMtautwo v_j\vat \sbra{(\delta_j-\delta_k)\Delta_M\phi + \psi}}\\
			&\leq C\sum_{j=1}^{k-1}\|\vat v_j\|_{\LStaut{p}}\bra{\|\Delta_M\phi\|_{\LStaut{p'}} + \|\psi\|_{\LStaut{p'}}}\\
			&\leq C\sum_{j=1}^{k-1}\|\vat v_j\|_{\LStaut{p}}.
		\end{aligned}
		\end{equation*}
		Therefore, we have
		\begin{equation}\label{estVI}
			|(VI)|\leq C_\eps + \eps\sumj\|v_j\|_{\LStaut{p}}+ C\sum_{j=1}^{k-1}\|\vat v_j\|_{\LStaut{p}}.
		\end{equation}
		
		\item {\bf Estimate $(VII)$}. We use similar estimates to that of $(B3)$ in \eqref{B33}--\eqref{B3}. More precisely, with $\alpha$ and $r$ are in \eqref{ff11}--\eqref{ff12} we have 
		\begin{equation*}
			\|u_i\|_{\LStaut{(\mM-\alpha)r}}^{\mM -\alpha} \leq C\|u_i\|_{\LStaut{\Lam+\vk}}^{\mM - \alpha} \leq C.
		\end{equation*}
		From Lemma \ref{dual_boundary},
		\begin{equation*}
			\|\phi\|_{\LStaut{(p')^*}} \leq C\|\psi\|_{\LStaut{p'}} \leq C.
		\end{equation*}
		Therefore
		\begin{align*}
			\sumi\intMtautwo \phi u_i^{\mM} &\leq \sumi\|\phi\|_{\LStaut{(p')^*}}\|u_i\|_{\LStaut{(\mM-\alpha)r}}^{\mM-\alpha}\|u_i\|_{\LStaut{p}}^{\alpha}\\
			&\leq C\sumi\|u_i\|_{\LStaut{p}}^{\alpha}\\
			&\leq C_\eps + \eps\sumi \|u_i\|_{\LStaut{p}}.
		\end{align*}
		Similarly,
		\begin{equation*}
			\sumi \intMtautwo \phi v_j^{\mM} \leq C\sumj\|v_j\|_{\LStaut{p}}^{\alpha} \leq C_\eps + \eps\sumj\|v_j\|_{\LStaut{p}}.
		\end{equation*}
		Finally,
		\begin{equation*}
			L_2\intMtautwo \phi \vat \leq C\|\phi\|_{\LStaut{p'}} \leq C.
		\end{equation*}
		Therefore,
		\begin{equation}\label{estVII}
			|(VII)| \le C_\eps + \eps\sumi \|u_i\|_{\LStaut{p}} + \eps\sumj\|v_j\|_{\LStaut{p}}.
		\end{equation}
	\end{itemize}
	Applying all the estimates of $(I)$ to $(VII)$ in \eqref{estI}, \eqref{estII}, \eqref{estIII}, \eqref{estIV}, \eqref{estV}, \eqref{estVI}, \eqref{estVII} into \eqref{d5} we obtain
	\begin{equation*}
	\begin{aligned}
		\intMtautwo (\vat v_k)\psi&\leq C_\eps + C_\eps\sum_{j=1}^{k-1}\|\vat v_j\|_{\LStaut{p}}\\ &+\eps\sumi\bra{\|u_i\|_{\LQtaut{p}} + \sumi \|u_i\|_{\LStaut{p}}} +  \eps\sumj\| v_j\|_{\LStaut{p}}.
	\end{aligned}
	\end{equation*}
	From this we get the estimate \eqref{d6} due to duality. Finally \eqref{d6_1} follows from \eqref{d6} by induction.
\end{proof}

We are now ready to show the uniform-in-time bound in Theorem \ref{thm2}.
\begin{theorem}[\eqref{O}-\eqref{D}-\eqref{F0}-\eqref{F1}-\eqref{F2}-\eqref{F3}-\eqref{eq:intsum1}-\eqref{eq:intsum2}-\eqref{pOpM}-\eqref{mM_general}-\eqref{quasi-uniform}] Assume that $L<0$ or $L=K=0$ in \eqref{F2}. \label{thm3}
	The global solution to \eqref{eq:mainsys} is bounded uniformly in time, i.e.
	\begin{equation*}
	\sup_{i=1,\ldots, m_1}\sup_{j=1,\ldots, m_2}\sup_{t\geq 0}\sbra{\|u_i(t)\|_{L^\infty(\Omega)} + \|v_j(t)\|_{L^\infty(M)}} <+\infty.
	\end{equation*}
\end{theorem}
\begin{proof}
	We claim that, for any $2\leq p$, there exists a constant $C_p>0$ such that
	\begin{equation}\label{claim}
		\sumi \bra{\|u_i\|_{\LQtau{p}} + \|u_i\|_{\LStau{p}}} + \sumj\|v_j\|_{\LStau{p}} \leq C_p \quad \text{ for all } \quad \tau\in \mathbb N.
	\end{equation}
	Indeed, from \eqref{ff6} in Lemma \ref{Lp_cylinder} and \eqref{d6_1} in Lemma \ref{lem2}, we get for any $\eps>0$ a constant $C_\eps>0$ such that
	\begin{align*}
		&\sumi\bra{\|\vat u_i\|_{\LQtaut{p}} + \|\vat u_i\|_{\LStaut{p}}} + \sumj\|\vat v_j\|_{\LStaut{p}} \\
		&\leq C_\eps + \eps\sbra{\sumi\bra{\|u_i\|_{\LQtaut{p}} + \|u_i\|_{\LStaut{p}}} + \sumj\|v_j\|_{\LStaut{p}}}.
	\end{align*}
	Using the same arguments as at the end of the proof of Lemma \ref{lem4}, we obtain \eqref{claim}.
	
	\medskip
	Now we can use \eqref{claim} in the truncated system \eqref{shifted_sys}, with $p$ is large enough, to conclude that there exists $C_\infty>0$ independent of $\tau\in \mathbb N$ such that
	\begin{equation*}
		\sumi\|u_i\|_{\LQtau{\infty}} + \sumj\|v_j\|_{\LStau{\infty}} \leq C_\infty \quad \text{ for all } \quad \tau\in \mathbb N,
	\end{equation*}
	which finishes the proof of Theorem \ref{thm3}.
\end{proof}

\subsection{Proof of Theorem \ref{main:thm0}}\label{sec:4.3}
\begin{proof}[Proof of Theorem \ref{main:thm0}]
	Thanks to Theorem \ref{thm:main1}, it's sufficient to show that \eqref{quasi-uniform} always holds for $\Lam = 2$. Indeed, in this case $\Lam' = 2$. We only need to show that 
	\begin{equation}\label{claim3}
	C_{\mr,2}^M \leq 1.
	\end{equation}
	To do that, we multiply the equation \eqref{parabolic_M} by $-\Delta_M\uu$ in $L^2(\Omega)$, to get
	\begin{equation*}
	\frac{1}{2}\frac{d}{dt}\|\na_M\uu\|_{\LM{2}}^2 + \|\Delta_M \uu\|_{L^2(\Omega)}^2 = -\intM \ff \Delta_M\uu \leq \frac 12\|\ff\|_{\LM{2}}^2 + \frac 12\|\Delta_M \uu\|_{\LM{2}}^2,
	\end{equation*}
	implying
	\begin{equation*}
	\frac{d}{dt}\|\na_M\uu\|_{\LM{2}}^2 + \|\Delta_M \uu\|_{\LM{2}}^2 \leq \|\ff\|_{\LM{2}}.
	\end{equation*}
	Integrating this on $(0,T)$ and using $\uu(\cdot,0) = 0$ we obtain 
	\begin{equation*}
	\|\na_M\uu(\cdot,T)\|_{\LM{2}}^2 + \|\Delta_M\uu\|_{\LS{2}}^2 \leq \|\ff\|_{\LS{2}}^2,
	\end{equation*}
	which implies
	\begin{equation*}
	\|\Delta_M\uu\|_{\LS{2}} \leq \|\ff\|_{\LS{2}},
	\end{equation*}
	and hence, the desired estimate \eqref{claim3}.
\end{proof}

\subsection{Proof of Theorem \ref{main:thm2}}\label{sec:4.4}
\begin{proof}[Proof of Theorem \ref{main:thm2}]
	First, thanks to Lemmas \ref{prepare} and \ref{lemma:Loft}, we have, similarly to \eqref{eps_critical} and \eqref{ff4}, for any $p>1$ and $\eps>0$ a constant $K_{p,\eps>0}$ such that
	\begin{equation*}
		\sumi\bra{\|u_i\|_{\LQ{p}} + \|u_i\|_{\LS{p}}} \leq K_{p,\eps}(1+T) + \eps\sumj\|v_j\|_{\LS{p}}.
	\end{equation*}
	Following Lemma \ref{lem:dual1}, we will show that for any $\eps>0$, and any $k\in \{1,\ldots, m_2\}$, there exists $C_{T,\eps}>0$ such that
	\begin{equation}\label{gg4}
		\|v_k\|_{\LS{p}} \leq C_{T,\eps} + C_T\sum_{j=1}^{k-1}\|v_j\|_{\LS{p}} + \eps\sumj \|v_j\|_{\LS{p}}.
	\end{equation}
	We consider the duality equation \eqref{dual_eq} with $0\leq \psi \in \LS{p'}$ satisfying $\|\psi\|_{\LS{p'}} = 1$. The same integration by parts arguments in Lemma \ref{lem:dual1} gives
	\begin{equation}\label{gg5}
	\begin{aligned}
	\intMT v_k\psi &\leq \intM v_{k,0}\phi(0) + \sumi a_{(k+m_1)i}\intO u_{i,0}\phi(0) + L_2\intQT \phi\sbra{\sumi u_i^{\pO}+1}\\
	&\quad + \sumi a_{(k+m_1)i}d_i\intMT u_i\pa_{\eta}\phi + \sumi a_{(k+m_1)i}\intQT u_i\sbra{\pa_t\phi + d_i\Delta \phi}\\
	&\quad +\sum_{j=1}^{k-1}a_{(k+m_1)(k+j)}\intM v_{j,0}\phi(0) - \sum_{j=1}^{k-1}\intMT v_j\sbra{\pa_t\phi + \delta_j\Delta_M\phi}\\
	&\quad + L_2\intMT \phi\sbra{\sumi u_i^{\mM} + \sumj v_j^{\mM}+1}.
	\end{aligned}
	\end{equation}
	All the terms on the right hand side of \eqref{gg5} can be estimated similarly to Lemma \ref{lem:dual1} except for two sums
	\begin{equation*}
		(\mathsf A) = \sumi\intQT \phi u_i^{\pO} \quad \text{ and } \quad (\mathsf B) = \sumi\intMT \phi u_i^{\mM} + \sumj\intMT \phi v_j^{\mM}.
	\end{equation*}
	To estimate ($\mathsf A$) we first use Lemma \ref{dual_boundary} to have
	\begin{equation*}
		\|\phi\|_{\LQ{(p')^\dag}} \leq C_T\|\psi\|_{\LQ{p'}} = C_T.
	\end{equation*} 
	By H\"older's inequality
	\begin{align}\label{gg6}
		\int_0^T\intO \phi u_i^{\pO - \alpha} u_i^{\alpha} \leq \|\phi\|_{\LQ{(p')^{\dag}}}\|u_i\|_{\LQ{(\pO-\alpha)s}}^{\pO - \alpha}\|u_i\|_{\LQ{p}}^{\alpha} \leq C_T\|u_i\|_{\LQ{(\pO-\alpha)s}}^{\pO - \alpha}\|u_i\|_{\LQ{p}}^{\alpha}
	\end{align}
	with
	\begin{equation*}
		\frac{1}{(p')^{\dag}} + \frac 1s + \frac{\alpha}{p}=1.
	\end{equation*}
	From this
	\begin{align*}
		\frac 1s = 1 - \frac \alpha p - \frac{1}{(p')^{\dag}}
		 = 1 - \frac \alpha p - \frac{n+1-2p'}{(n+2)p'}
		 = \frac{n(1-\alpha) + 3p + 1 - 2\alpha}{(n+2)p}
	\end{align*}
	or equivalently
	\begin{equation*}
		s = \frac{(n+2)p}{n(1-\alpha) + 3p + 1 - 2\alpha}.
	\end{equation*}
	Note that as $p\to\infty$, $s \to \frac{n+2}{3}$. Since 
	\begin{equation*}
		\pO < 1 + \frac{3a}{n+2}
	\end{equation*}
	we can choose $\alpha \in (0,1)$ and $p$ large enough such that
	\begin{equation*}
		(\pO - \alpha)s \leq \frac{3a}{n+2}\frac{n+2}{3} = a.
	\end{equation*}
	We then use this and the assumption $\|u_i\|_{L^\infty(0,T;\LO{a})} \leq \F(T)$ in \eqref{gg6} to get
	\begin{equation*}
		\intQT \phi u_i^{\pO} \leq C_T\F(T)^{\pO-\alpha}\|u_i\|_{\LQ{p}}^{\alpha} \leq C_{T,\eps} + \eps\|u_i\|_{\LQ{p}},
	\end{equation*}
	which gives the estimate 
	\begin{equation*}
		(\mathsf A) \leq C_{T,\eps} + \eps\sumi\|u_i\|_{\LQ{p}}.
	\end{equation*}
	The estimate of ($\mathsf B$) can be done in the same way as for $(B3)$ in the proof of Lemma \ref{lem:dual1}, where the integrability $b$ plays the role of $\Lam+\vk$ therein, so we omit it here. Ultimately, we have for any $\eps>0$ a constant $C_{T,\eps}>0$ such that
	\begin{equation*}
		(\mathsf B) \leq C_{T,\eps} + \sumi \|u_i\|_{\LS{p}} + \sumj\|v_j\|_{\LS{p}}.
	\end{equation*}
	With the estimates of $(\mathsf A)$ and $(\mathsf B)$, one can proceed from \eqref{gg5} the same way in Lemma \ref{lem:dual1} to finally obtain \eqref{gg4}. Lemma \ref{elementary} then gives
	\begin{equation*}
		\sumj \|v_j\|_{\LS{p}} \leq C_T
	\end{equation*}
	and consequently
	\begin{equation*}
		\sumi\bra{\|u_i\|_{\LQ{p}} + \|v_j\|_{\LS{p}}} \leq C_T
	\end{equation*}
	for all $p\geq 1$. This is enough to conclude the global existence of solutions to \eqref{eq:mainsys}. The uniform-in-time bound is obtained by using the truncated function $\varphi_\tau$ and working on each cylinder $Q_{\tau,\tau+2}$ for $\tau\in \mathbb N$. We omit the details since they are very similar to that of Section \ref{sec:uniform-in-time}.
\end{proof}

\section{Applications}\label{sec:application}
In this section, we show the application of our theorems to some models arising from cell biology. It's worth emphasizing that, all known results for volume-surface systems seem not applicable to these systems.

\subsection{Membrane protein clustering}
In a recent work of Lucas M. Stolerman et al \cite{stolerman2019stability}, the authors focus on stability analysis of a bulk-surface reaction-diffusion model of membrane protein clustering. In their model, $U$ represents a volume component which diffuses in the cytoplasm. It binds with the membrane and forms a surface monomer $A_1$ via a reaction flux (see below). Then subsequent oligomerization at the membrane is given by $A_{j-1}+A_1 \rightleftharpoons A_j$ for $j=2,3,...,N$. If we model the cytoplasm as a smooth bounded region $\Omega \subset \R^n$, and the membrane (boundary of $\Omega$) by $M$, then the model in \cite{stolerman2019stability} has the form
\begin{equation}
\begin{cases}
\frac{\partial u}{\partial t}=d\Delta u, & \text{ on }\Omega\times(0,T)\\
d\frac{\partial u}{\partial\eta}=F(u,v_1,...,v_N), & \text{ on }M\times(0,T)\\
\frac{\partial v_1}{\partial t}=\delta_{1}\Delta_M v_{1}+G_{1}(u,v_1,...,v_N), & \text{ on }M\times(0,T)\\
\vdots&\\
\frac{\partial v_N}{\partial t}=\delta_{N}\Delta_M v_{N}+G_{N}(u,v_1,...,v_N), & \text{ on }M\times(0,T)\\
u=u_0 & \text{ on }\overline{\Omega}\times{0}\\
v=v_0 & \text{ on }M\times{0}
\end{cases}\label{eq:example1}
\end{equation}
Here, $d$ and $\delta_i$ are positive diffusion coefficients, for $i=1,...,N$, $u$ represents the cytoplasm concentration density of $U$, and $v_1,...,v_N$ represent the membrane concentration densities of $A_1,...,A_N$, $u_0$ and the vector $v_0=(v_{0_i})$ represent bounded nonnegative initial data, and $k_0$, $k_b$, $k_d$ and $k_2,...,k_N$ are positive constants. Also, 
\begin{align*}
F(u,v_1,...,v_N)&=-(K_0+k_bv_N)u+k_dv_1\\
G_1(u,v_1,...,v_N) & =(K_0+k_bv_N)u-k_dv_1-2k_mv_1^2+2k_2v_2
-k_gv_1\sum_{l=1}^{N-1}v_l+\sum_{j=3}^Nk_jv_j\\
G_2(u,v_1,...,v_N)&=k_mv_1^2-k_gv_1v_2-k2v_2+k_3v_3\\
G_j(u,v_1,...,v_N)&=k_gv_1v_{j-1}-k_gv_1v_j-k_jv_j+k_{j+1}v_{j+1},\text{ for }j=3,...,N-1\\
G_N(u,v_1,...,v_N)&=k_gv_1v_{N-1}-k_Nv_N
\end{align*}
Note that the nonlinearities $F$ and $G_j$, $j=1,\ldots, N$, are locally Lipschitz continuous,
\begin{equation}\label{eq:qpf}
F(0,v_1,...,v_N)\ge 0\text{ for all }(v_1,...,v_N) \in \R_+^N
\end{equation} 
and
\begin{equation}\label{eq:qpG}
G_i(u,v_1,...,v_N)\ge 0\text { for all }u,v_1,...,v_N\ge 0\text{ with }v_i=0\text{ for }i=1,...,N,
\end{equation}
and they are polynomial. That means the assumptions \eqref{F0}, \eqref{F1}, \eqref{F3} are fulfilled. 
It's also simple to check that
\begin{equation}\label{eq:minibal}
F(u,v_1,...,v_N)+\sum_{j=1}^NjG_j(u,v_1,...,v_N)=0,\text{ for all }u,v_1,...,v_N\ge 0,
\end{equation}
and thus \eqref{F2} is satisfied with $L = K = 0$. It remains to check the intermediate sum conditions \eqref{eq:intsum1} and \eqref{eq:intsum2}, i.e. there is $(N+1)\times(N+1)$ lower triangular matrix $A=(a_{i,j})$ with positive entries on the diagonal, and nonnegative entries below the diagonal, and constants $K_1,K_2\ge 0$ so that
\begin{equation}\label{eq:miniintsum}
A\begin{bmatrix}F(u,v_1,...,v_N)\\G_1(u,v_1,...,v_N)\\\vdots\\G_N(u,v_1,...,v_N)\end{bmatrix}\le K_1\vec{1}\left(u+\sum_{j=1}^n v_j+K_2\right),\text{ for all }u,v_1,...,v_N\ge 0.
\end{equation} 
Indeed, this is done by choosing $K_1=2\max_{{i=2,...,m}}k_i$, $K_2=0$,  and
\[A=\begin{bmatrix}
1&0&\cdots&0\\
1&1&\cdots&0\\
\vdots&\vdots&\ddots&\vdots\\
1&\cdots&1&1
\end{bmatrix}.
\]
Since all the assumptions in Theorem \ref{main:thm0} are satisfied, we have the following
\begin{theorem}\label{application1}
	For any non-negative initial data $(u_0,v_0)\in W^{2-2/p}(\Omega)\times (W^{2-2/p}(M))^N$ for $p>n$ satisfying the compatibility condition
	\begin{equation*}
		d\pa_{\eta}u_0 = F(u_0, v_{1,0}, \ldots, v_{N,0}) \quad \text{ on }\quad M,
	\end{equation*}
	there exists a unique global classical solution to \eqref{eq:example1} which is uniformly bounded in time, i.e.
	\begin{equation*}
	\sup_{t\geq 0}\bra{\|u(t)\|_{\LO{\infty}} + \sum_{j=1}^N\|v_j(t)\|_{\LM{\infty}}} < +\infty.
	\end{equation*}
\end{theorem}

\subsection{Activation of Cdc42 in cell polarization}
This second example is from the recent paper \cite{borgqvist2020cell} where the authors derive a mathematical model for the activation of Cdc42 in cell polarization. The system reads as\footnote{The model considered herein is slightly different from that of \cite{borgqvist2020cell}, where $Q(A,I,G) = k_1G(k_{\max}-(A+I))$, since we take into account the saturation. It's noted that their choice of nonlinearity might lead to {\it negative} concentrations, for instance, when $I_0 \equiv 0$, $\beta \geq A_0(x) > k_{\max}+1$ for $x\in\Omega$, and $G_0(x)\geq \alpha>0$ for a large enough constant $\alpha$.}
\begin{equation}\label{eq:example2}
\begin{cases}
	\pa_t G = D_G\Delta G, &(x,t)\in Q_T,\\
	D_G\pa_\eta G = Q(A,I,G), &(x,t)\in M_T,\\
	\pa_tA = D_A\Delta_MA + F(A,I), &(x,t)\in M_T,\\
	\pa_t I = D_I\Delta_M I + H(A,I,G), &(x,t)\in M_T,\\
	G(x,0) = G_0(x), &x\in\Omega,\\
	A(x,0) = A_0(x), \; I(x,0) = I_0(x), &x\in M,
\end{cases}
\end{equation}
where
\begin{equation*}
	Q(A,I,G) = -k_1G(k_{\max} - (A+I))_+ + k_{-1}I, F(A,I) = k_2I - k_{-2}A + k_3A^2I
\end{equation*}
and
\begin{equation*}
	H(A,I,G) = -F(A,I)-Q(A,I,G),
\end{equation*}
with positive constant rates $k_1, k_{-1}, k_2, k_{-2}, k_{\max}$. Here $G$, $A$, $I$ are concentrations of the GTP-, GDP-, and GDI-bound form respectively. The interested reader is referred to \cite{borgqvist2020cell} for more details. It's easy to check that all the assumptions \eqref{F0}, \eqref{F1}, \eqref{F2} (with $L=K=0$) and \eqref{F3} are fulfilled. It's also clear that
\begin{equation*}
\begin{bmatrix}
	1 & 0 & 0\\
	1 & 1 & 0\\
	1 & 1 & 1
\end{bmatrix}
\begin{pmatrix} Q\\ H \\ F \end{pmatrix} \le \begin{pmatrix} k_{-1}I\\ k_{-2}A\\ 0\end{pmatrix},
\end{equation*}
which means that the intermediate sum condition \eqref{eq:intsum2} is satisfied (condition \eqref{eq:intsum1} is trivially fulfilled since we have no nonlinearities for the equation of $G$). Therefore Theorem \ref{main:thm0} applies and we have the following result.
\begin{theorem}
	For any non-negative initial data $(G_0, A_0, I_0)\in W^{2-2/p}(\Omega)\times(W^{2-2/p}(M))^{2}$ for $p>n$ satisfying the compatibility condition
	\begin{equation*}
		D_G\pa_{\eta}G_0 = Q(A_0,I_0,G_0) \quad \text{ on } \quad M,
	\end{equation*}
	there exists a unique global classical solution to \eqref{eq:example2} which is bounded uniformly in time,
	\begin{equation*}
		\sup_{t\ge 0}\bra{\|G(t)\|_{\LO{\infty}} + \|A(t)\|_{\LM{\infty}} + \|I(t)\|_{\LM{\infty}}} < +\infty.
	\end{equation*}
\end{theorem}

\subsection{A system with better a-priori estimates}\label{Exp3}
In this section we consider a system where better a-priori estimates can be derived using the system's special structures, which in turn allows us to obtain global existence for higher order nonlinearities. More precisely, the system reads
\begin{equation}\label{gg1}
\begin{cases}
	\pa_t u_1 - d_1\Delta u_1 = f_1(u) = u_1^3 - u_2^3, &x\in\Omega,\\
	\pa_t u_2 - d_2\Delta u_2 = f_2(u) = -u_1^3 + u_2^3, &x\in\Omega,\\
	d_1\pa_{\eta}u_1 = g_1(u,v) = u_1 - u_2 - 2u_1^3 + u_2^2 - v_2^2, &x\in M,\\
	d_2\pa_{\eta}u_2 = g_2(u,v) = -u_1 + u_2 - u_2^3 - v_1^2, &x\in M,\\
	\pa_t v_1 - \delta_1\Delta_M v_1 = h_1(u,v) = 2u_1^3 + u_1u_2^2 - v_2^6, &x\in M,\\
	\pa_t v_2 - \delta_2 \Delta_M v_2 = h_2(u,v) = 2u_2^3 + u_1^3 - v_1^6, &x\in M,\\
	u(x,0) = u_0(x), &x\in\Omega,\\
	v(x,0) = v_0(x), w(x,0) = w_0(x), &x\in M.
\end{cases}
\end{equation}
It's obvious that
\begin{equation*}
	f_1(u) + f_2(u) \leq 0.
\end{equation*}
By using the Young's inequality
\begin{equation}\label{gg2}
	u_1u_2^2 \leq \frac{u_1^3}{3} + \frac{2u_2^3}{3}
\end{equation}
we can check that
\begin{equation*}
	2g_1(u,v) + 2g_2(u,v) + h_1(u,v) + h_2(u,v) \leq -u_1^3 + u_1u_2^2 - 2u_3^3 \leq -\frac{2}{3}u_1^3 - \frac 43 u_2^3 \leq 0.
\end{equation*}
Therefore, assumption \eqref{F2} is satisfied. It's easy to check that by choosing the matrix
\begin{equation*}
A = \begin{pmatrix}
1&0&0&0\\
1&1&0&0\\
1&1&1&0\\
1&1&1&1
\end{pmatrix},
\end{equation*}
the assumptions \eqref{eq:intsum1} and \eqref{eq:intsum2} are satisfied with
\begin{equation*}
	\pO = 3, \quad \pM = 2, \quad \text{ and } \quad \mM = 3.
\end{equation*}
Therefore, the results of Theorem \ref{main:thm0} are not applicable to obtain global existence to \eqref{gg1}. We show here that by utilizing the special structure of \eqref{gg1}, one can use Theorem \ref{main:thm2} to still get global solutions. Indeed, direct computations give
\begin{equation}\label{gg3}
\begin{aligned}
\pa_t\bra{\intO \bra{u_1^4+u_2^4} + \intM (v_1 +v_2)} + 12d_1\|u_1\na u_1\|_{\LO{2}}^2 + 12d_2\|u_2\na u_2\|_{\LO{2}}^2\\
= 4\intO (f_1(u)u_1^3 + f_2(u)u_2^3) + 4\intM \bra{g_1(u,v)u_1^3 + g_2(u,v)u_2^3} + \intM\bra{h_1(u,v) + h_2(u,v)}.
\end{aligned}
\end{equation}
We have
\begin{equation*}
	f_1(u)u_1^3 + f_2(u)u_2^3 = -(u_1^3 - u_2^3)^2 \leq 0,
\end{equation*}
and
\begin{align*}
	&4\bra{g_1(u,v)u_1^3 + g_2(u,v)u_2^3}+h_1(u,v) + h_2(u,v)\\
	&\leq 4u_1^4 -8u_1^6 + 4u_1^3u_2^2 + 4u_2^4 - 4u_2^6 + 2u_1^3 + u_1u_2^2  - v_2^6 + 2u_2^3 + u_1^3 - v_1^6 \\
	&\leq -C(u_1^6 + u_2^6+v_1^6 + v_2^6)+ C
\end{align*}
where we used \eqref{gg2} and 
\begin{equation*}
	4u_1^4 + 4u_1^3u_2^2 + 4u_2^4 + 2u_1^3 + u_1u_2^2 + 2u_2^3 + u_1^3 \leq \eps\bra{u_1^6 + u_2^6} + C_\eps
\end{equation*}
for any $\eps>0$. Thus we get from \eqref{gg3} that
\begin{align*}
	\pa_t\bra{\intO (u_1^4 + u_2^4) + \intM\bra{v_1 + v_2}}+ C\intM\bra{u_1^6 + u_2^6 + v_1^6 + v_2^6} \leq C.
\end{align*}
By integrating in time we get
\begin{equation*}
	\sup_{i=1,2}\|u_i\|_{L^\infty(0,T;\LO{4})}^4 + \sup_{i=1,2}\bra{\|u_i\|_{\LS{6}}^6 + \|v_i\|_{\LS{6}}^6} \leq CT.
\end{equation*}
Thus, Theorem \ref{main:thm2} is applicable with $a = 4$ and $b = 6$, which allows us to get global existence of \eqref{gg2} in the physical dimension $n= 3$. The uniform-in-time bound of solutions remains unclear.
\begin{theorem}\label{thm4}
	Let $n\leq 3$. For any non-negative initial data $(u_0,v_0)\in W^{2-2/p}(\Omega)^2\times W^{2-2/p}(M)^2$ satisfying the compatibility condition
	\begin{equation*}
		d_1\pa_{\eta}u_{1,0} = g_1(u_0,v_0), \quad d_2\pa_{\eta}u_{2,0} = g_2(u_0,v_0), \quad \text{ on } \quad M,
	\end{equation*}
	there exists a unique global classical solution to \eqref{gg2}.
\end{theorem}

\appendix
\section{Technical lemmas}\label{appendix}
In this appendix, we prove two technical lemmas that are used in Lemma \ref{lemma:Loft}, more precisely the time derivative of the function $\H_p[u]$ and the integration by parts in \eqref{f3}--\eqref{a_ij}.
\begin{lemma}\label{Hp-lem7}
	Suppose $m_1\in \mathbb N$, $\theta= (\theta_1,\ldots, \theta_{m_1})$, where $\theta_1,...,\theta_{m_1}$ are positive real numbers, $\beta\in \mathbb Z_+^{m_1}$, and $\H_p[u]$ is defined in (\ref{Hp}). Then
	$$\frac{\partial}{\partial t}\H_0[u](t)=0,\text{ }\frac{\partial}{\partial t}\H_1[u](t)=\sum_{j=1}^{m_1}\theta_j\frac{\partial}{\partial t}u_j(t),$$  
	and for $p\in\mathbb N$ such that $p\ge 2$, 
	\begin{equation*}
	\frac{\partial}{\partial t}\H_p[u](t) = \sum_{|\beta| = p-1}\begin{pmatrix} p\\ \beta \end{pmatrix} \theta^{\beta^2}u(t)^{\beta}\sum_{j=1}^{m_1}\theta_j^{2\beta_j+1}\frac{\partial}{\partial t}u_j(t).
	\end{equation*}
\end{lemma}

\begin{proof}
	The results for $\H_0[u](t)$ and $\H_1[u](t)$ are trivial. The same is true for the case when $m_1=1$. Suppose $p\ge 2$ and $m_1\ge 2$. We proceed by induction on the value $m_1$, and assume $k\in \mathbb N$ such that the result is is true for $m_1=k$. Suppose $m_1=k+1$ and denote
	$$\tilde\beta=(\beta_2,...,\beta_{m_1})\text{ and }\tilde u=(u_2,...,u_{m_1}).$$
	Then we can rewrite $\H_p[u]$ as
	\begin{equation}\label{Hp-eq1}
	\H_p[u] = \sum_{\beta_1=0}^p\frac{1}{\beta_1!}\theta_1^{\beta_1^2}u_1^{\beta_1}\sum_{|\tilde\beta|=p-\beta_1}\begin{pmatrix}
	p\\ \tilde\beta\end{pmatrix}\tilde\theta^{\tilde\beta^2}\tilde u^{\tilde\beta}.
	\end{equation}
	Consequently,
	\begin{align}\label{Hp-eq2}
	\frac{\partial}{\partial t}\H_p[u] &= \sum_{\beta_1=1}^p\frac{1}{\beta_1!}\theta_1^{\beta_1^2}\beta_1u_1^{\beta_1-1}\frac{\partial}{\partial t}u_1\sum_{|\tilde\beta|=p-\beta_1}\begin{pmatrix}
	p\\ \tilde\beta\end{pmatrix}\tilde\theta^{\tilde\beta^2}\tilde u^{\tilde\beta}\nonumber\\
	&+\sum_{\beta_1=0}^p\frac{1}{\beta_1!}\theta_1^{\beta_1^2}u_1^{\beta_1}\frac{\partial}{\partial t}\left(\sum_{|\tilde\beta|=p-\beta_1}\begin{pmatrix}
	p\\ \tilde\beta\end{pmatrix}\tilde\theta^{\tilde\beta^2}\tilde u^{\tilde\beta}\right)\nonumber\\
	&=\sum_{\beta_1=1}^p\frac{1}{\beta_1!}\theta_1^{\beta_1^2}\beta_1u_1^{\beta_1-1}\frac{\partial}{\partial t}u_1\sum_{|\tilde\beta|=p-\beta_1}\begin{pmatrix}
	p\\ \tilde\beta\end{pmatrix}\tilde\theta^{\tilde\beta^2}\tilde u^{\tilde\beta}\nonumber\\
	&+\sum_{\beta_1=0}^{p-1}\frac{1}{\beta_1!}\theta_1^{\beta_1^2}u_1^{\beta_1}\frac{p!}{(p-\beta_1)!}\frac{\partial}{\partial t}H_{p-\beta_1}[\tilde u].
	\end{align}
	Now, from our induction hypothesis, 
	\begin{equation}\label{Hp-eq3}
	\frac{\partial}{\partial t}\H_{p-\beta_1}[\tilde u] = \sum_{|\tilde\beta| = p-\beta_1-1}\begin{pmatrix} p-\beta_1\\ \tilde\beta \end{pmatrix} \tilde\theta^{\tilde\beta^2}\tilde u^{\tilde\beta}\sum_{j=1}^{m_1-1}\tilde\theta_j^{2\tilde\beta_j+1}\frac{\partial}{\partial t}\tilde u_j.
	\end{equation}
	Therefore, substituting (\ref{Hp-eq3}) into (\ref{Hp-eq2}), and noting that $\tilde u_j=u_{j+1}$ and $\tilde\theta_j=\theta_{j+1}$, gives 
	\begin{align*}\label{Hp-eq4}
	\frac{\partial}{\partial t}\H_p[u] &=\sum_{\beta_1=1}^p\frac{1}{\beta_1!}\theta_1^{\beta_1^2}\beta_1u_1^{\beta_1-1}\frac{\partial}{\partial t}u_1\sum_{|\tilde\beta|=p-\beta_1}\begin{pmatrix}
	p\\ \tilde\beta\end{pmatrix}\tilde\theta^{\tilde\beta^2}\tilde u^{\tilde\beta}\nonumber\\
	&+\sum_{\beta_1=0}^{p-1}\frac{1}{\beta_1!}\theta_1^{\beta_1^2}u_1^{\beta_1}\frac{p!}{(p-\beta_1)!}\sum_{|\tilde\beta| = p-\beta_1-1}\begin{pmatrix} p-\beta_1\\ \tilde\beta \end{pmatrix} \tilde\theta^{\tilde\beta^2}\tilde u^{\tilde\beta}\sum_{j=1}^{m_1-1}\tilde\theta_j^{2\tilde\beta_j+1}\frac{\partial}{\partial t}\tilde u_j\nonumber\\
	&=\sum_{\beta_1=0}^{p-1}\frac{1}{\beta_1!}\theta_1^{(\beta_1+1)^2}u_1^{\beta_1}\frac{\partial}{\partial t}u_1\sum_{|\tilde\beta|=p-\beta_1-1}\begin{pmatrix}
	p\\ \tilde\beta\end{pmatrix}\tilde\theta^{\tilde\beta^2}\tilde u^{\tilde\beta}\nonumber\\
	&+\sum_{\beta_1=0}^{p-1}\frac{1}{\beta_1!}\theta_1^{\beta_1^2}u_1^{\beta_1}\frac{p!}{(p-\beta_1)!}\sum_{|\tilde\beta| = p-\beta_1-1}\begin{pmatrix} p-\beta_1\\ \tilde\beta \end{pmatrix} \tilde\theta^{\tilde\beta^2}\tilde u^{\tilde\beta}\sum_{j=1}^{m_1-1}\tilde\theta_j^{2\tilde\beta_j+1}\frac{\partial}{\partial t}\tilde u_j\nonumber\\
	&=\sum_{|\beta|=p-1}\begin{pmatrix}p\\\beta\end{pmatrix}\theta^{\beta^2}u^{\beta}\theta_1^{2\beta_1+1}\frac{\partial}{\partial t}u_1
	+\sum_{|\beta|=p-1}\begin{pmatrix}p\\\beta\end{pmatrix}\theta^{\beta^2}u^{\beta}\sum_{j=2}^{m_1}\theta_j^{2\beta_j+1}\frac{\partial}{\partial t}u_j\nonumber\\
	&=\sum_{|\beta|=p-1}\begin{pmatrix}p\\\beta\end{pmatrix}\theta^{\beta^2}u^{\beta}\sum_{j=1}^{m_1}\theta_j^{2\beta_j+1}\frac{\partial}{\partial t}u_j.
	\end{align*}
	Therefore, the result follows from induction.
\end{proof}

\begin{lemma}\label{Hp-lem8}
	Suppose $m_1\in \mathbb N$, $\theta= (\theta_1,\ldots, \theta_{m_1})$, where $\theta_1,...,\theta_{m_1}$ are positive real numbers. If $p\in\mathbb N$ such that $p\ge 2$, then
	$$\sum_{|\beta|=p-1}\begin{pmatrix}p\\ \beta\end{pmatrix}\theta^{\beta^2}\sum_{i=1}^{m_1}\theta_i^{2\beta_i+1}d_i\nabla u_i\cdot\nabla u^\beta =\sum_{|\beta|=p-2}\begin{pmatrix}p\\ \beta\end{pmatrix}\theta^{\beta^2}u^\beta\sum_{l=1}^n\sum_{i,j=1}^{m_1}a_{i,j}\frac{\partial u_i}{\partial x_l}\frac{\partial u_j}{\partial x_l},$$
	where $(a_{i,j})$ is the $m_1\times m_1$ symmetric matrix with entries
	\begin{align}
	a_{i,j}=\left\{
	\begin{matrix}
	\frac{d_i+d_j}{2}\theta_i^{2\beta_i+1}\theta_j^{2\beta_j+1},&\text{if }i\ne j\\
	d_i\theta_i^{4\beta_i+4},&\text{if }i=j
	\end{matrix}
	\right..
	\end{align}
\end{lemma}

\begin{proof}
	The result is easily verified when $m_1=1$, regardless of the choice of $p$, and for $p=2$, regardless of the choice of $m_1$. Suppose $p\ge 2$ and $m_1\ge 2$. We proceed by induction on the value $m_1$, and assume $k\in \mathbb N$ such that the result is is true for $m_1=k$. Suppose $m_1=k+1$ and (as in the proof of Lemma \ref{Hp-lem7}) denote
	$$\tilde\beta=(\beta_2,...,\beta_{m_1})\text{ and }\tilde u=(u_2,...,u_{m_1}).$$
	Then
	\begin{align}\label{Hp-eq5}
	\sum_{|\beta|=p-1}\begin{pmatrix}p\\ \beta\end{pmatrix}\theta^{\beta^2}\sum_{i=1}^{m_1}\theta_i^{2\beta_i+1}d_i\nabla u_i\cdot\nabla u^\beta &=\sum_{\beta_1=0}^{p-1}\sum_{|\tilde\beta|=p-\beta_1-1}\begin{pmatrix}p\\ \beta\end{pmatrix}\theta_1^{\beta_1^2}\tilde\theta^{\tilde\beta^2}\biggl[\theta_1^{2\beta_1+1}d_1\nabla u_1\cdot\nabla u^\beta\nonumber\\
	&+\sum_{i=1}^{m_1-1}\tilde\theta_i^{2\tilde\beta_i+1}d_{i+1}\nabla\tilde u_i\cdot\nabla\left(u_1^{\beta_1}\tilde u^{\tilde\beta}\right)\biggr].
	\end{align}
	Note that for $1\le \beta_1\le p-2$ and $|\tilde\beta|=p-\beta_1-1$
	\begin{align}\label{Hp-eq6}
	\nabla\left(u_1^{\beta_1}\tilde u^{\tilde\beta}\right)=\beta_1 u_1^{\beta_1-1}\tilde u^{\tilde\beta}\nabla u_1+\sum_{j=1,\tilde\beta_j\ne 0}^{m_1-1}\tilde\beta_j u_1^{\beta_1} \tilde u^{\tilde\beta-e_j}\nabla \tilde u_j.
	\end{align}
	Therefore, from (\ref{Hp-eq5}) and (\ref{Hp-eq6}), we have
	\begin{align}\label{Hp-eq7}
	\sum_{|\beta|=p-1}\begin{pmatrix}p\\ \beta\end{pmatrix}\theta^{\beta^2}\sum_{i=1}^{m_1}\theta_i^{2\beta_i+1}d_i\nabla u_i\cdot\nabla u^\beta=\text{I}+\text{II},
	\end{align}
	where
	\begin{align}\label{Hp-eq8}
	\text{I}=&\sum_{\beta_1=0}^{p-1}\sum_{|\tilde\beta|=p-\beta_1-1}\begin{pmatrix}p\\\beta\end{pmatrix}\theta_1^{\beta_1^2}\tilde\theta^{\tilde\beta^2}\biggl[\theta_1^{2\beta_1+1}d_1\nabla u_1\cdot \nabla u^{\beta}\nonumber\\
	&+\sum_{i=1,\beta_1\ne 0}^{m_1-1}\tilde\theta_i^{2\tilde\beta_i+1}d_{i+1}\nabla \tilde u_i\cdot\beta_1 u_1^{\beta_1-1}\tilde u^{\tilde\beta}\nabla u_1\biggr]
	\end{align}
	and
	\begin{align}\label{Hp-eq9}
	\text{II}=\sum_{\beta_1=0}^{p-2}\sum_{|\tilde\beta|=p-\beta_1-1}\begin{pmatrix}p\\\beta\end{pmatrix}\theta_1^{\beta_1^2}\tilde\theta^{\tilde\beta^2}\sum_{i,j=1,\tilde\beta_j\ne 0}^{m_1-1}\tilde\theta_i^{2\tilde\beta_i+1}d_{i+1}\nabla \tilde u_i\cdot\tilde\beta_ju_1^{\beta_1}\tilde u^{\tilde\beta-e_j}\nabla \tilde u_j.
	\end{align}
	Above, $e_j$ denotes row $j$ of the $(m_1-1)\times (m_1-1)$ identity matrix. We start with the analysis of II. We can rewrite
	\begin{align}\label{Hp-eq10}
	\text{II}&=\sum_{\beta_1=0}^{p-2}\frac{1}{\beta_1!}\theta_1^{\beta_1^2}u_1^{\beta_1}\sum_{|\tilde\beta|=p-\beta_1-1}\begin{pmatrix}p\\\tilde\beta\end{pmatrix}\tilde\theta^{\tilde\beta^2}\sum_{i,j=1,\tilde\beta_j\ne 0}^{m_1-1}\tilde\theta_i^{2\tilde\beta_i+1}d_{i+1}\nabla \tilde u_i\cdot\tilde\beta_j\tilde u^{\tilde\beta-e_j}\nabla \tilde u_j\nonumber\\
	&=\sum_{\beta_1=0}^{p-2}\frac{1}{\beta_1!}\theta_1^{\beta_1^2}u_1^{\beta_1}\sum_{|\tilde\beta|=p-\beta_1-1}\begin{pmatrix}p\\\tilde\beta\end{pmatrix}\tilde\theta^{\tilde\beta^2}\sum_{i,j=1,\tilde\beta_j\ne 0}^{m_1-1}\tilde\theta_i^{2\tilde\beta_i+1}d_{i+1}\nabla \tilde u_i\cdot \nabla\tilde u^{\tilde\beta}\nonumber\\
	&=\sum_{\beta_1=0}^{p-2}\frac{1}{\beta_1!}\theta_1^{\beta_1^2}u_1^{\beta_1}\frac{p!}{(p-\beta_1)!}\sum_{|\tilde\beta|=p-\beta_1-1}\begin{pmatrix}p-\beta_1\\\tilde\beta\end{pmatrix}\tilde\theta^{\tilde\beta^2}\sum_{i,j=1,\tilde\beta_j\ne 0}^{m_1-1}\tilde\theta_i^{2\tilde\beta_i+1}d_{i+1}\nabla \tilde u_i\cdot \nabla\tilde u^{\tilde\beta}\nonumber\\
	&=\sum_{\beta_1=0}^{p-2}\frac{1}{\beta_1!}\theta_1^{\beta_1^2}u_1^{\beta_1}\frac{p!}{(p-\beta_1)!}\sum_{|\tilde\beta|=p-\beta_1-2}\begin{pmatrix}p-\beta_1\\ \tilde\beta\end{pmatrix}\tilde\theta^{\tilde\beta^2}\tilde u^{\tilde\beta}\sum_{l=1}^n\sum_{i,j=1}^{m_1-1} a_{i+1,j+1}\frac{\partial \tilde u_i}{\partial x_l}\frac{\partial \tilde u_j}{\partial x_l}\nonumber\\
	&=\sum_{|\beta|=p-2}\begin{pmatrix}p\\\beta\end{pmatrix}\theta^{\beta^2}u^\beta\sum_{l=1}^n\sum_{i,j=2}^{m_1}a_{i,j}\frac{\partial u_i}{\partial x_l}\frac{\partial u_j}{\partial x_l}
	\end{align}
	where the last step follows from the induction hypothesis. Now let's investigate I. We begin by expanding the $\nabla u^\beta$ term to find
	\begin{align}\label{Hp-eq11}
	\text{I}=&\sum_{\beta_1=0}^{p-1}\sum_{|\tilde\beta|=p-\beta_1-1}\begin{pmatrix}p\\\beta\end{pmatrix}\theta_1^{\beta_1^2}\tilde\theta^{\tilde\beta^2}\biggl[\theta_1^{2\beta_1+1}d_1\nabla u_1\cdot \biggl(\beta_1u_1^{\beta_1-1}\tilde u^{\tilde\beta}\nabla u_1\nonumber\\
	&+\sum_{i=1}^{m_1-1}\tilde\beta_iu_1^{\beta_1}\tilde u^{\tilde\beta-e_i}\nabla \tilde u_i\biggr)+\sum_{i=1,\beta_1\ne 0}^{m_1-1}\tilde\theta_i^{2\tilde\beta_i+1}d_{i+1}\nabla \tilde u_i\cdot\beta_1 u_1^{\beta_1-1}\tilde u^{\tilde\beta}\nabla u_1\biggr]\nonumber\\
	=&\text{I}_{1,1}+\sum_{i=2}^{m_1}\text{I}_{i,1},
	\end{align}
	where 
	\begin{align}\label{Hp-eq12}
	\text{I}_{1,1}&=\sum_{\beta_1=1}^{p-1}\sum_{|\tilde\beta|=p-\beta_1-1}\begin{pmatrix}p\\\beta\end{pmatrix}\theta_1^{\beta_1^2}\tilde\theta^{\tilde\beta^2}\sum_{l=1}^n \theta_1^{2\beta_1+1}d_1\beta_1u_1^{\beta_1-1}\tilde u^{\tilde\beta}\left(\frac{\partial u_1}{\partial x_l}\right)^2\nonumber\\
	&=\sum_{\beta_1=1}^{p-1}\sum_{|\tilde\beta|=p-(\beta_1-1)-2}\begin{pmatrix}p\\ (\beta_1-1,\tilde\beta) \end{pmatrix}\theta^{(\beta_1-1,\tilde\beta)^2} u^{(\beta_1-1,\tilde\beta)} \sum_{l=1}^n \theta_1^{4(\beta_1-1)+4}d_1\left(\frac{\partial u_1}{\partial x_l}\right)^2\nonumber\\
	&=\sum_{\beta_1=0}^{p-2}\sum_{|\tilde\beta|=p-\beta_1-2}\begin{pmatrix}p\\ \beta \end{pmatrix}\theta^{\beta^2} u^{\beta} \sum_{l=1}^n \theta_1^{4\beta_1+4}d_1\left(\frac{\partial u_1}{\partial x_l}\right)^2\nonumber\\
	&=\sum_{|\beta|=p-2}\begin{pmatrix}p\\ \beta \end{pmatrix}\theta^{\beta^2} u^{\beta} \sum_{l=1}^n \theta_1^{4\beta_1+4}d_1\left(\frac{\partial u_1}{\partial x_l}\right)^2,
	\end{align}
	and for $i\in\{2,...,m_1\}$,
	\begin{align}\label{Hp-eq13}
	\text{I}_{i,1}&=\sum_{\beta_1=0}^{p-1}\sum_{|\tilde\beta|=p-\beta_1-1}\begin{pmatrix}p\\\beta\end{pmatrix}\theta^{\beta^2}\sum_{l=1}^n \left[ \theta_1^{2\beta_1+1}d_1\tilde\beta_{i-1}u_1^{\beta_1}\tilde u^{\tilde\beta-e_{i-1}}+\tilde\theta_{i-1}^{2\tilde\beta_{i-1}+1}d_i\beta_1u_1^{\beta_1-1}\tilde u^{\tilde\beta} \right]\frac{\partial u_1}{\partial x_l}\frac{\partial u_i}{\partial x_l}\nonumber\\
	&=\sum_{\beta_1=0}^{p-2}\sum_{|\tilde\beta|=p-\beta_1-1,\tilde\beta_{i-1}\ne 0}\begin{pmatrix}p\\(\beta_1,\tilde\beta-e_{i-1})\end{pmatrix}\theta^{\beta^2}\sum_{l=1}^n\theta_1^{2\beta_1+1}d_1u_1^{\beta_1}\tilde u^{\tilde\beta-e_{i-1}}\frac{\partial u_1}{\partial x_l}\frac{\partial u_i}{\partial x_l}\nonumber\\
	&+\sum_{\beta_1=1}^{p-1}\sum_{|\tilde\beta|=p-\beta_1-1}\begin{pmatrix}p\\(\beta_1-1,\tilde\beta)\end{pmatrix}\theta^{\beta^2}\sum_{l=1}^n\theta_i^{2\beta_i+1}d_iu_1^{\beta_1-1}\tilde u^{\tilde\beta}\frac{\partial u_1}{\partial x_l}\frac{\partial u_i}{\partial x_l}\nonumber\\
	&=\sum_{|\beta|=p-2}\begin{pmatrix}p\\\beta\end{pmatrix}\theta^{\beta^2}u^\beta\sum_{l=1}^n\theta_1^{2\beta_1+1}\theta_i^{2\beta_i+1}d_1\frac{\partial u_1}{\partial x_l}\frac{\partial u_i}{\partial x_l}\nonumber\\
	&+\sum_{|\beta|=p-2}\begin{pmatrix}p\\\beta\end{pmatrix}\theta^{\beta^2}u^\beta\sum_{l=1}^n\theta_1^{2\beta_1+1}\theta_i^{2\beta_i+1}d_i\frac{\partial u_1}{\partial x_l}\frac{\partial u_i}{\partial x_l}\nonumber\\
	&=\sum_{|\beta|=p-2}\begin{pmatrix}p\\\beta\end{pmatrix}\theta^{\beta^2}u^\beta\sum_{l=1}^n\theta_1^{2\beta_1+1}\theta_i^{2\beta_i+1}(d_1+d_i)\frac{\partial u_1}{\partial x_l}\frac{\partial u_i}{\partial x_l}.
	\end{align}
	The result follows by combining (\ref{Hp-eq7}), (\ref{Hp-eq8}), (\ref{Hp-eq10}), (\ref{Hp-eq11}), (\ref{Hp-eq12}) and (\ref{Hp-eq13}). 
\end{proof}

\medskip
\noindent{\bf Acknowledgement.} This work is partially supported by NAWI Graz, and the International Research Training Group IGDK 1754 ``Optimization and Numerical Analysis for Partial Differential Equations with Nonsmooth Structures'', funded by the German Research Council (DFG) project number 188264188/GRK1754 and the Austrian Science Fund (FWF) under grant number W 1244-N18.

\end{document}